\documentclass{amsart}
\usepackage{graphicx}
\usepackage{amsmath}
\usepackage[bbgreekl]{mathbbol}
\usepackage{amsfonts}
\usepackage{bbm}
\usepackage{amssymb}
\usepackage{mathtools}
\usepackage{mathrsfs}
\usepackage{verbatim}
\usepackage{color}
\usepackage{extpfeil}
\usepackage{bbding}
\usepackage{hyperref}

\usepackage{tikz}
\usepgfmodule{shapes}
\usetikzlibrary{calc,positioning, matrix}
\usetikzlibrary{arrows}

\newtheorem{thm}{Theorem}[section]
\newtheorem{cor}[thm]{Corollary}
\newtheorem{lem}[thm]{Lemma}
\newtheorem{prop}[thm]{Proposition}
\newtheorem{claim}[thm]{Claim}
\newtheorem{observation}[thm]{Observation}

\theoremstyle{definition}
\newtheorem{defn}[thm]{Definition}
\newtheorem{notation}[thm]{Notation}
\newtheorem{example}[thm]{Example}

\theoremstyle{remark}
\newtheorem{rem}[thm]{Remark}

\newtheorem{reminder}[thm]{Reminder}
\numberwithin{equation}{section}
\newcommand{\set}[1]{\left\{#1\right\}}
\newcommand{\eps}{\varepsilon}
\newcommand{\fii}{\varphi}
\newcommand{\xto}{\xrightarrow}
\newcommand{\xfrom}{\xleftarrow}
\newcommand{\longto}{\longrightarrow}
\newcommand{\Hom}{\mathrm{Hom}}

\newcommand{\N}{\mathbb N}

\newcommand{\K}{\mathbb K}
\newcommand{\Q}{\mathbb Q}

\newcommand{\Alt}{\mathrm{Alt}}
\newcommand{\sgn}{\mathrm{sgn}}
\newcommand{\Uni}{\mathcal{U}}  
\newcommand{\Unik}[1]{\Uni^{\leq #1}}
\newcommand{\Unih}{\Uni_\hbar}

\newcommand{\calI}{\mathcal I}

\newcommand{\calF}{\mathcal F}

\newcommand{\calS}{\mathcal S}
\newcommand{\calT}{\mathcal T}

\newcommand{\unitcat}{*}
\newcommand{\catC}{\mathscr{C}} 
\newcommand{\catD}{\mathscr{D}}
\newcommand{\catJ}{\mathscr{J}}
\newcommand{\catH}{\mathscr{G}}
\newcommand{\catA}{\mathscr{A}}
\newcommand{\catF}{\mathscr{F}}
\newcommand{\catP}{\mathscr{P}}
\newcommand{\catV}{\mathscr{V}}
\newcommand{\catTV}{{\catV^\mathrm{top}}}
\newcommand{\catHTV}{{\catV^\mathrm{haus}}}
\newcommand{\catCTV}{{\catV^\mathrm{cpl}}} 
\newcommand{\catTop}{\mathbf{Top}}

\newcommand{\catPh}{{\catP_\hbar}}
\newcommand{\catFh}{{\catF_\hbar}}

\newcommand{\catPhPhi}{{\catP_\hbar^\Phi}}
\newcommand{\catFhPhi}{{\catF_\hbar^\Phi}}

\newcommand{\catTVkh}{\catV^{\hbar\text{-}\mathrm{cpl}}_\kh}
\newcommand{\catTVkhp}{\catV^\mathrm{top}_\kh}
\newcommand{\catCTVkh}{\catV^\mathrm{cpl}_\kh}

\newcommand{\compl}{\mathrm{compl}}
\newcommand{\complh}{\mathrm{compl}_\hbar}
\newcommand{\hctp}{\tilde\otimes_\kh} 
\newcommand{\hh}{[[\hbar]]}

\newcommand{\tp}{\otimes} 
\newcommand{\ctp}{\mathbin{\hat\otimes}} 

\newcommand{\LM}[1]{#1 \text{-}\mathbf{Mod}} 
\newcommand{\RM}[1]{\mathbf{Mod}\text{-}#1} 
\newcommand{\RFree}[1]{\mathbf{Free}\text{-}#1} 

\newcommand{\Set}{\mathbf{Set}} 

\newcommand{\fun}[1]{\mathrm{#1}} 

\newcommand{\field}{\mathbb K}

\newcommand{\id}{\mathrm{id}}
\newcommand{\argument}{\text{\textvisiblespace}}
\newcommand{\kh}{{\field[[\hbar]]}}

\newcommand{\srdiecko}{\mathop{\heartsuit}}
\newcommand{\budzogan}{\mathop{\clubsuit}}
\newcommand{\srdco}{\mathop{\widetilde \srdiecko}}
\newcommand{\srd}{\mathop{\widetilde \srdco}}

\newcommand{\act}{\triangleright}
\newcommand{\coact}{\triangleleft}

\newcommand{\todo}[1]{{\Large #1}}
\newcommand{\forme}[1]{}   

\newcommand{\odpad}[1]{}

\newcommand{\End}{\operatorname{End}}
\newcommand{\Ima}{\operatorname{Im}}
\newcommand{\Ker}{\operatorname{Ker}}

\newcommand{\ob}{\mathrm{ob}}

\newcommand{\cesta}{./obrazky}

\newcommand{\Yoneda}{\mathbf{h}}
\newcommand{\UE}{\mathcal U}    
\newcommand{\Ud}{{\mathcal U \mathfrak d}}
\newcommand{\Ug}{{\mathcal U \mathfrak g}}
\newcommand{\Up}{{\mathcal U \mathfrak p}}

\newcommand{\g}{ \mathfrak g}
\newcommand{\p}{ \mathfrak p}

\newcommand{\FHP}{\fun{Forg}_\catH^\catP}

\newcommand{\FCTVP}{\fun F_\catCTV^{\catP}}

\newcommand{\half}{\frac{1}{2}}

\newcommand{\Ugh}{\Ug[[\hbar]]}
\newcommand{\Uph}{\Up[[\hbar]]}

\newcommand{\EndY}{\End(\catFhPhi \xto{\YPhi} \catCTVkh)}
\newcommand{\YA}{\catPhPhi\xto{\ctp A} \catFhPhi \xto{\YPhi} \catCTVkh}
\newcommand{\EndYA}{\End(\YA)}

\newcommand{\bg}{(\g,\delta,\fii)}
\newcommand{\bgprim}{(\g,\delta',\fii')}

\newcommand{\bp}{(\p,\delta_\p,\fii_\p)}
\newcommand{\bpprim}{(\p,\delta'_\p,\fii'_\p)}

\newcommand{\Forg}{\catPhPhi \xto\forg \catCTVkh}
\newcommand{\EndForg}{\End(\Forg)}

\newcommand{\srdieckof}{\srdiecko}
\newcommand{\srdcof}{\srdco}

\newcommand{\calG}{\mathcal{G}}

\newcommand{\cSg}{{\hat S\g}}

\newcommand{\YQp}{\Yoneda_{(Q\ctp A, p)}}

\newcommand{\YPhi}{\Yoneda^\Phi}

\newcommand{\os}{\subset_{\mathrm{os}}}  
\newcommand{\bs}{\subset_{\mathrm{bs}}}  

\newcommand{\limita}{\varprojlim}
\newcommand{\colim}{\operatornamewithlimits{\underset{\longrightarrow}{colim}}}

\newcommand{\hhom}{\catTV}

\newcommand{\Split}{\mathrm{Split}}

\newcommand{\Proj}{\Split(\catFh)}
\newcommand{\ProjPhi}{\Split(\catFhPhi)}

\newcommand{\forg}{\mathrm{forg}}

\newcommand{ \cFh}{\circ_\catFh}
\newcommand{ \tpFh}{\tp_\catFh}
\newcommand{ \cFhPhi}{\circ_\catFhPhi}
\newcommand{ \tpFhPhi}{\tp_\catFhPhi}

\newcommand{\Omegap}{t}

\newcommand{\footnoteremember}[2]{
\footnote{#2}
\newcounter{#1}
\setcounter{#1}{\value{footnote}}
}
\newcommand{\footnoterecall}[1]{
\footnotemark[\value{#1}]
}

\begin{document}

\title{On Quantization of quasi-Lie bialgebras}
\author{\v Stefan Sak\'alo\v s}
\thanks{\v Stefan Sak\'alo\v s was supported by the ProDoc program ``Geometry, Algebra and Mathematical Physics'' and the grant PDFMP2\_137071 of the Swiss National Science Foundation.}
\address{Department of Mathematics, Universit\'{e} de Gen\`{e}ve, Geneva, Switzerland; Dept. of Theoretical Physics, FMFI UK, Bratislava, Slovakia}
\email{pista.sakalos@gmail.com}

\author{Pavol \v Severa}
\address{Department of Mathematics, Universit\'{e} de Gen\`{e}ve, Geneva, Switzerland, on leave
from Dept. of Theoretical Physics, FMFI UK, Bratislava, Slovakia}
\email{pavol.severa@gmail.com}
\thanks{Pavol \v Severa was partially supported by the Swiss National Science Foundation (grants 140985 and 141329)}

\begin{abstract}
We modify the quantization of Etingof and Kazhdan so that it can be used to quantize quasi-Lie bialgebras.
\end{abstract}

\maketitle

\setcounter{tocdepth}{1}
\tableofcontents

\newcommand{\U}{\mathcal{U}}
\newcommand{\copr}{\blacktriangle}

\section{Introduction}

Quasi-bialgebras were introduced by Drinfeld in \cite{Drinfeld1} as a generalization of bial-
gebras in such a way that the corresponding category of modules would still be a
monoidal category. An additional structure of a quasi-triangular quasi-bialgebra
produces a braided monoidal category in a similar way. Drinfeld also introduced
quasi-Hopf universal enveloping algebras as deformations of  universal enveloping
algebras of Lie algebras in the category of quasi-bialgebras, and quasi-Lie bialgebras
as their classical limits. In the subsequent paper \cite{Drinfeld2} he introduced associators as a
way to quantize a Lie bialgebra $\g$ together with an invariant element $t\in S^2\g$ into
a quasi-triangular quasi-bialgebra.

Etingof and Kazhdan \cite{EK1} used Drinfeld's results to quantize an arbitrary Lie
bialgebra into a quantum enveloping algebra. Their method is explicit (depending
on a choice of a Drinfeld associator).

The natural question of quantization of quasi-Lie bialgebras was addressed by
Enriquez and Halbout in \cite{EnriquezHalbout}. Their methods are quite different ---  they compare
the cohomologies of the prop of quasi-Lie bialgebras with those of the prop of Lie-
bialgebras in order to reduce their problem to the one already solved by Etingof
and Kazhdan.

We propose a more direct approach to the same problem based on an adjustment
of the methods of Etingof and Kazhdan. We describe our construction in detail in
Section \ref{nutshell} and compare it with the (slightly reformulated) quantization of Etingof
and Kazhdan. The main difference is that we need to replace a certain free mod-
ule with a projective module and use the fact that projections can be explicitly
deformed.

\section{Our construction in a nutshell}\label{nutshell}
In this section we describe our quantization of quasi-Lie bialgebras almost completely. We omit the discussion of topologies and all the continuity requirements; we leave them, together with basic definitions and all the proofs, to the rest of the paper. This section is, however, sufficient for obtaining explicit formulas (depending on a choice of a Drinfeld associator $\Phi$). We shall also compare our method with Etingof-Kazhdan quantization of Lie bialgebras.

Let $\g$ be a quasi-Lie bialgebra. Our task is to define explicitly a quasi-bialgebra $\U_\hbar\g$ deforming the bialgebra $\U\g$.
It is done in two steps:
\begin{enumerate}
\item We construct a monoidal category $\catA^\Phi$ which is a deformation of the category of $\U\g$-modules ($\catA^\Phi$ is morally the category of $\U_\hbar\g$-modules)
\item We construct an object $C$ of $\catA^\Phi$ (morally, $C$ is $\U_\hbar\g$ as its own module) together with morphisms $C\to C\otimes C$ and $C\to1$, making $\Hom(C,\cdot)$ to a quasi-monoidal functor. These morphisms make the algebra $\Hom(C,C)$ to a quasi-bialgebra. We shall find an explicit isomorphism of vector spaces $\Hom(C,C)\cong\U\g$ and set $$\U_\hbar\g=\Hom(C,C).$$
\end{enumerate}

Let us now describe the two steps in some detail.

\subsection*{Step 1}
 Let $(\p,\g)$ be the Manin pair corresponding to the quasi-Lie bialgebra $\g$. We shall encode the category of $\U\g$-modules into the category $\catP$ of $\U\p$-modules as follows. Let
$$A=\Hom_{\U\g}(\U\p,\K),$$
where $\K$ is the base field
($A$ is roughly the algebra of functions on the homogeneous space $P/G$). $A$ is a commutative algebra object in the category $\catP$ of $\U\p$-modules. The category of $\U\g$-modules is equivalent to the category $\catA$ of $A$-modules in $\catP$ via the co-induction functor
$$\heartsuit:M\mapsto\Hom_{\U\g}(\U\p,M).$$

If $X$ is a $\U\p$-module and $X{\downarrow}^{\U\p}_{\U\g}$ is $X$ seen as a $\U\g$-module, then
$$\heartsuit(X{\downarrow}^{\U\p}_{\U\g})\cong X\otimes A$$
is a free $A$-module. For any $\U\p$-modules $X$ and $Y$ we have therefore bijections
\begin{equation}\label{eq:skol_free}
\Hom_{\U\g}(X,Y)\cong\Hom_{\mathcal{A}}(X\otimes A,Y\otimes A)\cong\Hom_{\U\p}(X,Y\otimes A).
\end{equation}
The category $\catA_\text{Free}$ of \emph{free} $A$-modules in $\catP$ is equivalent to the category whose objects are $\U\p$-modules and morphisms are $\U\g$-equivariant maps.

Any Drinfeld associator $\Phi$ can be used to change  the associativity constraint and the braiding in $\catP$. We shall denote the resulting braided monoidal category by $\catP^\Phi$. Let us recall that as categories, 
$$\catP^\Phi=\catP;$$
 only the braided monoidal structure is changed. The algebra $A$ (with its original product $A\otimes A\to A$) remains a commutative associative algebra object in $\catP^\Phi$. Let $\catA^\Phi$ be the category of $A$-modules in $\catP^\Phi$. 

Since $A$ is a commutative algebra in $\catP^\Phi$, the category $\catA^\Phi$ is monoidal. Let us describe explicitly its subcategory $\catA^\Phi_\text{Free}$ of \emph{free} $A$-modules, i.e.\ of $A$-modules in $\catP^\Phi$ of the form $X\otimes A$, where $X$ is a $\U\p$-module. 
As in \eqref{eq:skol_free} we have  bijections
\begin{equation}\label{eq:skol_freePhi}
\Hom_{\catA^\Phi}(X\otimes A,Y\otimes A)\cong\Hom_{\U\p}(X,Y\otimes A)\cong\Hom_{\U\g}(X,Y).
\end{equation}

The categories $\catA_\text{Free}$ and $\catA^\Phi_\text{Free}$ have \emph{the same objects}.
 Using the identifications \eqref{eq:skol_free} and \eqref{eq:skol_freePhi} we can say that they also have \emph{the same morphisms}. The composition of morphisms is, however, different:
The composition in $\catA^\Phi_\text{Free}$
$$\Hom_{\U\p}(X,Y\otimes A)\times\Hom_{\U\p}(Y,Z\otimes A)\to\Hom_{\U\p}(X,Z\otimes A)$$
is given by the diagram in $\catP^\Phi$
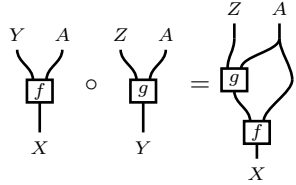
$$
\begin{tikzpicture}[baseline=0.0cm]
\draw [line width=1pt] (0.17,0.15) rectangle (-0.17,-0.15);
\node at (0.0,0.0) {{$\scriptstyle f$}};
\node at (0.0,-0.75) {{\scriptsize $X$}};
\draw [line width=1pt] (0.0,-0.55) ..controls +(0.0,0.17) and +(0.0,-0.17) .. (0.0,-0.15);
\node at (-0.3,0.75) {{\scriptsize $Y$}};
\draw [line width=1pt] (-0.09,0.15) ..controls +(0.0,0.17) and +(0.0,-0.17) .. (-0.3,0.55);
\node at (0.3,0.75) {{\scriptsize $A$}};
\draw [line width=1pt] (0.09,0.15) ..controls +(0.0,0.17) and +(0.0,-0.17) .. (0.3,0.55);
\end{tikzpicture}
\circ
\begin{tikzpicture}[baseline=0.0cm]
\draw [line width=1pt] (0.17,0.15) rectangle (-0.17,-0.15);
\node at (0.0,0.0) {{$\scriptstyle g$}};
\node at (0.0,-0.75) {{\scriptsize $Y$}};
\draw [line width=1pt] (0.0,-0.55) ..controls +(0.0,0.17) and +(0.0,-0.17) .. (0.0,-0.15);
\node at (-0.3,0.75) {{\scriptsize $Z$}};
\draw [line width=1pt] (-0.09,0.15) ..controls +(0.0,0.17) and +(0.0,-0.17) .. (-0.3,0.55);
\node at (0.3,0.75) {{\scriptsize $A$}};
\draw [line width=1pt] (0.09,0.15) ..controls +(0.0,0.17) and +(0.0,-0.17) .. (0.3,0.55);
\end{tikzpicture}
=
\begin{tikzpicture}[baseline=0.0cm]
\draw [line width=1pt] (0.17,-0.4) rectangle (-0.17,-0.7);
\node at (0.0,-0.55) {{$\scriptstyle f$}};
\draw [line width=1pt] (-0.13,0.3) rectangle (-0.47,0.0);
\node at (-0.3,0.15) {{$\scriptstyle g$}};
\draw [line width=1pt] (-0.09,-0.4) ..controls +(0.0,0.17) and +(0.0,-0.17) .. (-0.3,0.0);
\draw [line width=1pt] (0.09,-0.4) ..controls +(0.0,0.17) and +(0.0,-0.17) .. (0.47,0.15);
\draw [line width=1pt] (-0.39,0.3) ..controls +(0.0,0.17) and +(0.0,-0.17) .. (-0.3,0.7);
\draw [line width=1pt] (-0.21,0.3) ..controls +(0.0,0.17) and +(0.0,-0.17) .. (0.3,0.7);
\draw [line width=1pt] (0.47,0.15) ..controls +(0.0,0.17) and +(0.0,-0.17) .. (0.3,0.7);
\node at (0.0,-1.1) {{\scriptsize $X$}};
\draw [line width=1pt] (0.0,-0.9) ..controls +(0.0,0.0) and +(0.0,0.0) .. (0.0,-0.7);
\node at (-0.3,1.1) {{\scriptsize $Z$}};
\draw [line width=1pt] (-0.3,0.7) ..controls +(0.0,0.0) and +(0.0,0.0) .. (-0.3,0.9);
\node at (0.3,1.1) {{\scriptsize $A$}};
\draw [line width=1pt] (0.3,0.7) ..controls +(0.0,0.0) and +(0.0,0.0) .. (0.3,0.9);
\end{tikzpicture}
$$
Notice that this composition depends on the Drinfeld associator $\Phi$, which appears in the re-bracketing $(Z\otimes A)\otimes A\to Z\otimes(A\otimes A)$. The composition in $\catA_\text{Free}$ is given by the same diagram, but seen in $\catP$.

 The tensor product of objects in $\catA^\Phi_\text{Free}$ is the same as in $\catA_\text{Free}$:
$$(X\otimes A)\otimes_{\catA^\Phi}(Y\otimes A)=(X\otimes Y)\otimes A.$$
The tensor  product of morphisms in $\catA^\Phi_\text{Free}$
$$\Hom_{\U\p}(X,Y\otimes A)\times\Hom_{\U\p}(Z,W\otimes A)\to\Hom_{\U\p}(X\otimes Z,(Y\otimes W)\otimes A)$$
is given by the diagram in $\catP^\Phi$
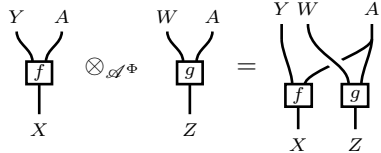
$$
\begin{tikzpicture}[baseline=0.0cm]
\draw [line width=1pt] (0.17,0.15) rectangle (-0.17,-0.15);
\node at (0.0,0.0) {{$\scriptstyle f$}};
\node at (0.0,-0.75) {{\scriptsize $X$}};
\draw [line width=1pt] (0.0,-0.55) ..controls +(0.0,0.17) and +(0.0,-0.17) .. (0.0,-0.15);
\node at (-0.3,0.75) {{\scriptsize $Y$}};
\draw [line width=1pt] (-0.09,0.15) ..controls +(0.0,0.17) and +(0.0,-0.17) .. (-0.3,0.55);
\node at (0.3,0.75) {{\scriptsize $A$}};
\draw [line width=1pt] (0.09,0.15) ..controls +(0.0,0.17) and +(0.0,-0.17) .. (0.3,0.55);
\end{tikzpicture}
\otimes_{\catA^\Phi}
\begin{tikzpicture}[baseline=0.0cm]
\draw [line width=1pt] (0.17,0.15) rectangle (-0.17,-0.15);
\node at (0.0,0.0) {{$\scriptstyle g$}};
\node at (0.0,-0.75) {{\scriptsize $Z$}};
\draw [line width=1pt] (0.0,-0.55) ..controls +(0.0,0.17) and +(0.0,-0.17) .. (0.0,-0.15);
\node at (-0.3,0.75) {{\scriptsize $W$}};
\draw [line width=1pt] (-0.09,0.15) ..controls +(0.0,0.17) and +(0.0,-0.17) .. (-0.3,0.55);
\node at (0.3,0.75) {{\scriptsize $A$}};
\draw [line width=1pt] (0.09,0.15) ..controls +(0.0,0.17) and +(0.0,-0.17) .. (0.3,0.55);
\end{tikzpicture}
=
\begin{tikzpicture}[baseline=0.0cm]
\draw [line width=1pt] (-0.2,-0.15) rectangle (-0.55,-0.45);
\node at (-0.38,-0.3) {{$\scriptstyle f$}};

\node at (0.38,-0.3) {{$\scriptstyle g$}};
\draw [line width=1pt] (-0.46,-0.15) ..controls +(0.0,0.17) and +(0.0,-0.17) .. (-0.6,0.45);
\draw [line width=1pt] (0.46,-0.15) ..controls +(0.0,0.17) and +(0.0,-0.17) .. (0.6,0.45);
\draw [line width=1pt] (-0.29,-0.15) ..controls +(0.0,0.17) and +(0.0,-0.17) .. (0.6,0.45);
\draw [white,line width=5pt] (0.29,-0.15) ..controls +(0.0,0.3) and +(0.0,-0.4) .. (-0.25,0.65);
\draw [line width=1pt] (0.29,-0.15) ..controls +(0.0,0.3) and +(0.0,-0.4) .. (-0.25,0.65);%
\draw [line width=1pt] (0.55,-0.15) rectangle (0.2,-0.45);
\node at (-0.38,-0.95) {{\scriptsize $X$}};
\draw [line width=1pt] (-0.38,-0.75) ..controls +(0.0,0.0) and +(0.0,0.0) .. (-0.38,-0.45);
\node at (0.38,-0.95) {{\scriptsize $Z$}};
\draw [line width=1pt] (0.38,-0.75) ..controls +(0.0,0.0) and +(0.0,0.0) .. (0.38,-0.45);
\node at (-0.6,0.85) {{\scriptsize $Y$}};
\draw [line width=1pt] (-0.6,0.45) ..controls +(0.0,0.0) and +(0.0,0.0) .. (-0.6,0.65);
\node at (-0.25,0.85) {{\scriptsize $W$}};
\node at (0.6,0.85) {{\scriptsize $A$}};
\draw [line width=1pt] (0.6,0.45) ..controls +(0.0,0.0) and +(0.0,0.0) .. (0.6,0.65);
\end{tikzpicture}
$$
and again depends on $\Phi$.

The object $C\in\catA^\Phi$ we define below is, unfortunately, not a free $A$-module in $\catP^\Phi$. It is, however, a projective module, i.e.\ a direct summand in a free module. Let us describe explicitly the monoidal categories $\catA_\text{Proj}$ and $\catA^\Phi_\text{Proj}$ of projective $A$-modules in $\catP$ and in $\catP^\Phi$.

The category $\catA_\text{Proj}$ is equivalent (via the functor $\heartsuit$) to the category of $\U\g$-modules which are direct summands in $\U\p$-modules, or equivalently, to the category with objects
\begin{equation}\label{eq:skol_karoubi}
(X,p),\ X\in\catP,\ p\in\Hom_{\U\g}(X,X)\text{ such that }p\circ p=p
\end{equation}
and with morphisms
$$f:(X,p)\to(Y,q),\ f\in\Hom_{\U\g}(X,Y)\text{ such that }f=q\circ f\circ p.$$

If $p:X\to X$ is as in \eqref{eq:skol_karoubi} then certainly
$$\heartsuit(p)\circ_{\catA}\heartsuit(p)=\heartsuit(p),$$
however, in general,
$$\heartsuit(p)\circ_{\catA^\Phi}\heartsuit(p)\neq\heartsuit(p).$$
Fortunately, we can deform $\heartsuit(p)$ to get a projection in $\catA^\Phi$: Let $a=2p-1$ (so that $a\circ a=1$),  let $\tilde a=\heartsuit(a)$ \emph{understood as a morphism in} $\catA^\Phi_\text{Free}$, and let
$$p^\Phi=(a^\Phi+1)/2,$$
where 
$$a^\Phi=\tilde a({\tilde a}^2)^{-1/2}.$$
Then $p^\Phi:X\otimes A\to X\otimes A$ is an idempotent in $\catA^\Phi_\text{Free}$ (as $(a^\Phi)^2=1$).

Using the correspondence
\begin{subequations}\label{eq:skol_corresp}
\begin{equation}
(X,p)\mapsto(X\otimes A,p^\Phi)
\end{equation}
we identify the objects in $\catA_\text{Proj}$ with the objects in $\catA^\Phi_\text{Proj}$.\footnote{More precisely, we should define $\catA_\text{Proj}^\Phi$ as the category, where objects are pairs $(X,\pi)$, where $X$ is a $\U\p$-module and $\pi:X\otimes A\to X\otimes A$ a projection in $\catA^\Phi$; the image of $\pi$ is our projective $A$-module in $\catP^\Phi$.}
We also identify morphisms via
\begin{equation}
(f:X\to Y,\ f=q\circ f\circ p)\mapsto (f^\Phi:=q^\Phi\circ\heartsuit(f)\circ p^\Phi:X\otimes A\to Y\otimes A).
\end{equation}
\end{subequations}
This makes the category $\catA^\Phi_\text{Proj}$ (together with its monoidal structure) fully explicit.
\subsection*{Step 2}
Suppose that $C$ is an object in a linear monoidal category, for example in $\catA^\Phi$, and that we have have a strong quasi-monoidal structure on the functor
\begin{equation}\label{eq:skol_funct}
\Hom(C,\cdot).
\end{equation}
This means that we have morphisms $\epsilon: C\to 1$ and $\copr:C\to C\otimes C$, satisfying the following conditions:
For any objects $S,T$, the composition
$$\delta_{S,T}:\Hom(C,S)\otimes\Hom(C,T)\xrightarrow{\otimes}\Hom(C\otimes C,S\otimes T)\xrightarrow{\circ\copr}\Hom(C,S\otimes T)$$
is a bijection, 
\begin{equation}\label{eq:skol_coprcou}
(\id\otimes\epsilon)\circ\copr=(\epsilon\otimes\id)\circ\copr=\id,
\end{equation}
and 
$$\Hom(C,1)=\mathbb K\epsilon.$$

In this case, the algebra
$$H=\Hom(C,C)$$
has a unique quasi-bialgebra structure, such that the functor \eqref{eq:skol_funct}, understood as a functor to the category of $H$-modules, is strongly monoidal, with the monoidal structure given by the morphisms $\delta_{S,T}$. In particular, if $B$ is a quasi-bialgebra, we take the (monoidal) category of $B$-modules, and set $C=B$, $\copr=\Delta_B$ and $\epsilon=\varepsilon_B$, then $H=B$.

 Explicitly, the coproduct 
$$\Delta:H\to H\otimes H$$
 is given by the composition
$$H=\Hom(C,C)\xrightarrow{\circ\copr}\Hom(C,C\otimes C)\xrightarrow{\delta_{C,C}^{-1}}\Hom(C,C)\otimes\Hom(C,C)= H\otimes H.$$
The associator $\Phi_H\in H\otimes H\otimes H$ is the image of $1\otimes1\otimes1$ under
\begin{multline*}
H\otimes H\otimes H \xrightarrow{\delta_{C,C}\otimes\text{id}}\\
\Hom(C,C\otimes C)\otimes H\xrightarrow{\delta_{C\otimes C,C}}\Hom\bigl(C,(C\otimes C)\otimes C\bigr)
\xrightarrow{\gamma_{C,C,C}\circ}\\
\Hom\bigl(C,C\otimes (C\otimes C)\bigr)
\xrightarrow{\delta_{C,C\otimes C}^{-1}}  H\otimes\Hom(C,C\otimes C)\\
\xrightarrow{\text{id}\otimes\delta_{C,C}^{-1}}H\otimes H\otimes H,
\end{multline*}
where $\gamma_{C,C,C}:(C\otimes C)\otimes C\to C\otimes (C\otimes C)$ is the associativity constraint. Finally, the counit
$$\varepsilon:H\to\mathbb K$$
is given by
$$f\circ\epsilon=\varepsilon(f)\epsilon.$$

Our task is to construct $(C,\copr,\epsilon)$ in the monoidal category $\catA^\Phi$, together with an isomorphism
\begin{equation}\label{eq:skol_homcc}
\Hom(C,C)\cong\U\g
\end{equation}
of vector spaces. Notice that if we use the category $\catA$ instead, which is equivalent to the category of $\U\g$-modules via the functor $\heartsuit$, then 
$$C=\heartsuit(\U\g),\ \copr=\heartsuit(\Delta_{\U\g}),\ \epsilon=\heartsuit(\varepsilon_{\U\g})$$
 gives $H=\U\g$ as a bialgebra. We need to find a suitable replacement for $\heartsuit(\U\g)$ in $\catA^\Phi$. A natural approach is to make  $\U\g$  to a $\U\p$-module, or at least to a direct summand of a $\U\p$-module, since then $C=\heartsuit(\U\g)$ can be seen as an object of $\catA^\Phi_\text{Free}$ or of $\catA^\Phi_\text{Proj}$ and \eqref{eq:skol_homcc} holds.

Let us first construct $(C,\copr,\epsilon)$ in the case when $\g$ is a Lie bialgebra. In that case $H=\U_\hbar\g$ will be the bialgebra constructed by Etingof and Kazhdan. The bijection
$$\U\g\cong\U\p/(\U\p)\g^*$$
(given by the inclusion $\U\g\subset\U\p$) makes $\U\g$ to a $\U\p$-module, and we define $C$ to be the free $A$-algebra in $\catP^\Phi$
$$C=\U\g\otimes A=\heartsuit(\U\g).$$
We also set
$$\copr=\heartsuit(\Delta_{\U\g}),\ \epsilon=\heartsuit(\varepsilon_{\U\g}),$$
where $\Delta_{\U\g}$ and $\varepsilon_{\U\g}$ are the coproduct and the counit of the (cocommutative) bialgebra $\U\g$. 

By \eqref{eq:skol_freePhi} we have a bijection
$$H=\Hom_{\catA^\Phi}(C,C)\cong\Hom_{\U\g}(\U\g,\U\g)=\U\g$$
so we indeed defined a new quasi-bialgebra (actually bialgebra, due to co-associativity of $\copr$) structure on $\U\g$. This is a slight reformulation of the quantization of Etingof and Kazhdan.

Let us now suppose again that $\g$ is a quasi-Lie bialgebra. The basic problem is that in this case we can't extend the $\U\g$-action (by multiplication) on $\U\g$ to a $\U\p$-action, hence we can't take for $C$ a free $A$-module. We shall, however, be able to find a $C$ which is a \emph{projective} $A$-module. Namely, let
$$Q=\heartsuit(\U\g)=\Hom_{\U\g}(\U\p,\U\g),$$
understood as a $\U\p$-module.
We have a surjection $\pi:Q=\Hom_{\U\g}(\U\p,\U\g)\to\U\g$ given by $f\mapsto f(1)$.
If we choose a $f_0\in Q$ such that $f_0(1)=1$, then
$$\iota:\U\g\to Q,\ z\mapsto z\cdot f_0$$
is a (one-side) inverse of $\pi$. In this way we can see $\U\g$ as a direct summand in the $\U\p$-module $Q$, i.e.\ we can replace it by the pair 
$$(Q,p=\iota\circ\pi).$$

A convenient $f_0$ can be constructed as follows. As a vector space,
$$\p=\g\oplus\g^*,$$
which gives an isomorphism of vector spaces
$$\U\g\otimes S\g^*\to\U\p,$$
namely $z\otimes x\mapsto z\,\sigma(x)$, where $\sigma: S\g^*\to\U\p$ is the symmetrization map. The projection $f_0:\U\p\cong\U\g\otimes S\g^*\to\U\g$ is defined by sending $S^{>0}\g^*$ to $0$.

We thus set (see \eqref{eq:skol_corresp})
$$C=(Q\otimes A,p^\Phi),\ \copr=p^\Phi\otimes p^\Phi\circ{\Delta_{\U\g}}\circ p^\Phi,\ \epsilon={\varepsilon_{\U\g}}^\Phi.$$
Unfortunately, Equation \eqref{eq:skol_coprcou} is not satisfied, which means that $\Hom_{\catA^\Phi}(C,C)$ is a quasi-bialgebra with a weak counit. This defect can be easily repaired by a twist. Namely, let
$$ \copr'=(r^{-1}\otimes s^{-1})\circ\copr$$
where
$$r=(\epsilon\otimes 1)\circ\copr, s=(1\otimes\epsilon)\circ\copr.$$
Then the triple
$$(C,\copr',\epsilon)$$
satisfies all the requirements and we set
\begin{equation}\label{eq:skol_uhg}
\U_\hbar\g=\Hom_{\catA^\Phi}(C,C).
\end{equation}

\subsection*{Remarks}
\begin{enumerate}
\item If $\g$ is finite-dimensional then we can replace the coinduction $\heartsuit$ with the induction
$$M\mapsto \U\p\otimes_{\U\g}M,$$
the algebra $A$ with the coalgebra
$$\U\p\otimes_{\U\g}\mathbb K=\U\p/(\U\p)\g$$
and the $\U\p$-module $Q$ with $\U\p$. In this case we don't need to use any topologies. 

\item While it is clear that $\U_\hbar\g$ is a deformation of $\U\g$, we should verify that its first order part is given by the quasi-Lie bialgebra structure on $\g$. To do it we repeat the idea of Etingof and Kazhdan.  The functor
$$F:\catP^\Phi\to\mathcal V$$
(where $\mathcal V$ is the category of vector spaces), given by
$$F(X)=\Hom_{\catA^\Phi}(C,X\otimes A),$$
is  naturally isomorphic to the forgetful functor. Both $F$ and the forgetful functor are quasi-monoidal, but with different quasi-monoidal structures. In this way we get two quasi-triangular quasi-bialgebra structures on the representing object $\U\p$, differing by a twist. One of them is the standard quasi-bialgebra structure coming from $\Phi$, while the other is such that $\U\p$ contains $\U_\hbar\g$ as a sub-quasi-bialgebra. We compute the classical part of the twist and show that we get the classical twist of $\p$ which makes $\g$ to a sub-quasi-Lie bialgebra.
\end{enumerate}

\section{Notation}\label{sect_notations}
Some general notations:
\begin{itemize}
\item $\act, \coact$ denote left and right actions. E.g. if $\g$ is a Lie algebra and $M$ its left module, $x\in \g$, $m\in M$ then $x\act m$ denotes $x$ applied to $m$.
\item $\os$ and $\bs$ mean ``is an open vector subspace of'' and ``is a bounded vector subspace of'' (Definition \ref{defn_bs}).
\item $\ctp$ is the completed tensor product (Appendix \ref{subsection_tensor_products}). It turns $\catCTV$ into a symmetric monoidal category.
\item $\Hom^C_\Ug$, $\Hom^C_\field$ denote spaces of continuous $\Ug$-linear or $\field$-linear maps equipped with the strong topology.
\end{itemize}
The quasi-Lie bialgebras  and some elements that we use:
\begin{itemize}
\item $\bg$ the quasi-Lie bialgebra over $\field$ we want to quantize.
\item $e_i$ a basis of $\g$ and $e^i$ the dual basis
\item $\bp$ the ``double'' quasi-Lie bialgebra of $\g$ defined in Section \ref{subsection_Double_quasi-Lie bialgebra}.
\item $\Omegap$ the invariant element of $S^2\p$. Section \ref{subsection_Double_quasi-Lie bialgebra}.
\end{itemize}
The categories:
\begin{itemize}
\item $\catC$ a general category. Hom-sets are denoted by $\Hom_\catC(X,Y)$ or by $\catC(X,Y)$.
\item $\Split(\catC)$ is the Karoubi envelope of a category $\catC$. Section \ref{subsection_Karoubi}. 
\item $\catV,\catTV, \catCTV, \catHTV$ --- vector spaces, topological vector spaces, complete topological vector spaces and Hausdorff topological vector spaces. Appendix \ref{section_Top_vect_spaces}.
\item $\catTVkh$ and $ \catCTVkh$ --- category of topological $\kh$-modules that are $\hbar$-complete and the category of complete $\kh$-modules. Appendix \ref{section_kh_modules}.
\item $\catH$, $\catP$--- equicontinuous left $\g$-modules and equicontinuous left $\p$-modules (Section \ref{subsection_P_H}).
\item $\catA,\catF$ --- right $A$-modules and free right $A$-modules in $\catP$ (Section \ref{subsection_A_F}). There is an obvious fully-faithful functor $\catF\to\catA$.
\item $\catPh$ --- category with the same objects as $\catP$ but with hom-sets $\catPh(M,N):= \catP(M,N)\hh$. See Section \ref{section_passing_to_kh}.
\item $\catFh$ --- the same construction applied to $\catF$. Equivalently, the category of free right $A$-modules in $\catPh$.
\item $\catPhPhi$ --- the same category as $\catPh$ but with the braiding and associativity isomorphism ``quantized'' using a Drinfeld associator $\Phi$ as described in Section \ref{subs_Drinfeld_assoc}.
\item $\catFhPhi$ --- right free $A$-modules in $\catPhPhi$ (Section \ref{subsection_catFhPhi}).
\end{itemize}
The functors:
\begin{itemize}
\item $\xto\forg$ denotes  forgetful functors. Possibly non-obvious $\catPh\xto\forg \catCTVkh$ is defined in Section \ref{section_passing_to_kh}.
\item $\srdiecko:\catH \to \catP: M \mapsto\srdiecko M:= \Hom^C_\Ug(\Ud,M)$. See Section \ref{subsection_srdiecko}. 
\item By $\srdco:\catH \to \catA$ we denote $\srdiecko$ when we want to stress that it goes to $\catA$. Section \ref{subsection_srdco}.
\item $\Yoneda_{\text{something}}$ is the functor represented by something (Section \ref{section_represented_functors}). 
\item $\Yoneda: \catF \to \catCTV$ the functor represented by $(Q\ctp A,p) \in \Split(\catF)$ (Section \ref{subsecion_Yoneda_from_catF}). We denote by the same letter also the functor $\Yoneda: \catFh \to \catCTVkh$ represented by $(Q\ctp A,p) \in \Split(\catFh)$.
\item $\YPhi: \catFhPhi \to \catCTVkh$ the functor represented by $(Q\ctp A,p^\Phi) \in \Split(\catFhPhi)$. Section \ref{subsection_YPhi}. 
\end{itemize}
Important objects
\begin{itemize}
\item $A = \srdiecko \field$ is a commutative algebra in $\catP$ (Section \ref{subsection_A}).
\item $C\in \ob(\catH)$ is $\Ug$ as its own module by left multiplication.
\item $Q = \srdiecko C \in \ob(\catP)$. It is useful, because the coalgebra $\srdco C$ in $\catA$ representing $\Yoneda$ is a direct summand of $Q\otimes A$. Section \ref{section_Y}.
\end{itemize}
Important maps:
\begin{itemize}
\item $\pi_M:\srdiecko(M)\to M$ a natural epimorphism in $\catH$ defined for $M\in \catH$ by 
$$ \pi_M: \Hom_\Ug^C(\Ud,M) \to M : s\mapsto s(1) \ .$$
In particular we have an epimorphism $\pi_C:Q \to C$ in $\catH$.  
\item We also denote $\eps_A := \pi_\field:A\to \field$. It is an augmentation on $A$ if we consider $A$ as an algebra in $\catH$. It is however not a morphism in $\catP$. Section \ref{subsection_eps_A}. On pictures, we denote $\eps_A$ by $\begin{tikzpicture}[baseline=0.0cm]
\filldraw [line width=1pt] (0.0,0.0) circle (0.08cm);
\node at (0.0,-0.5) {{\scriptsize $A$}};
\draw [line width=1pt] (0.0,-0.3) ..controls +(0.0,0.0) and +(0.0,0.0) .. (0.0,-0.08);
\end{tikzpicture}
$.
\item $\iota_C\in \catH(C,Q)$ and $s_0 \in Q$. Since $C$ is a free one-dimensional $\Ug$-module one can choose a right inverse $\iota_C:C\to Q$ of $\pi_C:Q \to C$ just by specifying an element $s_0 = \iota_C(1)\in Q$. Our choice of $s_0$ is described in Section \ref{subsection_choice__of_s0}. On our pictures, we denote the element $s_0\in Q$ (regarded as a linear map $\field \to Q$) as $\begin{tikzpicture}[baseline=0.0cm]
\draw [line width=1pt] (0.0,0.0) circle (0.08cm);
\node at (0.0,0.6) {{\scriptsize $Q$}};
\draw [line width=1pt] (0.0,0.08) ..controls +(0.0,0.17) and +(0.0,-0.17) .. (0.0,0.4);
\end{tikzpicture}
$.
\item Projection\footnote{By ``projection'' we want to say here that $p\circ p = p$. } $p \in \catA(Q\ctp A, Q\ctp A)$. The composition $\iota_C \circ \pi_C:Q\to Q$ is a projection in $\catH$ with image $C$. Applying $\srdco$ and composing with isomorphism $Q\ctp A \cong \srdco Q$ (which we get from Theorem \ref{fundamental_thm} part (2)) we get a projection in $\catA$
$$ p: Q\ctp A\xto\cong \srdco Q \xto{\srdco(\iota_C\circ \pi_C)} \srdco Q \xto\cong Q \ctp A .$$
This enables us to see $\srdco C$ as an object $(Q\ctp A, p)$ in $\Split(\catF)$. Details in Section \ref{section_Y}.
\item $\Delta \in \catH(C ,   C\ctp C)$, $\eps\in \catH(C,\field)$ is the coalgebra structure on $C\in \ob(\catH)$ coming from the bialgebra structure on $\Ug$.
\item $\Delta_Q \in \catF( Q\ctp A, Q\ctp Q \ctp A )$, $\eps_Q \in \catF(Q\ctp A, A)$ is a coalgebra structure on $(Q\ctp A, p)$ in $\Split(\catF)$. We get it as follows: $C$ is a coalgebra in $\catH$, so $\srdco C$ is a coalgebra in $\catA$. We see $\srdco C \in \catA$ as $(Q\ctp A,p)\in \Split(\catF)$ and thus get a coalgebra structure on $(Q\ctp A, p)$. Section \ref{section_Y}.
\item Projection $p^\Phi\in \catFhPhi(Q\ctp A, Q\ctp A)$. One can regard $ p\in\catFh(Q\ctp A, Q\ctp A)$ as a morphism in $\catFhPhi(Q\ctp A, Q\ctp A)$ although there it does not necessarily satisfy $p\circ p =p$. To force this condition, we use Lemma \ref{lem_pPhi} and get a projection $p^\Phi\in \catFhPhi(Q\ctp A, Q\ctp A)$. Section \ref{Map_on_Split_objects}.
\item $\Delta'_Q\in \catFhPhi(Q\ctp A , (Q\ctp A)\tp_A (Q\ctp A) )$, $\eps'_Q\in\catFhPhi(Q\ctp A, Q\ctp A)$ are an attempt to get a quasi-coalgebra structure on $(Q\ctp A, \p^\Phi)$ in $\Split(\catFhPhi)$. See the formulas (\ref{eq_Delta_prime}).
\item $r,s \in \catFhPhi(Q\ctp A, Q\ctp A)$ defined by the formulas (\ref{equa_f_and_g}) measure the failure of $\eps'_Q$ to be a counit for $\Delta'_Q$ as in Lemma \ref{lem_save_counit}.
\item $\Delta^\Phi_Q\in \catFhPhi(Q\ctp A , (Q\ctp A)\tp_A (Q\ctp A) )$, $\eps^\Phi_Q\in\catFhPhi(Q\ctp A, Q\ctp A)$ defined by (\ref{equa_def_DeltaQPhi}) give a quasi-coalgebra structure on $(Q\ctp A, \p^\Phi)\in \Split(\catFhPhi)$. (See also Lemma \ref{lem_save_counit}.)
\item    $m \mapsto \hat m : M \to \Yoneda(M\ctp A)$ is a natural isomorphism between $\catP\xto\forg \catCTV$ and $\catP \xto{\argument\ctp A} \catF \xto\Yoneda \catCTV$. It is a composition of isomorphisms
$$ M\xto\cong \catH(C,M) \underset{\cong}{\xto\srdco} \catA(\srdco C, \srdco M) \xto\cong \Split(\catF)\Big( (Q\ctp A, p) , M\ctp A \Big) $$
where the first map assigns to $m\in M$ a $\catH$-morphism $C\to M: 1\mapsto m$ and the last map uses the isomorphisms $\srdco C \cong (Q\ctp A, p)$ and $\srdco M\cong M\ctp A$. See Notation \ref{notation_m_hat}.

One can also regard $m\mapsto \hat m$  as an isomorphism between $\catPh\xto\forg \catCTVkh$ and $\catPh \xto{\argument\ctp A} \catFh \xto\Yoneda \catCTVkh$.
\item $\alpha: m \mapsto \hat m^\Phi = \hat m \cFhPhi p^\Phi : M \to \YPhi(M\ctp A)$ is a natural isomorphism between $\catPhPhi\xto\forg \catCTVkh$ and $\catPhPhi \xto{\argument\ctp A} \catFhPhi \xto\Yoneda \catCTVkh$. Section \ref{subsection_YPhi}.
\end{itemize}

\section{Quasi-bialgebras} \label{sect_Drinfeld_formulas}
For the convenience of the reader we gather here some results of Drinfeld from \cite{Drinfeld1}.
\begin{defn}
A quasi-bialgebra is an associative algebra $H$ together with algebra morphisms $\Delta:H \to H \otimes H$; $\eps: H\to \field$ and an invertible element $\Phi\in H\otimes H \otimes H$ satisfying 
\begin{align}
\label{B1} &(\id\otimes\Delta)\circ\Delta(a)=\Phi\cdot\big[(\Delta\otimes \id)\circ\Delta(a)\big]\cdot \Phi^{-1}\\ 
\label{B2} &(\id\otimes \id\otimes\Delta)(\Phi)\cdot(\Delta\otimes \id\otimes \id)(\Phi)=(1\otimes     \Phi)\cdot(\id\otimes\Delta\otimes \id)(\Phi)\cdot (\Phi\otimes 1)\\ 
\label{B3} &(\eps\otimes \id)\circ \Delta=\id\;;\quad\quad (\id\otimes\eps)\circ\Delta=\id\\
\label{B4} &(\id\otimes\eps\otimes \id)\Phi=1\otimes 1 
\end{align}
\end{defn}
\subsection{Quasi-triangular structure} $R$-matrix is an invertible element $R\in H\otimes H$ satisfying\footnote{$(a\otimes b\otimes c)^{312} = b\otimes c \otimes a$}
\begin{gather}
\Delta^{op}(a) = R\cdot \Delta(a)\cdot  R^{-1} \, ,\quad \forall a\in A\\
(\Delta\otimes\id)(R) = \Phi^{312}\cdot  R^{13}\cdot (\Phi^{132})^{-1} \cdot R^{23}\cdot  \Phi\\ 
(\id \otimes \Delta) (R) = (\Phi^{231})^{-1} \cdot R^{13} \cdot \Phi^{213} \cdot R^{12} \cdot \Phi^{-1}
\end{gather}
\subsection{Twisting}\label{section_twisting}
Given an invertible elements $F\in H\otimes H$ satisfying $$(\id\otimes\eps) F = 1 = (\eps\otimes \id) F$$
we get a new quasi-bialgebra structure on the algebra $H$ by keeping the same $\eps$ and replacing $\Delta$, $\Phi$ and $R$ (if the original was quasi-triangular) by
\begin{align}
&\widetilde\Delta(a) = F\cdot \Delta(a) \cdot F^{-1}  \, ,\quad \forall a\in A\\
&\widetilde\Phi = F^{23}\cdot (\id\otimes\Delta)(F) \cdot \Phi \cdot (\Delta\otimes\id)(F^{-1})\cdot ( F^{12})^{-1}\\
& \widetilde R = F^{21} \cdot R \cdot F
\end{align}

Next we recall what are the corresponding classical notions.
\subsection{Quasi-Lie bialgebras}
A quasi-Lie bialgebra is a triple $(\g, \delta, \fii)$ where
\begin{itemize}
\item $\g$ is a Lie-algebra;
\item $\delta$ is a 1-cocycle $\delta:\g \to \bigwedge^2 \g$;
\item $\fii \in \bigwedge^3 \g$
\end{itemize}
satisfying the following equations\footnote{$\Alt:\g^{\otimes n} \to \g^{\otimes n}: x_1\otimes\ldots\otimes x_n \mapsto \sum_{\sigma \in S_n} \sgn(\sigma)\cdot x_{\sigma(1)} \otimes \ldots \otimes x_{\sigma(n)} $, 
so we don't divide by $n!$.}
\begin{gather}
\half \Alt(\delta \otimes \id) \delta(x) =\big [ x\otimes 1\otimes 1 + 1\otimes x \otimes 1 + 1\otimes 1\otimes x,\ \fii \big] \quad \forall x\in \g \ ; \\
\Alt (\delta\otimes\id\otimes\id) (\fii) = 0\ .
\end{gather}

\subsection{Twisting of quasi-Lie bialgebras}
Let $(\g, \delta, \fii)$ be a quasi-Lie bialgebra and $f\in \bigwedge^2 \g$. We get a new quasi-Lie bialgebra structure $(\g, \widetilde\delta, \widetilde\fii)$ on the same Lie algebra $\g$ by taking
\begin{gather}
\label{twist_delta}  \widetilde\delta (x) := \delta(x) + [x\otimes 1 + 1\otimes x, f ]\  ;\\
\label{twist_fii}   \widetilde \fii = \fii + \half \Alt (\delta\otimes\id) f - CYB(f) 
\end{gather}
where
$$ CYB(f) := [f^{12}, f^{13} ] + [ f^{12}, f^{23}]+ [ f^{13}, f^{23}]$$
is the left hand side of the classical Yang-Baxter equation.

\subsection{Quasi-Hopf QUE algebras}
Quasi-Hopf QUE\footnote{QUE = quantum universal enveloping} algebra is a topological quasi-bialgebra $(A,\Delta, \eps, \Phi)$ over $\kh$ s.t.
\begin{enumerate}
\item $\Phi \equiv 1 \mod \hbar$;
\item $A/\hbar A$ is a universal enveloping algebra;
\item $A$ is a topologically free $\kh$-module;
\item $\Alt\ \Phi \equiv 0 \mod \hbar^2$.
\end{enumerate}
We twist as in the section (\ref{section_twisting}) by elements $F$ that satisfy also $F\equiv 1 \mod \hbar$. One can see that $\Alt\ \Phi \equiv 0 \mod \hbar^2$ is preserved under twisting.

\begin{thm}\label{thm_classical_limit}
Let $(A, \Delta, \eps, \Phi)$ be a quasi-Hopf QUE algebra and $ A/\hbar A = \mathcal U \g$.
\begin{enumerate}
\item By twisting one can achieve $\Phi \equiv 1 \mod \hbar^2$.
\item If we assume $\Phi \equiv 1 \mod \hbar^2$ then we get on $\g$ a structure of a quasi-Lie bialgebra by taking 
\begin{itemize}
\item $\fii := \frac{1}{\hbar^2}\cdot \Alt( \Phi) \mod \hbar$,
\item $\delta(x)  := \frac{1}{\hbar} \big( \Delta(x)- \Delta^{op} (x) \big)$.
\end{itemize}
\item If we twist $(A, \Delta, \eps, \Phi)$ by $F$ into $(A, \widetilde \Delta, \widetilde \eps, \widetilde \Phi)$ while keeping $\Phi \equiv\widetilde \Phi \equiv 1 \mod \hbar^2$
 then the quasi-Lie bialgebra  $(\g, \widetilde \delta, \widetilde \fii)$ corresponding to  $(A, \widetilde \Delta, \widetilde \eps, \widetilde \Phi)$ can be obtained from $(\g, \delta,\fii)$ by twisting via $$f= - \frac{1}{\hbar}\cdot  \Alt( F )\mod \hbar\ .$$.
\end{enumerate}
\end{thm}

\begin{thm}\label{thm_standard_form_of_quasitriang_quasibi}
Let $(A, \Delta, \eps, \Phi, R) $ be a quasi-triangular quasi-Hopf QUE algebra, $A/ \hbar A = \mathcal U \g$. Denote
$$ t := \frac{1}{\hbar} \cdot (R^{21} \cdot R - 1) \mod \hbar \quad \in\  \UE\g \otimes \UE\g\ .$$
Then
\begin{enumerate}
\item $t$ is a symmetric $\g$-invariant element of $\g\otimes \g$ and does not change under twists.
\item By twisting of $(A, \Delta, \eps, \Phi, R)$ one can achieve that we have both
\begin{itemize} 
\item $\frac{R-1}{\hbar} \equiv \frac t 2  \mod \hbar$;
\item $\Phi \equiv 1 \mod \hbar^2 $.
\end{itemize}
In that case one has
\begin{enumerate}
\item $\frac{1}{\hbar^2} \Alt(\Phi)  \equiv \frac{1}{4} [t^{12}, t^{23}] \mod \hbar$;
\item $\Delta \equiv \Delta^{op} \mod \hbar^2 $.
\end{enumerate}
\end{enumerate}
\end{thm}

\begin{defn}\label{defn_associated_quasi_lie_bialgebra}
Let $\g$ be a Lie algebra and $t\in \g\otimes\g$ a symmetric $\g$-invariant element. We can define on $\g$ a quasi-Lie bialgebra structure by taking
\begin{gather*}
\delta := 0\ ;\quad 
\fii :=\frac{1}{4} [t^{12}, t^{23}]=-\frac{1}{4}CYB(t,t)\ .
\end{gather*}
We call it the quasi-Lie bialgebra associated to the pair $(\g,t)$.
\end{defn}

\subsection{Double quasi-Lie bialgebra} \label{subsection_Double_quasi-Lie bialgebra}

For a quasi-Lie bialgebra $(\g,\delta, \fii)$ we form a topological\footnote{We equip $\g$ with the discrete and $\g^*$ with the corresponding strong topology (see the subsection \ref{subs_strong_topology}).} vector space $\p := \g \oplus \g^*$ and take the unique Lie bracket $[,]_\p$ satisfying the following properties:
\begin{enumerate}
\item $[,]_\p$ restricted to $\g\tp\g$ is $[,]_\g$.
\item If we regard $\delta$ and $\fii$ as maps (see the claim \ref{claim_dual_dualu})
$$ \delta: \g^* \ctp \g^* \to \g^* ;\quad \fii:\g^* \ctp \g^* \to \g  $$
then $[,]_\p$ restricted to $\g^*\tp \g^*$ is equal to $\delta +\fii$.
\item The canonical scalar product on $\p$ given by the pairing between $\g$ and $\g^*$ is an invariant scalar product on the Lie algebra $(\p, [,]_\p)$.
\end{enumerate}

We choose a vector space basis $\set{e_i}$ of $\g$ and denote the dual basis\footnote{Elements of $\g^*$ are expressed as possibly infinite linear combinations of $\set{e^i}$. We understand such infinite sums in topological sense.} $\set{e^i}$.  We denote the ``structural constants'' of $(\g,\delta, \fii)$ as follows
$$
[e_i ,e_j] = c^k_{ij} e_k ; \quad \delta(e_i) = \delta^{jk}_i e_j \otimes e_k ; \quad \fii = \fii^{ijk} e_i\otimes e_j \otimes e_k\ .
$$
The bracket on $\p$ is then given by the formulas 
\begin{gather*}
[e_i, e_j]_\p = c^k_{ij} e_k; \quad
[e^i, e^j ]_\p= \delta^{ij}_k e^k - \fii^{ijl} e_l; \\
[e_i, e^j]_\p = c^j_{ki} e^k + \delta^{jl}_i e_l;  \quad [e^i, e_j] = - c^i_{kj} e^k -\delta^{il}_j e_l \ . 
\end{gather*}

We denote the inverse of the invariant scalar product on $\p$ by $\Omegap$. That is 
$$\Omegap:= e_i \otimes e^i + e^j \otimes e_j \in \p\ctp\p$$
is a symmetric $\p$-invariant element of $\p\ctp\p$. Using it, we equip $\p$ with the associated quasi-Lie bialgebra structure as in the definition \ref{defn_associated_quasi_lie_bialgebra}.
Note that $\g$ is not a sub-quasi-Lie-bialgebra of $\p$ because the cobracket is trivial on $\p$ but possibly not on $\g$.

\begin{lem}\label{lem_twisting_p}
Let $\p'$ be the quasi-Lie bialgebra that we get from $\p$ via twisting by 
$$f := \half (e_i \otimes e^i - e^j \otimes e_j)\ .$$
Then $\g$ is a sub-quasi-Lie-bialgebra of $\p'$.
\end{lem}
\begin{proof}
First we show that $\g$ is closed under the cobracket $\delta_{\p'}$ and that the restriction $\delta_{\p'}|_\g = \delta$. Let $r=e_i\otimes e^i=t/2+f$. From (\ref{twist_delta}) we have
\begin{align*}
\delta_{\p'}(e_i) &= \delta_\p(e_i) + [e_i\otimes 1 + 1\otimes e_i, f] = \half [e_i \otimes 1 + 1\otimes e_i, r- r^{op}] = \\
&=  \half \Big( [e_i \otimes 1 + 1\otimes e_i, r] - [e_i \otimes 1 + 1\otimes e_i, r]^{op} \Big)
\end{align*}
\begin{align*}
[e_i \otimes 1 + 1\otimes e_i, r] &= [e_i \otimes 1 + 1\otimes e_i, e_j\otimes e^j] =\\
&= 
\begin{array}{r|l}
[e_i, e_j] \otimes e^j \quad &= c^k_{ij} e_k\otimes e^j \\
+ e_j\otimes [e_i,e^j] \quad &= c^j_{ki} e_j\otimes e^k + \delta^{jl}_i e_j \otimes e_j  
\end{array} \\
&= \delta(e_i)
\end{align*}

Next we want to show that $\fii_{\p'} = \fii$.  By the formula \eqref{twist_fii} (using that $\delta_{\p}=0$) we have
$$\fii_{\p'}=-CYB(t/2)-CYB(f)=-CYB(t/2+f)=-CYB(r).$$
The middle equality follows from the identity $[t^{12},x^1+x^2]=0$ for every $x\in\p$ (invariance of $t$), which, together with the skew-symmetry of $f$, implies that the mixed term
$$[t^{12},f^{13}]+[t^{12},f^{23}]+[t^{13},f^{23}]+
[f^{12},t^{13}]+[f^{12},t^{23}]+[f^{13},t^{23}]$$
vanishes.

We have
\begin{multline*}
-\phi_{\p'}=CYB(r,r)=[r^{12},r^{13}]+[r^{12},r^{23}]+[r^{13},r^{23}]\\
=[e_i, e_j] \otimes e^i \otimes e^j+e_i\otimes [e^i,e_j]\otimes e^j+e_i\otimes e_j \otimes [e^i, e^j]\\
=c_{ij}^k e_k\otimes e^i\otimes e^j 
+(c_{jk}^i e_i\otimes e^k\otimes e^j-\delta_j^{ik}e_i\otimes e_k\otimes e^j)\\
+(\delta_k^{ij}e_i\otimes e_j\otimes e^k-\phi^{ijk}e_i\otimes e_j\otimes e_k)\\
=-\phi^{ijk}e_i\otimes e_j\otimes e_k=-\phi,
\end{multline*}
as we wanted to show.
\end{proof}


\section{From categories to quasi-bialgebras} \label{sect_abstract_nonsense}
We assume that all our categories have strict unit objects. By $\catV$ we denote the category of vector spaces but the assertions of this section hold also if we replace $\catV$ by $\catCTV$ or $\catCTVkh$ and $\tp$ by $\ctp$.

\begin{defn}
A functor $F: \catC \to \catD$ between two monoidal categories is a \emph{quasi-monoidal functor} if it is equipped with $\catD$-morphisms 
\begin{equation*}
\lambda_0 : I_\catD \to  F(I_\catC) \quad \text{and}\quad \lambda_{M,N}:FM\tp FN \to F(M\tp N) 
\end{equation*}
natural in $M,N\in \ob(\catC)$, such that the following diagrams commute:
\begin{equation}\label{diags_quasi_functor}
\begin{tikzpicture}[label distance =-0.5 mm]	
	\node(lx) at (3,0){};
	\node(ly) at (0, -2){};
	\pgfresetboundingbox;

	\node(a00) {$ FM \tp F(I_\catC)$};
	\node(a01) at ($(a00)+(lx)$) {$ F(M \tp I_\catC ) $};
	\node(a10) at ($(a00)+(ly)$) {$ FM \tp I_\catD $};
	\node(a11) at ($(a00)+(lx)+(ly)$) {$FM$};	
	
	\draw[->] (a00)-- node[above] {$\lambda_{M,I}$} (a01);
	\draw[<-] (a00)-- node[fill=white] {$\id_{FM} \tp \lambda_0$} (a10);
	\draw[->] (a01)-- node[right] {$=$} (a11);
	\draw[->] (a10)-- node[below] {$=$} (a11);
\end{tikzpicture}
\begin{tikzpicture}[label distance =-0.5 mm]	
	\node(lx) at (3,0){};
	\node(ly) at (0, -2){};
	\pgfresetboundingbox;

	\node(a00) {$ F(I_\catC) \tp FM $};
	\node(a01) at ($(a00)+(lx)$) {$ F( I_\catC  \tp M) $};
	\node(a10) at ($(a00)+(ly)$) {$  I_\catD \tp FM$};
	\node(a11) at ($(a00)+(lx)+(ly)$) {$FM$};	
	
	\draw[->] (a00)-- node[above] {$\lambda_{I,M}$} (a01);
	\draw[<-] (a00)-- node[fill=white] {$ \lambda_0 \tp \id_{FM}$} (a10);
	\draw[->] (a01)-- node[right] {$=$} (a11);
	\draw[->] (a10)-- node[below] {$=$} (a11);
\end{tikzpicture}
\end{equation}

The functor $F$ will be called \emph{strong quasi-monoidal}, if the maps $\lambda_0$ and $\lambda_{M,N}$ are isomorphisms.
\end{defn}

\begin{lem}\label{lem_quasibi_from_functor}
Let $\catC$ be a monoidal category and $F:\catC \to \catV$ a strong quasi-monoidal functor with the property that the obvious algebra homomorphism\footnote{Note that this homomorphism does not depend on the monoidal structure of $\catC$ and $F$, just on $F$ as a functor.}
\begin{equation} \label{map_Theta}
\Theta: \Big( \End(\catC \xto{F} \catV) \Big)^{\otimes k}\, \longto \; \End\big( \catC^{\times k} \xto{F^{\times k}} \catV^{\times k} \xto{\otimes } \catV \big)
\end{equation}
is an isomorphism $\forall k$. Then $\End(F)$ is a quasi-bialgebra. If $\catC$ is braided, then $\End(F)$ becomes a quasi-triangular quasi-bialgebra.
\end{lem}
\begin{proof}
The composition $\circ$ turns $\End(F)$ into an algebra. We want to get the other parts of the bialgebra structure. 

The monoidal structure on $F$ consists of maps $\lambda_{M,N} : F(M) \otimes F(N) \xto{\cong} F(M\otimes N)$ and $\lambda_0: F(I) \to \field$ where $I\in \ob(\catC)$ and $\field\in\ob(\catV)$ are the unit objects. 
Let's denote by $\unitcat$ the category having just one object and one morphism and by 
$$ \iota_I: \unitcat \to \catC\quad\text{and} \quad \iota_\field : \unitcat \to \catV$$
the functors mapping the unique object of $\unitcat$ to $I$ and $\field$ respectively. One can now regard both $\lambda_{M,N}$ and $\lambda_0$ as natural transformations:\footnote{In the second diagram, one can think of the top $\unitcat$ as $\catC^{\times 0}$, of the bottom $\unitcat$ as $\catV^{\times 0}$ and of the identity between them as $F^{\times 0}$. This way the diagram for $\lambda_0$ becomes analogous to the one for $\lambda_{M,N}$.}
$$
\begin{tikzpicture}[label distance =-0.5 mm]	
	\node(lx) at (2.1,0){};
	\node(ly) at (0, -2){};
	\pgfresetboundingbox;

	\node(a00) {$\catC \times \catC $};
	\node(a01) at ($(a00)+(lx)$) {$\catC$};
	\node(a10) at ($(a00)+(ly)$) {$\catV \times \catV$};
	\node(a11) at ($(a00)+(lx)+(ly)$) {$\catV$};	
	
	\node(stred) at ($(a00)+0.5*(lx)+0.5*(ly)$) [inner sep = 0pt, shape=rectangle, rotate=30, label=above:$\lambda$, label=-115:{$\cong$}] {{\Huge $\Rightarrow$}};

	\draw[->] (a00)-- node[above] {$\otimes$} (a01);
	\draw[->] (a00)-- node[left] {$F \times F$} (a10);
	\draw[->] (a01)-- node[right] {$F$} (a11);
	\draw[->] (a10)-- node[below] {$\otimes$} (a11);
\end{tikzpicture}
\quad
\begin{tikzpicture}[label distance =-0.5 mm]	
	\node(lx) at (2.1,0){};
	\node(ly) at (0, -2){};
	\pgfresetboundingbox;

	\node(a00) {$\unitcat $};
	\node(a01) at ($(a00)+(lx)$) {$\catC$};
	\node(a10) at ($(a00)+(ly)$) {$\unitcat$};
	\node(a11) at ($(a00)+(lx)+(ly)$) {$\catV$};	
	
	\node(stred) at ($(a00)+0.5*(lx)+0.5*(ly)$) [inner sep = 0pt, shape=rectangle, rotate=30, label=above:$\lambda_0$, label=-115:{$\cong$}] {{\Huge $\Rightarrow$}};

	\draw[->] (a00)-- node[above] {$\iota_I$} (a01);
	\draw[->] (a00)-- node[left] {$=$} (a10);
	\draw[->] (a01)-- node[right] {$F$} (a11);
	\draw[->] (a10)-- node[below] {$\iota_\field$} (a11);
\end{tikzpicture}
$$
Given $\alpha \in \End(F)$, we can define 
\begin{gather*}
\Delta(\alpha) \in \End(\catC\times \catC \xto{F\times F} \catV \times \catV \xto\otimes \catV) \cong \End(F) \tp \End(F) \\
\eps(\alpha) \in \End( \iota_\field) \cong \field 
\end{gather*}
as the following compositions of natural transformations:
$$
\begin{tikzpicture}[baseline=-2cm]	
	\node(lx) at (1.5,0){};
	\node(ly) at (0, -2){};
	\pgfresetboundingbox;

	\node(a10)  {$\catC\times \catC$};
	
	\node(a20) at ($(a10)+(ly)-1.5*(lx)$) {$\catV\times \catV$};
	\node(a21) at ($(a10)+1*(ly)$) {$\catC$};	
	\node(a22) at ($(a10)+1*(ly)+1.5*(lx)$) {$\catV\times \catV$};

	\node(a30) at ($(a10)+2*(ly)$) {$\catV$};
	
	\node(stred3) at ($0.5*(a20)+0.5*(a21)$) [inner sep = 0pt, shape=rectangle, rotate=0, label=above:$\lambda$] {{\Huge $\Rightarrow$}};	
	\node(stred4) at ($0.5*(a21)+0.5*(a22)$) [inner sep = 0pt, shape=rectangle, rotate=0, label=above:$\lambda^{-1}$] {{\Huge $\Rightarrow$}};
	\node(stred5) at ($0.5*(a21)+0.5*(a30)$) [inner sep = 0pt, shape=rectangle, rotate=0, label=above:$\alpha$] {{\Huge $\Rightarrow$}};

	\draw[->] (a10)-- node[right] {$\otimes$} (a21);
	\draw[->] (a10)-- node[left] {$F \times F$} (a20);
	\draw[->] (a10)-- node[right] {$F \times F$} (a22);

	\draw[->] (a20)-- node[left,below] {$\otimes$} (a30);
	\draw[->] (a21)to[bend right=45] node[left, near start] {$F$} (a30);
	\draw[->] (a21)to[bend left=45] node[right, near start] {$F$} (a30);
	\draw[->] (a22)-- node[right, below] {$\otimes$} (a30);

\end{tikzpicture}
\quad
\begin{tikzpicture}[baseline=-2cm]	
	\node(lx) at (1.5,0){};
	\node(ly) at (0, -2){};
	\pgfresetboundingbox;

	\node(a10)  {$\unitcat$};
	
	\node(a20) at ($(a10)+(ly)-1.5*(lx)$) {$\unitcat$};
	\node(a21) at ($(a10)+1*(ly)$) {$\catC$};	
	\node(a22) at ($(a10)+1*(ly)+1.5*(lx)$) {$\unitcat$};

	\node(a30) at ($(a10)+2*(ly)$) {$\catV$};
	
	\node(stred3) at ($0.5*(a20)+0.5*(a21)$) [inner sep = 0pt, shape=rectangle, rotate=0, label=above:$\lambda_0$] {{\Huge $\Rightarrow$}};	
	\node(stred4) at ($0.5*(a21)+0.5*(a22)$) [inner sep = 0pt, shape=rectangle, rotate=0, label=above:$\lambda_0^{-1}$] {{\Huge $\Rightarrow$}};
	\node(stred5) at ($0.5*(a21)+0.5*(a30)$) [inner sep = 0pt, shape=rectangle, rotate=0, label=above:$\alpha$] {{\Huge $\Rightarrow$}};

	\draw[->] (a10)-- node[right] {$\iota_I$} (a21);
	\draw[->] (a10)-- node[left] {$=$} (a20);
	\draw[->] (a10)-- node[right] {$=$} (a22);

	\draw[->] (a20)-- node[left,below] {$\iota_\field$} (a30);
	\draw[->] (a21)to[bend right=45] node[left, near start] {$F$} (a30);
	\draw[->] (a21)to[bend left=45] node[right, near start] {$F$} (a30);
	\draw[->] (a22)-- node[right, below] {$\iota_\field$} (a30);

\end{tikzpicture}
$$
It's obvious that $\Delta$ and $\eps$ are homomorphisms of algebras --- they are just compositions of algebra homomorphisms
\begin{gather*}
\Delta: \End(F) \to \End(F\circ \otimes) \to \End\big( \otimes \circ (F\times F) \big) \xto{\Theta^{-1}}  \big( \End(F) \big)^{\otimes 2} \\
 \eps: \End(F) \to \End( F\circ \iota_I ) \to \End( \iota_\field) \cong \field . 
\end{gather*}

We also need the ``associator'' $\Phi\in \big(\End(F) \big)^{\otimes 3}$. On each face of the cube
$$
\begin{tikzpicture}[label distance =-0.5 mm]	
	\node(lx) at (4,0){};
	\node(ly) at (2.5, 1.5){};
	\node(lz) at (0, 4){};
	\pgfresetboundingbox;

	\node(a000) {$\catV \times \catV\times \catV $};
	\node(a100) at ($(a000)+(lx)$) {$\catV \times\catV$};
	\node(a010) at ($(a000)+(ly)$) {$\catV \times\catV$};
	\node(a110) at ($(a000)+(lx)+(ly)$) {$\catV$};	
	
	\node(a001) at ($(a000)+(lz)$)  {$\catC \times \catC\times \catC $};
	\node(a101) at ($(a001)+(lx)$) {$\catC \times\catC$};
	\node(a011) at ($(a001)+(ly)$) {$\catC \times\catC$};
	\node(a111) at ($(a001)+(lx)+(ly)$) {$\catC$};

	\draw[->] (a000)-- node[below right] {$\id \times \otimes$} (a010);
	\draw[->] (a000)-- node[below] {$\otimes \times \id$} (a100);
	\draw[->] (a010)-- node[below] {$\otimes$} (a110);
	\draw[->] (a100)-- node[below] {$\otimes$} (a110);

	\draw[->] (a001)-- node[above left] {$\id \times \otimes$} (a011);
	\draw[->] (a011)-- node[above] {$\otimes$} (a111);
	\draw[->] (a101)-- node[above] {$\otimes$} (a111);

	\draw[->] (a001)-- node[left] {$F\times F\times F$} (a000);
	\draw[->, white, line width = 10pt] (a101)-- (a100);
	\draw[->] (a101)-- node[right, near start] {$F\times F$} (a100);	
	\draw[->] (a011)-- node[left, near end] {$F\times F$} (a010);
	\draw[->] (a111)-- node[right] {$F$} (a110);

	\draw[->, white, line width = 10pt] (a001)-- (a101);	
	\draw[->] (a001)-- node[above, fill=white] {$\otimes \times \id$} (a101);
\end{tikzpicture}
$$
we have a natural isomorphism. If we compose all of them we get an automorphism of $$\catC\times\catC\times\catC \xto{F\times F \times F} \catV\times\catV\times\catV \xto{\otimes^{(3)} } \catV$$
that corresponds via $\Theta^{-1}$ to an element $\Phi\in \big(\End(F) \big)^{\otimes 3}$.

We should now prove that our $\Delta$, $\eps$ and $\Phi$ satisfy the axioms of a quasi-bialgebra.
First, it is clear from the cube-diagram that 
$$\big( (\Delta\otimes\id)\circ \Delta\big)\alpha,\, \big( (\id \otimes \Delta) \circ \Delta \big) \alpha \in \End\big( \catC\times\catC\times\catC \to \catV\times\catV\times\catV\xto{\otimes} \catV\big)$$
 differ by a conjugation by $\Phi$.

The next thing to check is the pentagon relation. The pentagon axiom in the monoidal category $\catC$ means the following equality of two natural transformations:
$$
\begin{tikzpicture}[baseline=-2cm,scale=0.75]	
	\node(lx) at (2,0){};
	\node(ly) at (0, -2){};
	\pgfresetboundingbox;

	\node(a00) {$\catC \times \catC \times\catC \times \catC$};

	\node(a10) at ($(a00)-(lx)+(ly)$) {$\catC\times\catC\times\catC$};
	\node(a11) at ($(a00)+(lx)+(ly)$) {$\catC \times \catC \times\catC$};

	\node(a20) at ($(a00)+2*(ly)-1.5*(lx)$) {$\catC\times \catC$};
	\node(a21) at ($(a00)+2*(ly)$) {$\catC\times \catC$};	
	\node(a22) at ($(a00)+2*(ly)+1.5*(lx)$) {$\catC\times \catC$};

	\node(a30) at ($(a00)+3*(ly)$) {$\catC$};
	
	\node(stred) at ($(a00)+(ly)$) {{\Huge $=$}};	
	\node(stred1) at ($0.5*(a20)+0.5*(a21)$) {{\huge $\Rightarrow$}};	
	\node(stred2) at ($0.5*(a21)+0.5*(a22)$) {{\huge $\Rightarrow$}};
	
	\draw[->] (a00)-- node[left] {$\scriptstyle\otimes\times\id\times\id$} (a10);
	\draw[->] (a00)-- node[right] {$\scriptstyle\id\times\id\times\otimes$} (a11);

	\draw[->] (a10)-- node[right, near start] {$\scriptstyle\id\times\otimes$} (a21);
	\draw[->] (a11)-- node[left] {$\scriptstyle\otimes\times\id$} (a21);
	\draw[->] (a10)-- node[left] {$\scriptstyle\otimes\times\id$} (a20);
	\draw[->] (a11)-- node[right] {$\scriptstyle\id\times\otimes$} (a22);

	\draw[->] (a20)-- node[left] {$\scriptstyle\otimes$} (a30);
	\draw[->] (a21)-- node[left] {$\scriptstyle\otimes$} (a30);
	\draw[->] (a22)-- node[left] {$\scriptstyle\otimes$} (a30);

\end{tikzpicture}
\text{ {\huge $=$ } }
\begin{tikzpicture}[baseline=2.5cm,scale=0.75]
	\node(lx) at (2,0){};
	\node(ly) at (0, 2){};
	\pgfresetboundingbox;

	\node(a00) {$\catC$};

	\node(a10) at ($(a00)-(lx)+(ly)$) {$\catC\times\catC$};
	\node(a11) at ($(a00)+(lx)+(ly)$) {$\catC \times \catC$};

	\node(a20) at ($(a00)+2*(ly)-1.7*(lx)$) {$\catC\times \catC \times\catC$};
	\node(a21) at ($(a00)+2*(ly)$) {$\catC\times \catC \times\catC$};	
	\node(a22) at ($(a00)+2*(ly)+1.7*(lx)$) {$\catC\times \catC \times\catC$};

	\node(a30) at ($(a00)+3*(ly)$) {$\catC \times\catC \times\catC \times\catC$};
	
	\node(stred) at ($(a00)+(ly)$) {{\huge $\Rightarrow$}};	
	\node(stred1) at ($0.5*(a20)+0.5*(a21)$) {{\huge $\Rightarrow$}};	
	\node(stred2) at ($0.5*(a21)+0.5*(a22)$) {{\huge $\Rightarrow$}};
	
	\draw[<-] (a00)-- node[left] {$\scriptstyle\otimes$} (a10);
	\draw[<-] (a00)-- node[right] {$\scriptstyle\otimes$} (a11);

	\draw[<-] (a10)-- node[right, near start] {$\scriptstyle\otimes\times\id$} (a21);
	\draw[<-] (a11)-- node[left] {$\scriptstyle\id\times\otimes$} (a21);
	\draw[<-] (a10)-- node[left] {$\scriptstyle\otimes\times\id$} (a20);
	\draw[<-] (a11)-- node[right] {$\scriptstyle\id\times\otimes$} (a22);

	\draw[<-] (a20)-- node[left] {$\scriptstyle\otimes\times\id\times\id$} (a30);
	\draw[<-] (a21)-- node[fill=white] {$\scriptstyle\id\times\otimes\times\id$} (a30);
	\draw[<-] (a22)-- node[right] {$\scriptstyle\otimes\times\id\times\id$} (a30);

\end{tikzpicture}
$$
If we consider the first of these diagrams as top base face of a $3$-dimensional prism with vertical arrows formed by the functors $F$, we get a polyhedron whose faces correspond to some natural isomorphisms. Composing all of them we get an endomorphism of the functor
$$\catC\times\catC\times\catC\times\catC\xto{F\times F\times F\times F} \catV\times\catV\times\catV\times\catV\xto{\otimes^{(4)}} \catV$$
that corresponds to an element of $\big( \End(F)\big)^{\otimes 4}$.
From the second diagram we get another such element and the above equality gives us equality of the two. This should be the pentagon axiom of the bialgebra definition.

To check the  axioms for $\eps$, let's rewrite the diagrams (\ref{diags_quasi_functor}) as equalities between natural transformations:
\begin{equation}\label{equa_xxqq1}
\begin{tikzpicture}[label distance =-0.5 mm, baseline=-1 cm]	
	\node(lx1) at (2.2,0){};
	\node(lx2) at (3.5,0){};
	\node(lx3) at (2.2,0){};
	\node(lx4) at (3.1,0){};
	\node(ly) at (0, -2){};
	\pgfresetboundingbox;

	\node(a00) {$\catC $};
	\node(a01) at ($(a00)+(lx1)$) {$\catC \times \unitcat$};
	\node(a02) at ($(a01)+(lx2)$) {$\catC\times \catC$};
	\node(a03) at ($(a02)+(lx3)$) {$\catC$};
	
	\node(a10) at ($(a00)+(ly)$) {$\catV$};
	\node(a11) at ($(a01)+(ly)$) {$\catV\times\unitcat$};
	\node(a12) at ($(a02)+(ly)$) {$\catV\times\catV$};
	\node(a13) at ($(a03)+(ly)$) {$\catV$};

	\draw[->] (a00)-- node[above] {$=$} (a01);
	\draw[->] (a01)-- node[above] {$\id_\catC \times \iota_I$} (a02);
	\draw[->] (a02)-- node[above] {$\otimes$} (a03);

	\draw[->] (a00)-- node[left] {$F$} (a10);
	\draw[->] (a01)-- node[fill=white] {$F\times\id_\unitcat$} (a11);
	\draw[->] (a02)-- node[fill=white] {$F\times F$} (a12);
	\draw[->] (a03)-- node[right] {$F$} (a13);

	\draw[->] (a10)-- node[below] {$=$} (a11);
	\draw[->] (a11)-- node[below] {$ \id_\catC \times \iota_\field$} (a12);
	\draw[->] (a12)-- node[below] {$\otimes$} (a13);

	\node(stred1) at ($(a00)+0.4*(lx1)+0.5*(ly)$) [inner sep = 0pt, shape=rectangle, rotate=30] {{\Huge $=$}};
	\node(stred2) at ($(a01)+0.65*(lx2)+0.5*(ly)$) [inner sep = 0pt, shape=rectangle, rotate=30, label=90:{\small$\id_F\times\lambda_0$}, label=-115:{$\cong$}] {{\Huge $\Rightarrow$}};
	\node(stred3) at ($(a02)+0.6*(lx3)+0.5*(ly)$) [inner sep = 0pt, shape=rectangle, rotate=30, label=above:{\small$\lambda$}, label=-115:{$\cong$}] {{\Huge $\Rightarrow$}};

\end{tikzpicture}
\; =\; \id_F
\end{equation}

\begin{equation}\label{equa_xxqq2}
\begin{tikzpicture}[label distance =-0.5 mm, baseline=-1 cm]	
	\node(lx1) at (2.2,0){};
	\node(lx2) at (3.5,0){};
	\node(lx3) at (2.2,0){};
	\node(lx4) at (3.1,0){};
	\node(ly) at (0, -2){};
	\pgfresetboundingbox;

	\node(a00) {$\catC $};
	\node(a01) at ($(a00)+(lx1)$) {$\unitcat \times \catC$};
	\node(a02) at ($(a01)+(lx2)$) {$\catC\times \catC$};
	\node(a03) at ($(a02)+(lx3)$) {$\catC$};
	
	\node(a10) at ($(a00)+(ly)$) {$\catV$};
	\node(a11) at ($(a01)+(ly)$) {$\unitcat\times \catV$};
	\node(a12) at ($(a02)+(ly)$) {$\catV\times\catV$};
	\node(a13) at ($(a03)+(ly)$) {$\catV$};

	\draw[->] (a00)-- node[above] {$=$} (a01);
	\draw[->] (a01)-- node[above] {$\iota_I \times \id_\catC$} (a02);
	\draw[->] (a02)-- node[above] {$\otimes$} (a03);

	\draw[->] (a00)-- node[left] {$F$} (a10);
	\draw[->] (a01)-- node[fill=white] {$id_\unitcat \times F$} (a11);
	\draw[->] (a02)-- node[fill=white] {$F\times F$} (a12);
	\draw[->] (a03)-- node[right] {$F$} (a13);

	\draw[->] (a10)-- node[below] {$=$} (a11);
	\draw[->] (a11)-- node[below] {$   \iota_\field \times \id_\catC$} (a12);
	\draw[->] (a12)-- node[below] {$\otimes$} (a13);

	\node(stred1) at ($(a00)+0.4*(lx1)+0.5*(ly)$) [inner sep = 0pt, shape=rectangle, rotate=30] {{\Huge $=$}};
	\node(stred2) at ($(a01)+0.65*(lx2)+0.5*(ly)$) [inner sep = 0pt, shape=rectangle, rotate=30, label=90:{\small$\lambda_0 \times \id_F$}, label=-115:{$\cong$}] {{\Huge $\Rightarrow$}};
	\node(stred3) at ($(a02)+0.6*(lx3)+0.5*(ly)$) [inner sep = 0pt, shape=rectangle, rotate=30, label=above:{\small$\lambda$}, label=-115:{$\cong$}] {{\Huge $\Rightarrow$}};

\end{tikzpicture}
\; =\; \id_F
\end{equation}
We get (\ref{B3}) directly from (\ref{equa_xxqq2}) and (\ref{equa_xxqq1}). Really, if  $\alpha \in \End(F)$ then $ \big( (\eps \tp \id )\circ \Delta \big) (\alpha)$ corresponds to conjugation of $\alpha$ by the left hand side of (\ref{equa_xxqq2}). So from (\ref{equa_xxqq2}) we get that it is just $\alpha$.

In order to see (\ref{B4}) recall that an axiom of monoidal categories states that the natural transformation
\begin{equation}\label{equa_uuu}
\begin{tikzpicture}[label distance =-0.5 mm]	
	\node(lx1) at (1.9,0){};
	\node(lx2) at (3.7,0){};
	\node(lx3) at (2,0){};
	\node(lx4) at (2,0){};
	\node(ly) at (0, -2){};
	\node(lz) at (0, 2){};
	\pgfresetboundingbox;

	\node(a00) {$\catC \times \catC $};
	\node(a01) at ($(a00)+(lx1)$) {$\catC\times\unitcat\times\catC$};
	\node(a02) at ($(a01)+(lx2)$) {$\catC\times\catC\times\catC$};
	\node(a03) at ($(a02)+(lx3)+(lz)$) {$\catC \times\catC$};
	\node(a031) at ($(a02)+(lx3)-(lz)$) {$\catC \times\catC$};
	\node(a04) at ($(a02)+(lx3)+(lx4)$) {$\catC$};
	
	\node(stred) at ($(a02)+(lx3)$) [inner sep = 0pt, shape=rectangle, rotate=90, label=115:$\phi$, label=-115:{$\cong$}] {{\Huge $\Rightarrow$}};
	
	\draw[->] (a00)-- node[above] {$=$} (a01);
	\draw[->] (a01)-- node[above] {{\small $\id\times\iota_I \times \id$}} (a02);
	\draw[->] (a02)-- node[left] {$\id\times\tp$} (a03);
	\draw[->] (a02)-- node[left] {$\otimes\times\id$} (a031);
	\draw[->] (a03)-- node[right] {$\tp$} (a04);
	\draw[->] (a031)-- node[right] {$\tp$} (a04);	

\end{tikzpicture}
\end{equation}
is equal to the identity on the functor $\catC\times \catC \xto\tp \catC$. The left hand side of (\ref{B4}) corresponds to the endomorphism of $\catC\times\catC \xto{F\times F} \catV \times \catV \xto\tp \catV$ that we get as a composition of natural isomorphisms on the surface of the following diagram:
$$
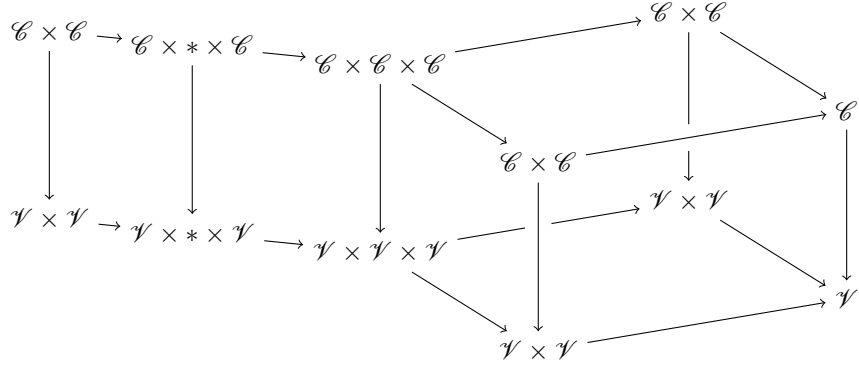
\begin{tikzpicture}[label distance =-0.5 mm]
	\node(lxx) at (1,-0.1){};	
	\node(lx1) at ($1.9*(lxx)$){};
	\node(lx2) at ($2.5*(lxx)$){};
	\node(lx3) at ($3.1*(lxx)$){};
	\node(lx4) at ($3.1*(lxx)$){};
	\node(ly) at (0, -2.5){};
	\node(lz) at (1, 1){};
	\pgfresetboundingbox;

	\node(a00) {$\catC \times \catC $};
	\node(a01) at ($(a00)+(lx1)$) {$\catC\times\unitcat\times\catC$};
	\node(a02) at ($(a01)+(lx2)$) {$\catC\times\catC\times\catC$};
	\node(a03) at ($(a02)+(lx3)-(lz)$) {$\catC \times\catC$};
	\node(a031) at ($(a02)+(lx3)+(lz)$) {$\catC \times\catC$};
	\node(a04) at ($(a02)+(lx3)+(lx4)$) {$\catC$};
	
	\node(a10) at ($(a00)+(ly)$) {$\catV \times\catV$};
	\node(a11) at ($(a01)+(ly)$) {$\catV \times \unitcat\times \catV$};
	\node(a12) at ($(a02)+(ly)$) {$\catV \times\catV \times\catV$};
	\node(a13) at ($(a03)+(ly)$) {$\catV \times\catV$};
	\node(a131) at ($(a031)+(ly)$) {$\catV \times\catV$};
	\node(a14) at ($(a04)+(ly)$) {$\catV$};

	\draw[->] (a00)-- node[above] {} (a01);
	\draw[->] (a01)-- node[above] {} (a02);
	\draw[->] (a02)-- node[above] {} (a03);
	\draw[->] (a02)-- node[above] {} (a031);

	\draw[->] (a031)-- node[above] {} (a04);

	\draw[->] (a00)-- node[left] {} (a10);
	\draw[->] (a01)-- node[left] {} (a11);
	\draw[->] (a02)-- node[left] {} (a12);
	\draw[->] (a031)-- node[left] {} (a131);
	\draw[->] (a04)-- node[left] {} (a14);

	\draw[->] (a10)-- node[below] {} (a11);
	\draw[->] (a11)-- node[below] {} (a12);
	\draw[->] (a12)-- node[below] {} (a13);
	\draw[->] (a13)-- node[below] {} (a14);
	\draw[->] (a12)-- node[below] {} (a131);
	\draw[->] (a131)-- node[below] {} (a14);

	\draw[->, white, line width = 10pt] (a03)--  (a04);
	\draw[->] (a03)--  (a04);
	
	\draw[->, white, line width = 10pt] (a03)-- node[left] {} (a13);
	\draw[->] (a03)-- node[left] {} (a13);

\end{tikzpicture}
$$
One sees that this surface is basically composed of diagrams (\ref{equa_xxqq1}), (\ref{equa_xxqq2}) and (\ref{equa_uuu}) (and one $\lambda$ and $\lambda^{-1}$) and thus the composition is just the identity on $\catC\times\catC \xto{F\times F} \catV \times \catV \xto\tp \catV$.

In the second part of our lemma, we suppose that we have a braiding:\footnote{$T$ denotes the functor switching the two arguments.}
$$
\begin{tikzpicture}[label distance =-0.5 mm]	
	\node(lx) at (2.1,0){};
	\node(ly) at (0, -2){};
	\pgfresetboundingbox;

	\node(a00) {$\catC \times \catC $};
	\node(a01) at ($(a00)+(lx)$) {$\catC$};
	\node(a10) at ($(a00)+(ly)$) {$\catC \times \catC$};
	
	\node(stred) at ($(a00)+0.25*(lx)+0.35*(ly)$) [inner sep = 0pt, shape=rectangle, rotate=-90, label=100:{$\cong$}] {{\Huge $\Rightarrow$}};

	\draw[->] (a00)-- node[above] {$\otimes$} (a01);
	\draw[->] (a00)-- node[left] {$T$} (a10);
	\draw[<-] (a01)-- node[right] {$\otimes$} (a10);
\end{tikzpicture}
$$
The composition of natural isomorphisms on the faces of
$$
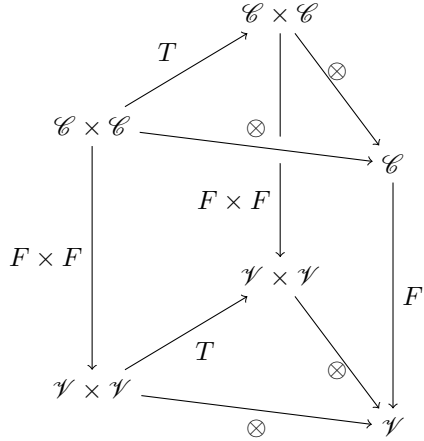
\begin{tikzpicture}[label distance =-0.5 mm]	
	\node(lx) at (4,-0.5){};
	\node(ly) at (2.5, 1.5){};
	\node(lz) at (0, 3.5){};
	\pgfresetboundingbox;

	\node(a000) {$\catV \times \catV$};
	\node(a100) at ($(a000)+(lx)$) {$\catV$};
	\node(a010) at ($(a000)+(ly)$) {$\catV \times\catV$};
	
	\node(a001) at ($(a000)+(lz)$)  {$\catC \times \catC $};
	\node(a101) at ($(a001)+(lx)$) {$\catC $};
	\node(a011) at ($(a001)+(ly)$) {$\catC \times\catC$};

	\draw[->] (a000)-- node[below right] {$T$} (a010);
	\draw[->] (a000)-- node[below] {$\otimes$} (a100);
	\draw[->] (a010)-- node[below] {$\otimes$} (a100);
	
	\draw[->] (a001)-- node[above left] {$T$} (a011);
	\draw[->] (a011)-- node[above] {$\otimes$} (a101);

	\draw[->] (a001)-- node[left] {$F\times F$} (a000);
	\draw[->] (a101)-- node[right] {$F$} (a100);	
	\draw[->] (a011)-- node[left, near end] {$F\times F$} (a010);
	
	\draw[->, white, line width = 10pt] (a001)-- (a101);	
	\draw[->] (a001)-- node[above] {$\otimes$} (a101);
\end{tikzpicture}
$$
provides us with an automorphism of $\catC\times \catC\xto{F\times F} \catV\times\catV\xto{\otimes}\catV$ that corresponds to an $R$-matrix $R\in \big( \End(F) \big)^{\otimes 2}$. This turns our quasi-bialgebra into a quasi-triangular one.
\end{proof}

\begin{lem}\label{lem_functoriality}
The assignment from the previous lemma:
\begin{align*}
\set{\text{strong quasi-monoidal functors to $\catV$ s.t. (\ref{map_Theta}) is iso}} & \to \set{\text{quasi-bialgebras}} \\
\big( \catC \xto{F} \catV \big) &\mapsto \End(F)
\end{align*}
is functorial in the sense that a strong monoidal (not quasi-monoidal) functor $G$ s.t. the diagram of quasi-monoidal functors
$$
\begin{tikzpicture}
	\node(lx) at (2.1,0){};
	\node(ly) at (0, -1.3){};
	\pgfresetboundingbox;

	\node(a00) {$\catC_1$};
	\node(a01) at ($(a00)+(lx)$) {$\catC_2$};
	\node(a10) at ($(a00)+0.5*(lx)+(ly)$) {$\catV$};
	
	\draw[->] (a00)-- node[above] {$G$} (a01);
	\draw[->] (a00)-- node[left] {$F_1$} (a10);
	\draw[->] (a01)-- node[right] {$F_2$} (a10);
\end{tikzpicture}
$$
commutes, induces a quasi-bialgebra morphism $\End(F_2) \xto{G^*} \End(F_1)$.
\end{lem}
\begin{proof}
By the equality between quasi-monoidal functors $F_1=F_2\circ G$ we mean also that the quasi-monoidal structure on the functor $F_1=F_2\circ G$ is given by composition of the quasi-monoidal structures that is  by the natural transformation
$$
\begin{tikzpicture}[label distance =-0.5 mm]	
	\node(lx) at (2.1,0){};
	\node(ly) at (0, -2){};
	\pgfresetboundingbox;

	\node(a00) {$\catC_1 \times \catC_1 $};
	\node(a01) at ($(a00)+(lx)$) {$\catC_1$};
	\node(a10) at ($(a00)+(ly)$) {$\catC_2 \times \catC_2$};
	\node(a11) at ($(a00)+(lx)+(ly)$) {$\catC_2$};	
	\node(a20) at ($(a00)+2*(ly)$) {$\catV \times \catV$};
	\node(a21) at ($(a00)+(lx)+2*(ly)$) {$\catV$};	
	
	\node(stred1) at ($(a00)+0.5*(lx)+0.5*(ly)$) [inner sep = 0pt, shape=rectangle, rotate=30, label=above:$\mu$] {{\Huge $\Rightarrow$}};	
	\node(stred2) at ($(a10)+0.5*(lx)+0.5*(ly)$) [inner sep = 0pt, shape=rectangle, rotate=30, label=above:$\lambda$] {{\Huge $\Rightarrow$}};	
	
	\draw[->] (a00)-- node[above] {$\otimes$} (a01);
	\draw[->] (a10)-- node[below] {$\otimes$} (a11);
	\draw[->] (a20)-- node[below] {$\otimes$} (a21);

	\draw[->] (a00)-- node[left] {$G \times G$} (a10);
	\draw[->] (a01)-- node[right] {$G$} (a11);
	\draw[->] (a10)-- node[left] {$F_2 \times F_2$} (a20);
	\draw[->] (a11)-- node[right] {$F_2$} (a21);
	
\end{tikzpicture}
$$
Now the equality $\Delta (G^* \alpha) = (G^* \otimes G^*)(\Delta\alpha)$ corresponds to the obvious equality of the two natural transformations:
$$
\begin{tikzpicture}[baseline=-2cm]	
	\node(lx) at (1.6,0){};
	\node(ly) at (0, -2){};
	\pgfresetboundingbox;

	\node(a00) {$\catC_1 \times \catC_1$};

	\node(a10) at ($(a00)+1*(ly)-1.5*(lx)$) {$\catC_2\times \catC_2$};
	\node(a11) at ($(a00)+1*(ly)$) {$\catC_1$};	
	\node(a12) at ($(a00)+1*(ly)+1.5*(lx)$) {$\catC_2 \times \catC_2 $};

	\node(a20) at ($(a00)+2*(ly)-1.5*(lx)$) {$\catV\times \catV$};
	\node(a21) at ($(a00)+2*(ly)$) {$\catC_2$};	
	\node(a22) at ($(a00)+2*(ly)+1.5*(lx)$) {$\catV\times \catV$};

	\node(a30) at ($(a00)+3*(ly)$) {$\catV$};
	
	\node(stred1) at ($0.5*(a10)+0.5*(a11)$) [inner sep = 0pt, shape=rectangle, rotate=0, label=above:$\mu$] {{\Huge $\Rightarrow$}};	
	\node(stred2) at ($0.5*(a11)+0.5*(a12)$) [inner sep = 0pt, shape=rectangle, rotate=0, label=above:$\mu^{-1}$] {{\Huge $\Rightarrow$}};
	\node(stred3) at ($0.5*(a20)+0.5*(a21)$) [inner sep = 0pt, shape=rectangle, rotate=0, label=above:$\lambda$] {{\Huge $\Rightarrow$}};	
	\node(stred4) at ($0.5*(a21)+0.5*(a22)$) [inner sep = 0pt, shape=rectangle, rotate=0, label=above:$\lambda^{-1}$] {{\Huge $\Rightarrow$}};
	\node(stred5) at ($0.5*(a21)+0.5*(a30)$) [inner sep = 0pt, shape=rectangle, rotate=0, label=above:$\alpha$] {{\Huge $\Rightarrow$}};
	
	\draw[->] (a00)-- node[left] {$G\times G$} (a10);
	\draw[->] (a00)-- node[right] {$\otimes$} (a11);
	\draw[->] (a00)-- node[right] {$G\times G$} (a12);

	\draw[->] (a10)-- node[right] {$\otimes$} (a21);
	\draw[->] (a12)-- node[left] {$\otimes$} (a21);
	\draw[->] (a10)-- node[left] {$F_2 \times F_2$} (a20);
	\draw[->] (a12)-- node[right] {$F_2 \times F_2$} (a22);
	\draw[->] (a11)-- node[right] {$G$} (a21);

	\draw[->] (a20)-- node[left,below] {$\otimes$} (a30);
	\draw[->] (a21)to[bend right=45] node[left, near start] {$F_2$} (a30);
	\draw[->] (a21)to[bend left=45] node[right, near start] {$F_2$} (a30);
	\draw[->] (a22)-- node[right, below] {$\otimes$} (a30);

\end{tikzpicture}
=
\begin{tikzpicture}[baseline=-2cm]	
	\node(lx) at (1.5,0){};
	\node(ly) at (0, -2){};
	\pgfresetboundingbox;

	\node(a00) {$\catC_1 \times \catC_1$};

	\node(a10) at ($(a00)+1*(ly)$) {$\catC_2\times \catC_2$};
	
	\node(a20) at ($(a00)+2*(ly)-1.5*(lx)$) {$\catV\times \catV$};
	\node(a21) at ($(a00)+2*(ly)$) {$\catC_2$};	
	\node(a22) at ($(a00)+2*(ly)+1.5*(lx)$) {$\catV\times \catV$};

	\node(a30) at ($(a00)+3*(ly)$) {$\catV$};
	
	\node(stred3) at ($0.5*(a20)+0.5*(a21)$) [inner sep = 0pt, shape=rectangle, rotate=0, label=above:$\lambda$] {{\Huge $\Rightarrow$}};	
	\node(stred4) at ($0.5*(a21)+0.5*(a22)$) [inner sep = 0pt, shape=rectangle, rotate=0, label=above:$\lambda^{-1}$] {{\Huge $\Rightarrow$}};
	\node(stred5) at ($0.5*(a21)+0.5*(a30)$) [inner sep = 0pt, shape=rectangle, rotate=0, label=above:$\alpha$] {{\Huge $\Rightarrow$}};
	
	\draw[->] (a00)-- node[left] {$G\times G$} (a10);
	
	\draw[->] (a10)-- node[right] {$\otimes$} (a21);
	\draw[->] (a10)-- node[left] {$F_2 \times F_2$} (a20);
	\draw[->] (a10)-- node[right] {$F_2 \times F_2$} (a22);

	\draw[->] (a20)-- node[left,below] {$\otimes$} (a30);
	\draw[->] (a21)to[bend right=45] node[left, near start] {$F_2$} (a30);
	\draw[->] (a21)to[bend left=45] node[right, near start] {$F_2$} (a30);
	\draw[->] (a22)-- node[right, below] {$\otimes$} (a30);

\end{tikzpicture}
$$

In order to show  $\Phi_1= (G^* \otimes G^* \otimes G^*)(\Phi_2)$ one draws the diagram
$$
\begin{tikzpicture}[label distance =-0.5 mm]	
	\node(lx) at (4,0){};
	\node(ly) at (2.5, 1.5){};
	\node(lz) at (0, 4){};
	\pgfresetboundingbox;

	\node(a000) {$\catV \times \catV\times \catV $};
	\node(a100) at ($(a000)+(lx)$) {$\catV \times\catV$};
	\node(a010) at ($(a000)+(ly)$) {$\catV \times\catV$};
	\node(a110) at ($(a000)+(lx)+(ly)$) {$\catV$};	
	
	\node(a001) at ($(a000)+(lz)$)  {$\catC_1 \times \catC_1\times \catC_1 $};
	\node(a101) at ($(a001)+(lx)$) {$\catC_1 \times\catC_1$};
	\node(a011) at ($(a001)+(ly)$) {$\catC_1 \times\catC_1$};
	\node(a111) at ($(a001)+(lx)+(ly)$) {$\catC_1$};	

	\node(a002) at ($(a001)+(lz)$)  {$\catC_2 \times \catC_2\times \catC_2 $};
	\node(a102) at ($(a002)+(lx)$) {$\catC_2 \times\catC_2$};
	\node(a012) at ($(a002)+(ly)$) {$\catC_2 \times\catC_2$};
	\node(a112) at ($(a002)+(lx)+(ly)$) {$\catC_2$};

	\draw[->] (a000)-- node[below right] {$\id \times \otimes$} (a010);
	\draw[->] (a000)-- node[below] {$\otimes \times \id$} (a100);
	\draw[->] (a010)-- node[below] {$\otimes$} (a110);
	\draw[->] (a100)-- node[below] {$\otimes$} (a110);

	\draw[->] (a001)-- node[left] {$F_2\times F_2\times F_2$} (a000);
	\draw[->, white, line width = 10pt] (a101)-- (a100);
	\draw[->] (a101)-- node[right, near start] {$F_2\times F_2$} (a100);	
	\draw[->] (a011)-- node[left, near end] {$F_2\times F_2$} (a010);
	\draw[->] (a111)-- node[right] {$F_2$} (a110);

	\draw[->] (a001)-- node[above left] {$\id \times \otimes$} (a011);
	\draw[->] (a011)-- node[above] {$\otimes$} (a111);
	\draw[->] (a101)-- node[above] {$\otimes$} (a111);
	\draw[->, white, line width = 10pt] (a001)-- (a101);	
	\draw[->] (a001)-- node[above, fill=white] {$\otimes \times \id$} (a101);

	\draw[->] (a002)-- node[left] {$G\times G\times G$} (a001);
	\draw[->, white, line width = 10pt] (a102)-- (a101);
	\draw[->] (a102)-- node[right, near start] {$G\times G$} (a101);	
	\draw[->] (a012)-- node[left, near end] {$G\times G$} (a011);
	\draw[->] (a112)-- node[right] {$G$} (a111);

	\draw[->] (a002)-- node[above left] {$\id \times \otimes$} (a012);
	\draw[->] (a012)-- node[above] {$\otimes$} (a112);
	\draw[->] (a102)-- node[above] {$\otimes$} (a112);
	\draw[->, white, line width = 10pt] (a002)-- (a102);	
	\draw[->] (a002)-- node[above, fill=white] {$\otimes \times \id$} (a102);
\end{tikzpicture}
$$
Now $\Phi_1$ corresponds to the composition of the natural transformations on the outer faces of the union of the two cubes and $\Phi_2$ to the composition of the faces of the bottom one. The monoidality of $G$ means that the composition of the faces of the upper cube is the identity. So one gets precisely $\Phi_1= (G^* \otimes G^* \otimes G^*)(\Phi_2)$.
\end{proof}
\begin{rem}
There is a neat way to describe the quasi-bialgebra structure on the algebra $A:=\End(F)$ in Lemma \ref{lem_quasibi_from_functor}.  Since $\End(F)$ acts on each $FM$, $M\in\ob(\catC)$, we can lift $F$ to a functor $\widetilde F : \catC \to \LM A$ to the category of $A$-modules. Any quasi-bialgebra structure on $A$ makes $\LM A$ into a monoidal category. The one from Lemma \ref{lem_quasibi_from_functor} is the one that ensures that $\widetilde F$ becomes a monoidal functor (with respect to the monoidal structure given by the same maps that define the quasi-monoidal structure of $F$).

For the proof one just realizes that $A=\End(\LM A \xto\forg \catV)$ and that the monoidality of $\widetilde F$ implies that the bijection 
$$A=\End(\LM A \xto\forg \catV) \xto{\widetilde F^*} \End(F) $$
is a homomorphism of quasi-bialgebras by Lemma \ref{lem_functoriality}. 
\end{rem}
\begin{lem}
Let $\catC$ be a category enriched over $\catV$ and let $C\in \ob(\catC)$. 
Then the functor represented by $C$
$$F: \catC \to \catV ,\, Y\mapsto \catC (C,Y) $$
satisfies that $\Theta$ in (\ref{map_Theta}) is an isomorphism.
\end{lem}
\begin{proof} To keep the notations simple we show that $\Theta$ is an isomorphism for $k=2$.
Denote by $\catC \otimes \catC$ the category whose objects and morphism are
\begin{align*}
\ob(\catC \otimes \catC)  & := \ob(\catC) \times \ob(\catC) \\
\catC \otimes \catC \big( (X_1,X_2), (Y_1,Y_2) \big)&:= \catC(X_1 , Y_1) \otimes \catC(X_2 , Y_2 )
\end{align*}
with the obvious compositions. One has a functor $\catC\times\catC \xto{J} \catC\otimes \catC$. We further denote by 
$$G:\catC\otimes\catC \to \catV:(Y_1,Y_2) \mapsto \catC \otimes \catC \big( (C,C), (Y_1,Y_2) \big)$$
the functor represented by $(C,C)$. Then the diagram
$$
\begin{tikzpicture}[label distance =-0.5 mm]	
	\node(lx) at (2.1,0){};
	\node(ly) at (0, -1.5){};
	\pgfresetboundingbox;

	\node(a00) {$\catC \times \catC $};
	\node(a01) at ($(a00)+(lx)$) {$\catC \otimes \catC$};
	\node(a10) at ($(a00)+(ly)$) {$\catV \times \catV$};
	\node(a11) at ($(a00)+(lx)+(ly)$) {$\catV$};	
	
	\draw[->] (a00)-- node[above] {$J$} (a01);
	\draw[->] (a00)-- node[left] {$F \times F$} (a10);
	\draw[->] (a01)-- node[right] {$G$} (a11);
	\draw[->] (a10)-- node[above] {$\otimes$} (a11);
\end{tikzpicture}
$$
is strictly commutative so one has an equality
$$\End \big( \catC \otimes \catC \xto{F\times F} \catV \times \catV \xto{\otimes} \catV \big) = \End \big( \catC\times \catC \xto{J} \catC \otimes \catC \xto{G} \catV\big) \ .$$
Using this we can factorize $\Theta$ as a composition:
$$
\begin{tikzpicture}[label distance =-0.5 mm]	
	\node(lx) at (5,0){};
	\node(ly) at (0, -1.5){};
	\pgfresetboundingbox;

	\node(a00) {$\big( \End( \catC \xto{F} \catV ) \big)^{\otimes 2}$};
	\node(a01) at ($(a00)+(lx)$) {$ \End \big( \catC\times \catC \xto{J} \catC \otimes \catC \xto{G} \catV\big) $};
	\node(a11) at ($(a00)+(lx)+(ly)$) {$\End \big( \catC \otimes \catC \xto{G} \catV\big)$};	
	
	\draw[->] (a00)-- node[above] {$\Theta$} (a01);
	\draw[->] (a00)-- node[left] {} (a11);
	\draw[<-] (a01)-- node[right] {$J^*$} (a11);
\end{tikzpicture}
$$
It is easy to see that the vertical map is an iso. By Yoneda we have 
$$\big( \End( \catC \xto{F} \catV ) \big)^{\otimes 2} = \big( \catC(C,C) \big)^{\otimes 2}$$ and also
$$\End \big( \catC \otimes \catC \xto{G} \catV\big) = \catC\otimes\catC \big( (C,C),\  (C,C) \big)  = \big( \catC(C,C) \big)^{\otimes 2}$$
so we see that the oblique map is also an iso.
\end{proof}

\begin{lem}\label{lem_twists}
Let $\catC \xto{F} \catV$ be a functor satisfying the condition of Lemma \ref{lem_quasibi_from_functor}. Let $(\lambda_0, \lambda)$ and $(\lambda_0, \lambda')$ be two strong quasi-monoidal structures on $F$ with the same $\lambda_0$. Then $\lambda, \lambda': \otimes\circ (F\times F) \xto\cong F\circ\otimes   $ differ by an invertible element 
$$J \in \End(\otimes\circ (F\times F)) = \End(F)\otimes \End(F).$$ 
and the two quasi-bialgebra structures one gets on $\End(F)$ differ by the twist by $J$.  
\end{lem}

\begin{rem}\label{rem_twists}
Actually, we will use the lemma (\ref{lem_twists}) in a slightly different context. We will have two quasi-monoidal functors $(F,\lambda), (G,\mu) :\catC\to\catV$ and a natural isomorphism $\alpha: F\to G$, that will not be monoidal. The two quasi-bialgebras $\End(F,\lambda)$ and $\End(G,\mu)$ can be considered to be the same  as algebras thanks to $\alpha$, but their quasi-bialgebra structures differ by a twist by the element\footnote{This map measures the failure of $\alpha$ to be monoidal.}
$$ FX\otimes FY \xto{\alpha_X \otimes \alpha_Y} GX\otimes GY \xto{\mu_{X,Y}} G(X\otimes Y) \xto{\alpha^{-1}_{X\otimes Y}} F(X\otimes Y) \xto{\lambda^{-1}_{X,Y}} FX\otimes FY$$
in  $\End(\otimes \circ (F\times F)) = \End(F)\otimes\End(F)$.
\end{rem}

\section{Modules in braided monoidal categories}
\subsection{Category of $A$-modules}\label{subsection_Mod_A}
In this subsection, $\catC$ will denote a preabelian braided monoidal category satisfying that for any $X\in \ob(\catC)$ the functor $\argument \otimes X: \catC \to \catC$ preserves finite colimits. We also fix a commutative (w.r.t. the braiding in $\catC$) algebra (i.e. a monoid) $A$ in $\catC$.
\begin{notation}
If $\catC$ is a monoidal category and $A$ an algebra in $\catC$, we denote by $\LM{A}(\catC)$ and $\RM{A}(\catC)$ the categories of left and right $A$-modules in $\catC$.
\end{notation}

\begin{defn}\label{defn_tensor_1} Let $M$ be a right and $N$ a left $A$-module in $\catC$. We define $M \otimes_A N$ as an object in $\catC$ together with an $A$-bilinear map $M \tp N \xto{} M \tp_A N$ that is universal in the usual sense. 
\end{defn}
Since $\catC$ is preabelian, we can reformulate this as follows:
\begin{defn}\label{defn_tensor_2} Let $M\tp A \xto{\mu} M$ and $A \tp N \xto{\nu} N$ be the $A$-module structures on $M$ and $N$. $M\tp_A N$ is the cokernel of
\begin{align*}
M\tp A \tp N \xto{ (\id \tp \nu) - (\mu \tp \id)} M\tp N .
\end{align*}
\end{defn}
\begin{prop}
If $M\in \ob(\RM{A})$, $N\in \ob(\LM A)$, $X\in \ob(\catC)$ then
$$ M\otimes_A(N \otimes X) \cong (M\otimes_A N)\otimes X$$
\end{prop}
\begin{proof}
Just use that $\argument \otimes X$ preserves cokernels and the definition \ref{defn_tensor_2}.
\end{proof}
\begin{cor}\label{cor_tensor}
Assume that on $N$ we have also a right $A$-module structure, commuting with the left one. Then one can define a right $A$-module structure on $M\tp_A N$ by the following composition:
$$ (M \tp_A N) \tp A \cong M \tp_A (N \tp A) \to M\tp_A N\ .    $$
\end{cor}
If $N$ is just a right $A$-module, we can use the commutativity of $A$ and the braiding in $\catC$ to give it a commuting left $A$-module structure:
$$\input{\cesta/obr2}:= \input{\cesta/obr3}$$
This together with the corollary \ref{cor_tensor} enables us to see $\tp_A$ as a monoidal structure on $\RM A$.

\begin{prop}\label{prop_free_is_monoidal}
The functor $\argument \otimes A: \catC \to \RM A$ is strong monoidal.
\end{prop}
\begin{proof}
We want to identify $(M\otimes A)\otimes_A (N\otimes A)$  with $ (M\otimes N)\otimes A$. For that we have just to specify a universal $A$-bilinear $\catC$-morphism
$$(M\otimes A)\otimes (N\otimes A) \to (M\otimes N)\otimes A\ .$$
We take this:
$$\input{\cesta/obr4}$$
\end{proof}

Recall that we assumed $\catC$ preabelian, braided monoidal category, s.t.  the tensor product with any fixed object of $\catC$ preserves finite colimits. One can ask whether $\RM A$ is also like that.
\begin{prop} \label{properties_of_A_mods}
$\RM A$ is a preabelian monoidal category satisfying that the functor $\argument \tp_A M: \RM A \to \RM A$ preserves finite colimits.

 If $\catC$ was symmetric-monoidal, $\RM A$ is also symmetric-monoidal. 

If $\catC$ was a $\catTV$-category (definition \ref{defn_TV-category}) or a $\catTVkh$-category (the subsection \ref{subsection_TVkh-categories}) then $\RM A$ also stays like that.
\end{prop}
\begin{proof}
Obvious.
\end{proof}

\subsection{Category of free $A$-modules}\label{subsection_FreeMod}
Let now $\catC$ be a braided monoidal category (we don't impose the conditions of Subsection \ref{subsection_Mod_A}) and $A$ a commutative algebra in it. 

We take $\ob(\RFree A) := \ob(\catC)$ but we will denote the objects in $\RFree A$ as $X\tp A$ for $X\in \ob(\catC)$.
The hom sets are
$$ \RFree A (X\tp A, Y\tp A) :=  \catC (  X, Y\tp A  )$$
with the composition 
\begin{equation} \label{equa_comp_freemod}
 \input{\cesta/obr5} \circ \input{\cesta/obr6} := \input{\cesta/obr7} .
\end{equation}

The monoidal structure is denoted again by $\tp_A$ and is defined on the objects as $(X_1\tp A) \tp_A (X_2 \tp A) := (X_1 \tp X_2) \tp A$ and on the morphisms as
\begin{equation} \label{equa_tp_freemod}
 \input{\cesta/obr8} \tp_A \input{\cesta/obr9} := \input{\cesta/obr10} 
\end{equation}

\begin{rem}
If we can form also $\RM A$ then we have an obvious fully-faithful, strong monoidal functor $\RFree A \to \RM A $. However, an object $X\tp A$ in $\RFree A$ carries more information than the corresponding $X\tp A$ in $\RM A$, namely it remembers the object $X\in \ob(\catC)$.
\end{rem}

\section{Represented functors}\label{section_represented_functors}
\subsection{Karoubi envelope}\label{subsection_Karoubi}
\newcommand{\YXp}{\Yoneda_{(X,p)}}
\begin{defn}
Let $\catC$ be a category. We define a new category
$\Split(\catC)$ (called the \emph{Karoubi envelope} of $\catC$) whose objects are pairs $(X,p)$ where $X\in \ob(\catC)$ and $p\in \catC(X,X)$ satisfying $p^2 = p$. Morphisms $f:(X,p) \to (Y,q) $ in $\Split(\catC)$ are morphisms $X\xto f Y$ in $\catC$ satisfying $q\circ f\circ p = f$. If $\catC$ is a monoidal category then $\Split(\catC)$ will be again a monoidal category: 
$$  (X,p) \tp (Y,q) := (X\tp Y, p\tp q) .  $$ 

\end{defn}
There is a canonical fully faithful functor $\catC \to \Split(\catC): X \mapsto (X, \id_X)$.

Ideologically, one regards an object $(X,p)$ of $\Split(\catC)$ as the image $\Ima(p) \subset X$, with a complement $\Ker(p)$ (that is $X = \Ima(p) \oplus \Ker(p)$) even if kernels and images don't make sense in $\catC$. Thus we occasionally call objects of $\Split(\catC)$ ``direct summands in $\catC$''.

\begin{defn}
We define the functor represented by a direct summand $(X,p)$ as
$$ \YXp :\;\; \catC \to \Set : \;\;Y \mapsto  \Hom_{\Split(\catC)} \big( (X,p), Y \big) = \set{ f \in \catC( X, Y) \;|\; f\circ p = f } \ . $$ 
If $\catC$ is a $\catTV$-category (or $\catTVkh$-category), we get a functor to $\catTV$ (or $\catTVkh$).
\end{defn}

\begin{prop}
$$\End(\YXp) = \End_{\Split(\catC)}\big((X,p) \big) = \set{ f \in \catC( X, X) \;|\; p \circ f \circ p = f } . $$
\end{prop}
\begin{proof}
Obvious.
\end{proof}

\subsection{Quasi-monoidal structure on a represented functor}\label{subsec_monoidal_forgetful}
\begin{defn}
A \emph{quasi-coalgebra} in $\catC$ is an object $C$ together with two morphisms $\Delta: C \to C\tp C$, $\eps: C\to I$ satisfying that the two compositions
$$ C\xto\Delta C\tp C \xto{\id\tp \eps} C \quad\text{and}\quad  C\xto\Delta C\tp C \xto{\eps\tp\id} C$$
are equal to $\id_C$. Thus a quasi-coalgebra is a coalgebra that is not necessarily co-associative. 

\end{defn}
If we assume that $\catC$ is enriched over vector spaces then the functor $\Yoneda_C$ represented by $C$ will get a quasi-monoidal structure in the following way:  the morphism $\Yoneda_C(X)\tp \Yoneda_C(Y) \to \Yoneda_C(X\tp Y)$ is given by the composition
$$ \catC(C,X) \tp \catC(C,Y) \xto\tp \catC( C\tp C, X\tp Y ) \xto{\Delta^*} \catC(C, X\tp Y)$$
and the morphism $\field \to \Yoneda_C(I)$ is 
$$ \field \xto{ 1 \mapsto \id_I} \catC(I,I) \xto{\eps^*} \catC(C, I) .$$

Consequently, if $(X,p)$ is a quasi-coalgebra in $\Split(\catC)$, the functor represented by it will be quasi-monoidal. Explicitly, if we have morphisms $\Delta_X:X\to X\tp X$, $\eps_X: X \to I$ satisfying
$$( p\tp p )\circ  \Delta_X \circ p = \Delta_X;\quad \eps_X \circ p = \eps_X \quad \text{and}$$
$$ (\eps_X \tp \id_X)\circ \Delta_X = (\id_X \tp \eps_X)\circ \Delta_X = p $$
then we get a natural quasi-monoidal structure on the functor represented by $(X,p)$.
\begin{rem}
Everything generalizes to $\catTV$-categories and $\catTVkh$-categories.
\end{rem}

\section{Our categories}\label{section_our_categories}

\subsection{Categories $\catH$ and $\catP$}\label{subsection_P_H}
Denote by $\catH$ and $\catP$ the categories of complete equicontinuous $\g$ and $\p$ modules. The product $\ctp$ makes them into symmetric monoidal categories. We equip the hom-sets with strong topologies and regard $\catH$ and $\catP$ as $\catTV$-categories.
 
\subsection{Functor $\srdieckof:\catH \to \catP$}\label{subsection_srdiecko}
Given $M\in \catH$, define the topological vector space $\srdieckof M:= \Hom^C_\Ug(\Up, M)$. We define a left $\Up$-module structure on $\srdieckof M$ by
$$ (z\act s)(x): = s(xz) \quad \text{for}\quad z,x\in \Up,\; s\in \srdieckof M = \Hom^C_\Ug(\Up, M) .$$

In fact, $\srdieckof M$ is complete (Proposition \ref{prop_srdiecko_complete}) and an equicontinuous $\p$-module (corollary \ref{cor_srdiecko_equicont}) and thus we get a functor $\srdieckof:\catH \to \catP$.

\begin{notation}
Define an algebra $\cSg$ in $\catCTV$ as the following limit of discrete commutative algebras:
$$ \cSg := \varprojlim_k   \big( (S\g) /(S^{> k} \g) \big) \ .$$
In other words, $\cSg$ is the algebra of formal power series on $\g^*$.
\end{notation}

\begin{prop}\label{prop_srdiecko_complete}
There is a natural isomorphism of topological vector spaces $\srdieckof M \cong M\ctp \cSg$ . In particular, $\srdieckof M$ is complete.
\end{prop}
\begin{proof}
We have the following maps:
$$ \srdiecko M = \Hom_\Ug^C(\Up, M) \xto{(1)} \Hom_\field^C(S\g^*, M) \xfrom{(2)} M \ctp  (S\g^*)^* \xlongequal{(3)} M\ctp \cSg $$
(1) is the restriction map that is a topological isomorphism by the claim \ref{claim_fhitr}. (2) and (3) are also topological isomorphisms by claims \ref{claim_ch} and \ref{claim_cSg}.
\end{proof}

\begin{prop}
The functor $\srdieckof$ has a monoidal structure
$$\field \to \srdiecko \field;\quad\srdieckof M \ctp \srdieckof N \to \srdieckof(M\ctp N)  $$
given by $ \field\to \Hom^C_\Ug(\Up, M): 1\mapsto \eps_\Up $ and
\begin{align*}
\Hom^C_\Ug( \Up, M) \ctp \Hom^C_\Ug(\Up,N) &\to \Hom^C_\Ug( \Up, M\ctp N):\\
(\Up \xto{s} M) \otimes (\Up \xto{t} N) &\mapsto (\Up \xto{\Delta} \Up\otimes \Up \xto{s\otimes t} M\otimes N \to M\ctp N) \ .
\end{align*}
\end{prop}
\begin{proof}
Since all the maps in the composition
$$\Up \xto{\Delta} \Up\otimes \Up \xto{s\otimes t} M\otimes N \to M\ctp N $$
are $\Ug$-linear and continuous, their composition also is. Thus we really get an element in $\srdieckof(M\ctp N)$.

One has also to check that the resulting map $\srdieckof M \ctp \srdieckof N \to \srdieckof(M\ctp N)  $ is $\Up$-linear (that is easy) and continuous w.r.t. the strong topologies (it is, because $\catTV$ is a $\catTV$-category and $\Up \xto{\Delta} \Up\otimes \Up$ and  $M\otimes N \to M\ctp N$ are continuous). 

Coassociativity of $\Up \xto{\Delta} \Up\otimes \Up$ implies that this structure is monoidal, not just quasi-monoidal.
\end{proof}

\subsection{Algebra $A$.}\label{subsection_A}
One can regard $\field$ as a commutative algebra in $\catH$.
Since $\srdieckof$ is a monoidal functor, $A:=\srdieckof \field$ will be a commutative algebra in $\catP$. Explicitly, the product of $f,g\in A = \Hom^C_\Ug(\Up, \field)$ is given by 
$$(f\cdot g)(z) = \sum f(z_{(1)}) \cdot  g(z_{(2)})  \quad\text{for } z\in\Up$$
and the unit is just the augmentation on the universal enveloping algebra $\eps_\Up :\Up \to \field$.

\subsection{Natural epimorphism $\pi_M : \srdiecko M \to M$ in $\catH$}\label{subsec_pi} For $M\in \catH$ we define
$$ \pi_M:\; \srdieckof M = \Hom_\Ug^C(\Up,M) \ \to\  M\; :\; s\mapsto s(1) \ .$$
It is a natural transformation between $\catH \xto\srdiecko \catP \xto\forg \catH$ and $\catH\xto\id\catH$.

\subsection{Augmentation $\eps_A :=\pi_\field : A \to \field$ in $\catH$}\label{subsection_eps_A}
Explicitly
$$ \eps_A:\; A = \Hom_\Ug^C(\Up,\field) \ \to\  \field\; :\; s\mapsto s(1) \ .$$
It is easy to see, that $\eps_A$ is an augmentation on the algebra $A$ in $\catH$ (but it is not a morphism in  $\catP$). On pictures, we denote $\eps_A$ by $$.

\subsection{Categories $\catA$ and $\catF$.}\label{subsection_A_F}
Denote by $\catA$ the category of all right $A$-modules and by $\catF$ the category of free right $A$-modules in $\catP$ (in the sense of Subsection \ref{subsection_FreeMod}). They are symmetric monoidal categories with the monoidal product $\otimes_A$.

\subsection{Functor $\srdcof$}\label{subsection_srdco}
Since any $M\in \ob(\catH)$ is a right module over $\field$ in $\catH$, $\srdieckof M$ will be a right module over $A$. 
Explicitly, if $f\in A = \Hom^C_\Ug(\Up, \field)$ and $s\in \srdieckof M = \Hom^C_\Ug(\Up, M)$ we define $s\coact f \in \srdieckof M$ by  
$$( s \coact f)(z) = \sum s(z_{(1)}) \cdot  f(z_{(2)})  \quad\text{for } z\in\Up\ .$$
We can thus regard $\srdieckof$ as a functor:
$$\srdcof: \catH \to \catA\ . $$

\begin{rem}\label{rem_forgetted_B}
From the proposition \ref{prop_srdiecko_complete} we have an isomorphism topological vector spaces $ A\cong \cSg$.
Since this isomorphism is essentially the morphism 
$$\Hom^C_\Ug(\Up, \field) \to  \Hom^C_\field(S\g,\field)$$
induced by a coalgebra homomorphism $S\g\to \Up = \Ug\otimes S\g$, we see that $ A\cong \cSg$ is actually an isomorphism of algebras.

Similarly, if we forget an $A$-module $\srdiecko M\in \ob(\catP)$  into $\catCTV$, we get an $A=\cSg$-module in $\catCTV$ and the isomorphism $\srdieckof M \cong M\ctp \cSg$ from the proposition \ref{prop_srdiecko_complete}
is an isomorphism of $\cSg$-modules.

\end{rem}

Since the monoidal structure $\srdieckof M \ctp \srdieckof N \to \srdieckof(M\ctp N) $ of the functor $\srdieckof$ is given by an $A$-bilinear morphism in $\catP$, it induces a (strong as we will soon see) monoidal structure  $\srdcof M \otimes_A \srdcof N \to \srdcof(M\ctp N)$ on $\srdco$.

\subsection{The main technical theorem}
\begin{thm}\label{fundamental_thm}
\begin{enumerate}
\item $\srdco: \catH \to \catA$ is a strong monoidal and fully faithful functor.
\item If $M\in \ob(\catP)$, there is a natural $\catA$-isomorphism $\srdco M \cong M\ctp A$. That is, we have a natural monoidal isomorphism:
$$
\begin{tikzpicture}[label distance =-0.5 mm]	
	\node(lx) at (2,0){};
	\node(ly) at (0, 1.5){};
	\pgfresetboundingbox;

	\node(a00) {$\catH$};
	\node(a01) at ($(a00)+(lx)$) {$ \catA $};
	\node(a11) at ($(a00)+0.5*(lx)+(ly)$) {$\catP$};	
	
	\draw[->] (a00)-- node[below] {$\srdco$}  (a01);
	\draw[<-] (a00)-- node[left] {$ \forg$} (a11);
	\draw[<-] (a01)-- node[right] {$\ctp A$} (a11);

	\node at ($(a00)+0.5*(lx)+0.3*(ly)$) {{ \huge $\cong$ }};
\end{tikzpicture}
$$
\item The above statements give us for $M,N\in \ob(\catP)$ an isomorphism
$$ \catH(M,N) \underset{\tiny\cong}{\xto{\srdco}} \catA(\srdco M, \srdco N) \cong \catA(M\ctp A, N\ctp A) \cong \catP(M, N\ctp A). $$
Explicitly, (the inverse of) this isomorphism is 
\begin{equation}\label{iso_homsets}
 \catP(M, N\ctp A) \xto{\cong} \catH(M,N): \input{\cesta/obr11}  \mapsto \input{\cesta/obr12}  
\end{equation}
\end{enumerate}
\end{thm}
The rest of this subsection is devoted to the proof of the theorem \ref{fundamental_thm}.
\begin{rem}
The category $\catCTV$ is not an abelian, just a preabelian category. The fact of having an exact sequence $ A\to B\to C\to 0 $ does not determine the topology on $C$ uniquely. Thus in this section, by saying that such a sequence is (right) exact, we mean that $C$ is a cokernel of the map $A\to B$. Similarly, a functor will be said to be right-exact (to reflect right-exactness) if it preserves (reflects) cokernels. We believe this abuse makes it more readable.
\end{rem}

\begin{proof}[Proof that $\srdco$ is strong monoidal]
By the definition \ref{defn_tensor_2} of the tensor product over $A$ we need to show that the sequence
$$ \srdieckof M \ctp A\ctp \srdieckof N \to\srdieckof M \ctp \srdieckof N \to  \srdieckof(M\ctp N) \to 0 $$
is exact. The forgetful functor $\catP\xto\forg \catCTV$ reflects right-exactness and so it's enough to see the exactness of the following sequence in $\catCTV$ (see the remark \ref{rem_forgetted_B}):
 $$ (M \ctp \cSg)\ctp \cSg\ctp (N\ctp \cSg) \to (M \ctp \cSg)\ctp (N\ctp \cSg) \to M \ctp N\ctp \cSg \to 0 \ . $$
But the last sequence just says that
$$ (M \ctp \cSg )\otimes_\cSg  (N\ctp \cSg)\; \cong\; M \ctp N\ctp \cSg  $$
what comes from the proposition \ref{prop_free_is_monoidal}.
\end{proof}

To prove that $\srdco$ is fully faithful, we introduce a functor $\clubsuit: \catA\to \catH$. Recall that we have an augmentation $\eps_A : \FHP A \to \field$. This turns $\field \in \ob(\catH)$ into an $\FHP A$-module. If $M\in \ob(\catA)$ then $\FHP M$ is also an $\FHP A$-module and so we can define
$$ \clubsuit: \catA\to \catH : M \mapsto  \FHP M  \otimes_{\FHP A} \field\ .$$
\begin{claim}\label{claim_bubzogan}
$\clubsuit \circ \srdco \cong \id_\catH$.
\end{claim} 
\begin{proof}
We will omit the forgetful functor $\FHP$ in the notations. As an $A$-module, $\field = A/\Ker \eps_A$ and thus 
$$\field \otimes_A \srdieckof M = \srdieckof M /\big( \srdieckof M \coact \Ker(\eps_A) \big)\ .$$ 
To see that this is naturally isomorphic to $M$, we show that the sequence in $\catH$:
\begin{equation}\label{equa_seg_nieco0}
\srdieckof M \ctp \Ker \eps_A \xto{\coact} \srdieckof M \xto{\pi_M} M \to 0
\end{equation}
is exact. If we apply to it the functor $\catH\xto\forg \catCTV$, we get the sequence
\begin{equation}\label{equa_seg_nieco1}
( M\ctp \cSg )\ctp \hat S^{>0} \g \xto{\id_M \ctp \mu_\cSg} M\ctp \cSg \to M\to 0 \ .
\end{equation}
The functor $\catH\xto\forg \catCTV$ reflects exactness, so to prove that (\ref{equa_seg_nieco0}) is exact it is enough to show that (\ref{equa_seg_nieco1}) is. But (\ref{equa_seg_nieco1}) is just the image of the obviously exact sequence
$$\cSg \ctp \hat S^{>0} \g \xto{\mu_\cSg} \hat S \g \to \field\to 0 $$
under the right-exact functor $M\ctp\argument$ (see corollary \ref{cor_ctp_and_finite_colims}).
\end{proof}
\begin{proof}[Proof that $\srdco$ is full and faithful] 
The above claim gives us for any $X,Y\in \catH$ a commutative diagram
$$
\begingroup
\newcommand{\vzdy}{1.5}
\newcommand{\vzdx}{4}
\begin{tikzpicture}
	\node(a) {$\catH(X,Y)$};
	\node(b) at ($(a)+(\vzdx,0)$) {$\catA(\srdcof X, \srdcof Y)$};
	\node(c) at ($(a)+(0,-\vzdy)$) {};
	\node(d) at ($(b)+(0,-\vzdy)$) {$\catH(X,Y)$};
	
	\draw[->] (a)-- node[above] {$\tilde\srdieckof$} (b);
	\draw[->] (a)-- node[auto,swap] {$=$} (d);
	\draw[->] (b)-- node[right] {$\budzogan$} (d);	
\end{tikzpicture}
\endgroup
$$
From that we see that $\tilde\srdieckof$ is faithful. In order to see that it is full we show that the map $\catA(\srdcof X, \srdcof Y) \xto{\budzogan} \catH( X, Y)$ in the above diagram is injective. To prove it, we take $\phi\in\catA(\tilde\srdieckof X, \tilde\srdieckof Y)$ s.t. $\budzogan\phi=0$ in $\catH(X,Y)$ and we want to show that then $\phi$ itself must be $0$, i.e. that
$$(\phi s)(\xi)=0, \quad \text{for all } s\in\srdieckof X\ ,\, \xi\in \Up\ . $$ 
And really
\begin{align*}
(\phi s)(\xi) &= \Big( \xi \act (\phi s) \Big) (1) = \Big( \phi( \xi \act s ) \Big) (1) =\\
&= \pi_Y \Big( \phi (\xi \act s ) \Big) = (\budzogan \phi)\Big( \pi_X ( \xi \act s ) \Big) = 0 \ .
\end{align*}
Here the first equality comes from the definition of the action of $\xi\in \Up$ on $\srdieckof X$, the second is $\Up$-linearity of $\phi$, the third is just the definition of $\pi_Y$, the fourth comes from commutativity of\footnote{From the proof of the claim one can see that $\pi_X = \eps_A \otimes_A \id_X$. Thus our square is just the commutative square
$$
\begingroup
\newcommand{\vzdy}{1.5}
\newcommand{\vzdx}{4}
\begin{tikzpicture}
	\node(a) {$A \otimes_A \tilde\srdieckof X$};
	\node(b) at ($(a)+(\vzdx,0)$) {$\field \otimes_A \tilde\srdieckof X$};
	\node(c) at ($(a)+(0,-\vzdy)$) {$A \otimes_A \tilde\srdieckof Y$};
	\node(d) at ($(b)+(0,-\vzdy)$) {$\field \otimes_A \tilde\srdieckof Y$};
	
	\draw[->] (a)-- node[above] {$\eps_A \otimes_A \id_X$} (b);
	\draw[->] (a)-- node[left] {$\id_A\otimes_A \phi$} (c);
	\draw[->] (b)-- node[right] {$ \id_\field \otimes_A \phi$} (d);	
	\draw[->] (c)-- node[below] {$\eps_A \otimes_A \id_Y$} (d);	
\end{tikzpicture}
\endgroup
$$
}
$$
\begingroup
\newcommand{\vzdy}{1.5}
\newcommand{\vzdx}{4}
\begin{tikzpicture}
	\node(a) {$\tilde\srdieckof X$};
	\node(b) at ($(a)+(\vzdx,0)$) {$X$};
	\node(c) at ($(a)+(0,-\vzdy)$) {$\tilde\srdieckof Y$};
	\node(d) at ($(b)+(0,-\vzdy)$) {$Y$};
	
	\draw[->] (a)-- node[above] {$\pi_X$} (b);
	\draw[->] (a)-- node[left] {$ \phi$} (c);
	\draw[->] (b)-- node[right] {$\budzogan \phi$} (d);	
	\draw[->] (c)-- node[below] {$\pi_Y$} (d);	
\end{tikzpicture}
\endgroup
$$
and the last one from our assumption that $\budzogan \phi = 0 $.
\end{proof}

\begin{proof}[Proof that $\srdco M \cong M\ctp A$ for $ M \in \ob(\catP)$]
Assume that an $X\in \ob(\catCTV)$ is equipped with commuting equicontinuous left actions of $\g$ and $\p$ at the same time. Then we can use the $\g$-action to form the space $\Hom^C_\Ug( \Up, X )$ and the $\p$ action to turn it into a left $\p$-module by formula
$$  (z\act f)( \argument) := \sum z_{(1)}\act \big( f(\argument \cdot z_{(2)} ) \big), \quad \text{for } f \in \Hom^C_\Ug(\Up,X);\; z \in \Up.  $$ 
Note that $\Hom^C_\Ug(\Up,X)$ has also a natural right $A$-module structure, so we get $ \Hom^C_\Ug(\Up,X) \in \ob(\catA)$.

An $M\in \ob(\catP)$ can be turned into a $\g\tp\p$-module in two natural ways:
\begin{enumerate}
\item[a)] $\g$-action is the restriction of the original $\p$-action and the new $\p$-action is trivial;
\item[b)] $\g$-action is trivial and the $\p$-action is the original one.
\end{enumerate}  
Denote these two bimodules by $M_a$ and $M_b$. One can see\footnote{The first equality is just the definition of $\srdiecko M$. For the second, we have an $\catA$-morphism $ M \ctp A = M\ctp \Hom^C_\Ug (\Up, \field) \longto  \Hom^C_\Ug (\Up, M_b) $ and it is iso by the proposition \ref{prop_srdiecko_complete}.
} that
$$ \srdco M = \Hom^C_\Ug( \Up, M_a )\quad\text{and}\quad M\ctp A \cong  \Hom^C_\Ug(\Up,M_b) .$$
We want to show that these two objects are isomorphic in $\catA$. Take a map
\begin{align}
\Hom^C_\Ug(\Up,  M_a) &\longto \Hom^C_\Ug(\Up, M_b) \nonumber \\
s &\longmapsto \Bigg( \Up \xto{\Delta} \Up\otimes\Up \xto{S\otimes s} \Up\otimes M \xto{\act} M \Bigg)
\label{map_srdiecko_tensor}
\end{align} 
with an obvious inverse
\begin{align*}
\Hom^C_\Ug(\Up, M_b) &\longto \Hom^C_\Ug(\Up,  M_a) \\
s &\longmapsto \Bigg( \Up \xto{\Delta} \Up\otimes\Up \xto{\id \otimes s} \Up\otimes M \xto{\act} M \Bigg)\ .
\end{align*}
One should check that  
\begin{itemize}
\item it transforms continuous $\Ug$-linear maps $\Up\to M_a$ into continuous $\Ug$-linear maps $\Up\to M_b$;
\item it is $\Up$-linear and continuous w.r.t. the strong topologies;
\item it is $A$-linear.
\end{itemize}
Linearities are straightforward though lengthy. We relegate the continuity to the appendix (claim \ref{claim_some_continuity}).
\end{proof}

\begin{proof}[Proof of the part (3) of Theorem \ref{fundamental_thm}]
Let's denote by $\Xi$ the isomorphism $\Xi:\catA(\srdco M, \srdco N) \xto\cong \catA(M\ctp A, N\ctp A)$. 
Our statement can be formulated as: for $f\in\catH(M,N)$ the composition
\begin{equation}\label{comp_qejrtne}
M \xto{\id\ctp 1_A} M\ctp A \xto{\Xi \srdco f} N\ctp A \xto{ \id_N \ctp \eps_A} N 
\end{equation}
is equal to $f$. In the diagram
$$
\begin{tikzpicture}[label distance =-0.5 mm, scale = 1.2]	
	\node(lx) at (2.5,0){};
	\node(ly) at (3, 1.0){};
	\node(lz) at (-1.3, 1.6){};
	\pgfresetboundingbox;

	\node(a000) {$M $};
	\node(a010) at ($(a000)+(ly)$) {$N$};
	
	\node(a001) at ($(a000)+(lz)$)  {$\srdco M$};
	\node(a101) at ($(a001)+(lx)$) {$M\ctp A$};
	\node(a011) at ($(a001)+(ly)$) {$\srdco N$};
	\node(a111) at ($(a001)+(lx)+(ly)$) {$N\ctp A$};

	\draw[->] (a000)-- node[below right] {\small $f$} (a010);

	\draw[->] (a001)-- node[above left] {\small $\srdco f$} (a011);
	\draw[->] (a011)-- node[above] {\small $\cong$} (a111);

	\draw[->] (a001)-- node[left] {\small $ \pi_M$} (a000);
	\draw[->] (a101)-- node[right, near start] {\small $\id_N\ctp \eps_A$} (a000);	
	\draw[->] (a011)--  node[left, near end] {\small $\pi_N$} (a010);
	\draw[->] (a111)-- node[right] {\small $\id_M\ctp\eps_A$} (a010);

	\draw[->] (a001)-- node[below] {$\small \cong$} (a101);

	\draw[-, white, line width = 10pt] (a101)-- (a111);
	\draw[->] (a101)-- node[above] {\small$ \Xi\srdco f$} (a111);

\end{tikzpicture}
$$
the upper rectangle commutes by definition of $\Xi$, the rear rectangle is the naturality of $\pi_M$ and the two triangles are easily seen to commute. Thus we get the commutative square
$$
\begin{tikzpicture}[label distance =-0.5 mm]	
	\node(lx) at (3,0){};
	\node(ly) at (0, -1.5){};
	\pgfresetboundingbox;

	\node(a00) {$ M\ctp A$};
	\node(a01) at ($(a00)+(lx)$) {$N\ctp A$};
	\node(a10) at ($(a00)+(ly)$) {$M$};
	\node(a11) at ($(a00)+(lx)+(ly)$) {$N$};	
	
	\draw[->] (a00)-- node[above] {\small $\Xi\srdco f$} (a01);
	\draw[->] (a00)-- node[left] {\small $\id_M\ctp \eps_A$} (a10);
	\draw[->] (a01)-- node[right] {\small $\id_N \ctp \eps_A$} (a11);
	\draw[->] (a10)-- node[below] {\small $f$} (a11);
\end{tikzpicture}
$$
from which we see that the composition (\ref{comp_qejrtne}) is indeed equal to $f$. 
\end{proof}

\section{Representing $\Ug$ as endomorphisms of $\Yoneda:\catF\to\catCTV$}\label{section_Y}
Let's introduce a new category $\catH'$ with $\ob(\catH') := \ob(\catP)$ and $\catH'(M,N) := \catH(M,N)$ for $M,N\in \ob(\catH')$. 

\subsection{Categories $\catH'$ and $\catF$ are isomorphic.}
\begin{thm}\label{thm_srd}
We have an isomorphism of symmetric monoidal categories $ \srd: \catH' \xto \cong \catF   $ that makes the following diagram commutative:
$$
\begin{tikzpicture}[label distance =-0.5 mm]	
	\node(lx) at (2,0){};
	\node(ly) at (0, 1.5){};
	\pgfresetboundingbox;

	\node(a00) {$\catH'$};
	\node(a01) at ($(a00)+(lx)$) {$ \catF $};
	\node(a11) at ($(a00)+0.5*(lx)+(ly)$) {$\catP$};	
	
	\draw[->] (a00)-- node[below] {$\srd$} node[above]{$\cong$} (a01);
	\draw[<-] (a00)-- node[left] {$ \forg$} (a11);
	\draw[<-] (a01)-- node[right] {$\ctp A$} (a11);
\end{tikzpicture}
$$
The inverse $\srd^{-1}$ is given on objects as $M\tp A \mapsto M$ and on morphisms by Formula (\ref{iso_homsets}).
\end{thm}
\begin{proof}
We have $\ob(\catH') = \ob(\catP) = \ob(\catF)$ so $\srd$ is the identity on objects. To define it on the morphisms, take $M,N\in \ob(\catP)$. The theorem \ref{fundamental_thm} gives us an isomorphism
$$ \catH'(M,N) = \catH(M,N) \underset{\cong}{\xto\srdco} \catA(\srdco M, \srdco N  ) \cong \catF(M\ctp A, N\ctp A). $$
\end{proof}

\subsection{Coalgebra $(Q, \iota_C\circ \pi_C) $ in $\Split(\catH')$}
We have a fully faithful, strong-monoidal functor $\Split(\catH') \to \catH$. Our next goal is to see $C = \Ug$ as a coalgebra in $\Split(\catH')$.

Denote by $Q:= \srdieckof C\in \ob(\catP)$. We have a surjective $\catH$-morphism $\pi_C:Q \to C$ (see the subsection \ref{subsec_pi}).
Since $C$ is a free one dimensional $\Ug$-module (we take the base vector $1\in \Ug=C$), we can define a right inverse $\iota_C :C \to Q$ of $\pi_C$ just by specifying any $s_0 := \iota_C(1) \in Q = \Hom^C_\Ug(\Up, \Ug)$ that satisfies $s_0(1_\Up) = 1_\Ug$. This way we get that $C$ is isomorphic to $(Q, \iota_C \circ \pi_C)$.

The coalgebra structure $\Delta:C\to C\tp C$, $\eps: C \to \field$ on $C$ gives us the following coalgebra structure on $(Q, \iota_C\circ \pi_C) $:
\begin{equation}\label{coalg_str_on_Q_daco} 
Q \xto{\pi_C} C \xto{\Delta} C\ctp C \xto{\iota_C\ctp \iota_C} Q\ctp Q\; ; \quad Q\xto{\pi_C} C \xto\eps \field\ . 
\end{equation}

\begin{rem}\label{rem_daco}
Note that the explicit isomorphism between the functor $\Yoneda_{(Q, \iota_C\circ\pi_C)}:\catH'\to \catTV$ represented by $(Q, \iota_C \circ \pi_C)$ and $\catH' \xto{\text{forg}} \catCTV$ is given for $M\in \ob(\catH')=\ob(\catP)$ as
\begin{align*}
\Yoneda_{(Q, \iota_C\circ\pi_C)}(M) = \set{ f\in \catH(Q,M) \;|\; f\circ (\iota_C \circ\pi_C) = f } \;\;& \xto\cong \;\; M\\
f \;\;& \mapsto \;\; f( s_0).
\end{align*}
\end{rem}

\subsection{Our choice of $s_0$}\label{subsection_choice__of_s0}

We choose a special $s_0$ as follows. The restriction isomorphism (claim \ref{claim_fhitr}) gives us 
$$Q = \srdieckof C = \Hom_\Ug^C(\Ud,\Ug) = \Hom_\field^C(S\g^*, \Ug)  \ .$$ 
So we prescribe $s_0$ by taking $s_0(1_{S\g^*}) = 1_\Ug$ and $s_0|_{S^{>0} \g^*} = 0 $.
\begin{rem}\label{rem_iota}
We will use in the computations the following formulation of our definition of $s_0$:
$$ \pi_C(s_0) = 1\quad\text{and}\quad \pi_C\big( y\act s_0 \big) = 0 \quad \text{if $y\in S^k \g^*$, $k> 0$. } $$
\end{rem}

\subsection{Coalgebra $\big( (Q\ctp A, p), \Delta_Q, \eps_Q \big)  $ in $\Split(\catF)$}
We apply the isomorphism $\srd$ to $(Q, \iota_C\circ \pi_C)$ in $\Split(\catH')$ and get an object $\big( Q\ctp A, p \big)$ in $\Split(\catF)$ where $p:= \srd(\iota_C\circ \pi_C)$. Applying $\srd$ to (\ref{coalg_str_on_Q_daco}) we get a coproduct and a counit on $(Q\ctp A, p)$ which we denote as $\Delta_Q$ and $\eps_Q$.

\subsection{Functor $\Yoneda: \catF \to \catCTV $}\label{subsecion_Yoneda_from_catF}
Denote by $\Yoneda$ the monoidal functor represented by the coalgebra $(Q\tp A, p) \in \Split(\catF)$. It is a priori just a monoidal functor to $\catTV$ but since it is isomorphic to the functor
$$ \catF \xto\cong \catH' \xto{\text{forget}} \catCTV $$  we see that $\Yoneda$ is actually a strong monoidal functor to $\catCTV$. One also sees that the bialgebra $\End(\Yoneda)$ is isomorphic to $\Ug$.

\subsection{Explicit formulas}
In the rest of this section we write down some pictorial formulas which will be used in the subsequent computations. First of all, the maps  
\begin{align*}
p \in&\; \catF(Q\ctp A, Q\ctp A) = \catP(Q, Q\ctp A)\\
\Delta_Q \in&\; \catF(Q\ctp A, (Q\ctp A)\tp_A (Q\ctp A)) = \catP(Q, Q\ctp Q\ctp A) ; \\
\eps_Q \in&\; \catF( Q\ctp A, A) = \catP(Q, A) 
\end{align*}
are uniquely determined by (see the formula (\ref{iso_homsets})):
\begin{align}
\label{equa_def_p}
 \input{\cesta/obr13} \quad &= Q \xto{\pi_C} C \xto{\iota_C} Q \\
\label{equa_def_DeltaQ} \input{\cesta/obr14}  \quad &= \quad Q \xto{\pi_C} C \xto\Delta C\ctp C \xto{\iota_C \ctp\iota_C}  Q\ctp Q  \\
\label{equa_def_epsQ} \input{\cesta/obr15} \quad &= \quad Q \xto{ \pi_C } C \xto\eps \field 
\end{align}

\begin{rem}\label{rem_explicit_iso}
Using the remark \ref{rem_daco}, the explicit isomorphism between $\catP\xto{\ctp A} \catF \xto{\Yoneda} \catCTV$ and $\catP \xto{\text{forget}} \catCTV$ is given for $M\in \ob(\catP)$ by 
\begin{align*}
\Yoneda(M\ctp A) = \set{ \input{\cesta/obr16} \in \catP( Q, M\ctp A) \; \Bigg| \; \input{\cesta/obr17}= f}\to M: f\mapsto \input{\cesta/obr18}
\end{align*}
where $$ is the linear map $\field \to Q: 1\mapsto s_0$.
\end{rem}

\begin{notation}\label{notation_m_hat}
For later use we denote by $\hat m \in \Yoneda(M\ctp A)$ the element corresponding to $m\in M$ under this isomorphism. That is, $\hat m \in \catP(Q, M\ctp A)$ is uniquely characterized by:
\begin{equation}\label{equa_def_mhat}
 \input{\cesta/obr19}\quad  = \quad  Q\xto{\pi_C} C=\Ug \xto{\argument \act m} M . 
\end{equation}
\end{notation}

\begin{rem}
$\Delta_Q$ and $\eps_Q$ are actually morphisms in $\Split(\catF)$:
$$ \Delta_Q : (Q\ctp A, p) \to (Q\ctp A, p) \tp_A (Q\ctp A, p), \quad \eps_Q: (Q\ctp A, p) \to A .$$
This translates into diagram equations:
\begin{equation}\label{equas_Delta_and_eps_are_maps_of_direct_summands}
 \input{\cesta/obr20} = \input{\cesta/obr21}  \quad\text{and}\quad \input{\cesta/obr22} = \input{\cesta/obr23}. 
\end{equation}
Similarly, the fact that $\eps_Q$ is a counit for $\Delta_Q$ looks in pictures as:
\begin{equation}\label{equa_epsQ_is_counit}
\input{\cesta/obr24} = \input{\cesta/obr25} = \input{\cesta/obr26} .
\end{equation}
\end{rem}
\begin{rem}
The explicit monoidal structure on $\Yoneda$ is
\begin{align*} 
\Yoneda(M\ctp A) \ctp \Yoneda(N\ctp A) &\to \Yoneda( M\ctp N \ctp A ) & \field & \to \Yoneda( \field \ctp A ) \\ 
\input{\cesta/obr27} \tp \input{\cesta/obr28} &\mapsto \input{\cesta/obr29}  & 1 &\mapsto \input{\cesta/obr23} 
\end{align*}
\end{rem}

\section{Passing from $\field$ to $\kh$}\label{section_passing_to_kh}
To get a similar representation for $\Ugh$, we simply apply to our categories the construction $\catC \mapsto \catC\hh$ described in the subsection \ref{subsection_TVkh-categories}.
We denote the categories important for us as follows: 
$$ \catPh := \catP\hh; \quad \catFh := \Big(\RFree A (\catP)\Big)\hh = \RFree A (\catPh) \ .$$
In $\catFh$ we still have our $Q$ and $p:Q\to Q$ and can form the functor $\YQp$ represented by $(Q,p)$. A priori it will be a functor to $\catTVkh$, but according to Proposition \ref{prop_Mh_complete} it actually takes values in $\catCTVkh$.
This way, $\YQp: \catFh \to \catCTVkh$ becomes a strong-monoidal functor (it is just monoidal as a a functor to $\catTVkh$) and we get an isomorphism of bialgebras   $\Ugh \cong \End\Big( \YQp : \catFh \to \catCTVkh  \Big)$.
\begin{notation}
We define the forgetful functor $\catPh \xto\forg \catCTVkh$ on objects by $$M\in \ob(\catPh) = \ob(\catP) \mapsto \big(\FCTVP(M)\big)\hh$$ and on morphisms in a natural way.  
\end{notation}

\section{Applying an associator}
\subsection{Drinfeld associator}\label{subs_Drinfeld_assoc}
Denote by $\field \langle\langle X,Y\rangle\rangle$ the degree-wise completed free algebra generated by non-commuting elements $X,Y$.
Recall that an associator $\Phi$ is a group-like\footnote{The coproduct is determined by $X,Y$ being primitive.} element $\Phi(X,Y)\in \field \langle\langle X,Y\rangle\rangle$ satisfying some properties (pentagon and hexagon relations) that ensure (and are equivalent to) that if we take a Lie algebra $\p$ with an invariant $t\in S^2\p$ then $\Uph$ with the ordinary algebra structure, coproduct and counit but with\footnote{Here $t_{1,2}:= 1\tp t$ and $t_{2,3}:= t\tp 1$ are elements of $\Up\ctp \Up \ctp \Up $.}
\begin{equation}\label{equa_dr_quant}
  R_\p:= \exp(\frac{\hbar t}{2});\quad \Phi_\p := \Phi( \hbar t_{12}, \hbar t_{23}) 
\end{equation}
becomes a quasi-triangular quasi-bialgebra that we denote as $\Unih\p$. 
\begin{rem}
One can show that $\Phi(X,Y)$ is of the form $\Phi(X,Y) = 1 +$ terms of order at least 2 in $X,Y$.
\end{rem}
We define a new braided monoidal category $\catPhPhi$ as being the same category as $\catPh$, with the same bifunctor $\ctp:\catPhPhi \times \catPhPhi \to \catPhPhi$ and the same identity object but with the braiding and associativity isomorphisms defined as\footnote{Technically speaking, objects in $\ob(\catPh) = \ob(\catP)$ are just $\p$-modules over $\field$ and thus the multiplication by $R_\p$ does not make sense. However, the multiplication by $t$ is a well defined morphism $M\ctp N \xto{t\act} M\ctp N$ in $\catP$ and thus we can define the $\catPh$-morphism $M\ctp N \xto{ R_\p\act} M\ctp N$ as $\exp\Big(\frac \hbar 2 \cdot ( t \act \argument )\Big)$. Multiplication by $\Phi_\p$ is defined in the same way.}
$$ \beta: M\ctp N \xto{R_\p\act } M\ctp N \xto{\text{ symmetry in $\catPh$ }} N\ctp M $$
$$ \phi:(K\ctp L)\ctp M \xto{ \Phi_\p \act } (K\ctp L) \ctp M \xto{\text{assoc. in $\catPh$}} K\ctp (L \ctp M ) .$$
\begin{rem}\label{rem_Drinfeld_construction}
The identity functor $\catPhPhi\xto = \catPh$ is strong quasi-monoidal so the composition $$\forg:\catPhPhi\xto = \catPh \xto{\forg} \catCTVkh$$ also is. The quasi-bialgebra $\End(\catPhPhi \xto{\forg} \catCTVkh)$ is isomorphic to $\Unih \p$ defined above.
Note also that the classical limit of the quasi-bialgebra $\Unih\p$ is the quasi-Lie bialgebra associated to the pair $(\p,t)$ (see Definition \ref{defn_associated_quasi_lie_bialgebra}).
\end{rem}

\subsection{Category $\catFhPhi$.}\label{subsection_catFhPhi}
\begin{prop}
The algebra $A$ stays a commutative associative algebra also in $\catPhPhi$ (with the same multiplication).
\end{prop}
\begin{proof}
Since both $R_\p$, $\Phi_\p$ are sums of the form $1 +$ terms of order at least one in $\hbar\Omegap$, we just need to show that the $\catP$-morphism $A \ctp A \xto{\Omegap\act} A\ctp A \xto{\mu_A} A $ is 0. From the last part of the theorem \ref{fundamental_thm} we see that this is equivalent to showing
$$ \big( A \ctp A \xto{\Omegap\act} A\ctp A \xto{\mu_A} A \xto{\eps_A} \field \big) =0.$$
Since $\eps_A$ is an augmentation on $A$ in $\catH$, we rewrite the left had side as 
\begin{equation}\label{equa_ehrkr}
A \ctp A \xto{\Omegap\act} A\ctp A \xto{\eps_A\ctp\eps_A} \field .
\end{equation}
Now $\Omegap \in (\g\ctp \g^*) + (\g^*\ctp \g)$ and from $\g$-linearity of $\eps_A\in \catH(A,\field)$ we see that 
$$ \eps_A\circ (x\act\argument) = (x\act\argument ) \circ \eps_A = 0 \quad \forall x \in \g$$
and thus the composition (\ref{equa_ehrkr}) is 0.
\end{proof}

We denote $\catFhPhi := \RFree A (\catPhPhi)$. 

\subsection{Bijection $\catFh \to \catFhPhi$} Our definitions give us
$$ \ob(\catFh) = \ob(\catPh) = \ob(\catPhPhi) = \ob(\catFhPhi) $$
$$ \catFh(X\ctp A,Y \ctp A) = \catPh( X, Y\ctp A ) = \catPhPhi(X,Y \ctp A) = \catFhPhi(X\ctp A,Y \ctp A)   \quad X,Y \in \ob(\catPh). $$

Note that the formula (\ref{equa_comp_freemod}) defining the composition in $\catFh$ and $\catFhPhi$ depends on the associativity isomorphisms and thus the identity map $\catFh \to \catFhPhi$ is not a functor. Similarly, although it respects the tensor product on objects, it does not preserve the tensor product of morphism since the expression (\ref{equa_tp_freemod}) involves the braiding and associativity isomorphisms in the underlying categories. 

At least however, the associativity isomorphisms in $\catPh$ and $\catPhPhi$ are equal modulo $\hbar^2$ and the braidings are equal modulo $\hbar$, so the compositions of morphisms in $\catFh$ and $\catFhPhi$ are the same modulo $\hbar^2$ and the tensor products of morphisms are the same modulo $\hbar$.

Next we want to extend this construction to a map $\Split(\catFh)\to \Split(\catFhPhi)$.

\subsection{Map $\Split(\catFh) \to \Split(\catFhPhi)$ on objects}\label{Map_on_Split_objects}
Given $(X\ctp A, p) \in \ob(\Split(\catFh))$ we can consider $X\ctp A$ as an object  in $\catFhPhi$ and $p$ as a morphism $p: X \ctp A \to X\ctp A$ in $\catFhPhi$ but we can't be sure whether $p^2=p$ holds also in $\catFhPhi$. We know however that it holds modulo $\hbar$ and can use the following lemma to tweak our $p$ a bit.
\begin{lem}\label{lem_pPhi}
Let $B$ be an algebra over $\kh$, complete w.r.t. the $\hbar$-adic topology. Let $p\in B$ satisfies $p^2 \equiv p  \mod \hbar$. Then there exists an element $p' \in B$ s.t. 
$$p'^2 = p' \quad\text{and} \quad p' \equiv p \mod \hbar \ .$$
More concretely, there exist coefficients $\alpha_1,\alpha_2,\ldots \in \Q$ independent of $B$ and $p$ s.t.
$$ p' = p+ (p-\half) \cdot \Big[  \alpha_1(p^2 - p) + \alpha_2 (p^2 - p)^2 + \ldots  \Big]$$
\end{lem}
\begin{proof}
Denote $a:= 2p-1$, $a' := 2 p' -1$. Then $a^2 - 1 = 4(p^2 -p)$ and so 
$$ \Big( p^2 \equiv p \mod \hbar \Big) \iff \Big( a^2 \equiv 1 \mod \hbar \Big) \quad\text{and}\quad \Big( p'^2  = p' \Big) \iff \Big( a'^2 = 1 \Big)\ .   $$
Given $a$  satisfying $a^2 \equiv 1 \mod \hbar$, we can find $a'$, $\hbar$-close to it from the formula
$$ a':= \frac{a}{\sqrt{a^2}} = a\cdot \big( 1 + (a^2 - 1) \big)^{-\half} = a\cdot \big( 1 + \beta_1(a^2 - 1) + \beta_2 (a^2 - 1 )^2 + \ldots \big)$$
where $\beta_1,\beta_2, \ldots \in \Q$ are the coefficients of the Taylor expansion of $(1+ x)^{-\half}$ at $x=0$. So we can express $p'$ out of it as 
\begin{align*}
 p' =  \half (1+ a') & = \half \Big[ (1+a) + a\cdot\big( \beta_1 (a^2 -1) + \beta_2(a^2 - 1)^2 + \ldots \big)  \Big] = \\
& = p+ (p-\half ) \big( \frac{\beta_1}{2} (p^2 - p) + \frac{\beta_2 }{2} (p^2 - p)^2 + \ldots \big) \ . 
\end{align*}
\end{proof}
In our case $B = \catFhPhi(X \ctp A , X\ctp A)$ and we denote $p'$ from the lemma by $p^\Phi$. This way, we get a $(Q\ctp A, p^\Phi)$ in $\catFhPhi$.

\subsection{Map $\Split(\catFh) \to \Split(\catFhPhi)$ on morphisms}\label{subs_quant_of_morphisms}
Given $(X\ctp A, p), (Y\ctp A, q)  \in \ob(\Split(\catFh))$ we define a linear map
\begin{align*}
 \Proj\Big( (X\ctp A, p), (Y\ctp A, q)  \Big) &\to \ProjPhi\Big( (X\ctp A, p^\Phi), (Y\ctp A, q^\Phi)  \Big): \\ f &\mapsto q^\Phi \circ f \circ p^\Phi \quad\text{(composition in $\catFhPhi$).} 
\end{align*}

\begin{rem} 
This is actually a topological isomorphism since we can define a map in the opposite direction $f \mapsto q\circ f\circ p$ (compositions in $\catFh$) and then use the following lemma:
\end{rem}
\begin{lem}\label{lem_iso_projections}
Let $V \in \ob(\catTVkh)$ (i.e. $V$ is $\hbar$-complete) and let's have $ \Pi_1, \Pi_2 \in \catTVkh(V,V) $ satisfying $ \Pi_1^2 = \Pi_1$, $\Pi_2^2 = \Pi_2$ and $\Pi_1 \equiv \Pi_2 \mod \hbar$.  Then $\Pi_2\big|_{\Ima(\Pi_1)} : \Ima(\Pi_1) \to \Ima(\Pi_2)$ defines an isomorphism (in $\catTVkh $) between $\Ima(\Pi_1)$ and $\Ima(\Pi_2)$.
\end{lem}
\begin{proof}
Denote $X_i := \Ima(\Pi_i)$. We have two maps: $$\Pi_2\big|_{X_1} : X_1 \to X_2 \quad\text{and} \quad  \Pi_1\big|_{X_2} : X_2 \to X_1 \ .$$ 
To see that they are isomorphisms we show that $\Pi_1\circ \Pi_2 \big|_{X_1}: X_1 \to X_1$ and symmetrically $\Pi_2\circ \Pi_1 \big|_{X_2}: X_2 \to X_2$ are isomorphisms. We want to define the inverse of $\Pi_1\circ \Pi_2 \big|_{X_1}$ as $\sum_n \xi^n$ where 
$$  \xi := \id_{X_1} - \Pi_1\circ \Pi_2 \big|_{X_1}  = \big( \Pi_1\circ \Pi_1 - \Pi_1 \circ \Pi_2 \big) \Big|_{X_1} = \Pi_1\circ \big( \Pi_1 - \Pi_2 \big) \Big|_{X_1} \ .$$
One needs to check that this sum converges to a continuous map. From the last expression one sees that $\Ima\  \xi \subset \hbar V $ and consequently $\Ima\  \xi^k \subset \hbar^k V$. Thus $\xi^k \to 0$ uniformly in the $\hbar$-topology. Since $X_1$ is $\hbar$-complete (it is a direct summand of a $\hbar$-complete space) we see that our sum converges uniformly in the $\hbar$-topology.  From the claim \ref{claim_h_topology_has_fewer_open_sets} we see that our sum actually converges uniformly also w.r.t. the ordinary topology on $V$. This implies that its limit is continuous.
\end{proof}

\subsection{Functor $\YPhi: \catFhPhi \to \catCTVkh$}\label{subsection_YPhi}
Recall that in the section \ref{section_Y} we defined an object $(Q\ctp A, p)$ in $\Proj$ and denoted the functor represented by it as $\catFh \xto \Yoneda  \catCTVkh$. Now we have $(Q\ctp A, p^\Phi)$ in $\ProjPhi$ and denote the represented functor $\catFhPhi \xto \YPhi \catCTVkh$.

Note that the functors $(\catPh \xto{\argument \ctp A} \catFh \xto{\Yoneda} \catTVkh)$ and $(\catPh = \catPhPhi \xto{\argument \ctp A} \catFhPhi \xto\YPhi \catTVkh)$ are naturally isomorphic. Explicitly,
$\Yoneda(M\ctp A) = \Proj\big((Q\ctp A, p), M\ctp  A\big)$, $\YPhi(M\ctp A) = \ProjPhi\big((Q\ctp A, p^\Phi), M\ctp  A\big)$ and the isomorphism is $f\mapsto f\circ_\catFhPhi p^\Phi$ from the subsection \ref{subs_quant_of_morphisms}.

Now we have three functors $\catPh \to \catCTVkh$ and natural isomorphisms between them:\footnote{Note that all the functors in the diagram are strong monoidal with exception of $\YPhi$ and $\catPh \leftrightarrow \catPhPhi$ which are strong quasi-monoidal (we will soon define the quasi-monoidal structure on $\YPhi$). The natural equivalence in the left window is monoidal, the other is not. }
$$
\begin{tikzpicture}[label distance =-0.5 mm, baseline=-1cm]	
	\node(lx) at (2.1,0){};
	\node(ly) at (0, -2){};
	\pgfresetboundingbox;

	\node(a00) {$\catPh$};
	\node(a01) at ($(a00)+(lx)$) {$\catPhPhi$};
	\node(a10) at ($(a00)+(ly)$) {$\catFh$};
	\node(a11) at ($(a00)+(lx)+(ly)$) {$\catFhPhi$};	
	\node(a20) at ($(a00)+2*(ly)$) {$\catCTVkh$};	
	
	\draw[<->] (a00)-- node[above] {$\cong$} (a01);
	\draw[->] (a00)-- node[right] {$ \argument \ctp A$} (a10);
	\draw[->] (a01)-- node[right] {$ \argument \ctp A$} (a11);
	\draw[->] (a10)-- node[right] {$ \Yoneda$} (a20);
	\draw[->] (a11)-- node[right] {$ \YPhi$} (a20);

	\draw[->] (a00) to [out=180,in=180]  node[left] {$ \forg$} (a20);

	\node(zdvih) at ($0.2*(ly)$){};
	\node at ($(a10)+0.5*(lx)-(zdvih)$) {{ \huge $\cong$ }};	
	\node at ($(a10)-0.35*(lx)-(zdvih)$) {{ \huge $\cong$ }};		

\end{tikzpicture}
$$

\begin{notation}
We denote by $\hat m \in \Yoneda(M\ctp A)$ and $\hat m^\Phi \in \YPhi (M\ctp A)$ the elements corresponding under these natural isomorphisms to $m\in M\hh$ for $  M \in  \ob(\catPh)$. The explicit isomorphism $\hat m \mapsto \hat m^\Phi $ is given by the formula $\hat m^\Phi = \hat m \circ_\catFhPhi p^\Phi$ i.e.
$$ \input{\cesta/obr30} = \input{\cesta/obr31} \quad\text{(diagram in $\catPhPhi$).}$$

\end{notation}

\subsection{Bijection $ \Ug[[\hbar]] \xto\simeq \End(\catFhPhi \xto{\YPhi} \catCTVkh) $.}\label{subsection_bijection}
Recall that we have a bialgebra isomorphism
$$ \Ugh\cong \End_{\Proj}\big( ( Q\ctp A, p) \big) = \End \big(\catFh \xto\Yoneda \catCTVkh \big) .$$
The subsection \ref{subs_quant_of_morphisms} gives us an isomorphism of topological $\kh$-modules
$$ \End_{\Proj}\big( ( Q\ctp A, p) \big) \cong \End_{\ProjPhi}\big( ( Q\ctp A, p^\Phi) \big) = \End \big(\catFhPhi \xto\YPhi \catCTVkh \big) $$
Composing, we get an isomorphism of topological $\kh$-modules
$$ \Ugh \xto\cong \End \big(\catFhPhi \xto\YPhi \catCTVkh \big) $$

\begin{lem}\label{lem_incl}
The composition of $\kh$-module morphisms
\begin{align*}
&\Ug[[\hbar]] \; \xto{\cong}\;  \End(\catFhPhi \xto{\YPhi} \catCTVkh)\; \to \\
&\quad \to \; \End(\catPhPhi \xto{\ctp A} \catFhPhi \xto{\YPhi} \catCTVkh)\; \xto{\cong}\; \End( \catP \xto{\forg} \catCTVkh ) = \Up[[\hbar]]
\end{align*}
is modulo $\hbar$ equal to the ordinary inclusion $\Ug[[\hbar]] \hookrightarrow \Up[[\hbar]]$.
\end{lem}

\subsection{Quasi-monoidal structure on $\YPhi$}\label{subsection_quasi_mon_str_on_YPhi}
Following the subsection \ref{subsec_monoidal_forgetful} , we want to find a quasi-coalgebra structure on $(Q\ctp A, p^\Phi)$. The first try is to define 
\begin{equation}\label{eq_Delta_prime}
\Delta'_Q := (p^\Phi  \tpFhPhi p^\Phi)  \cFhPhi \Delta_Q  \cFhPhi p^\Phi; \quad \eps'_Q := \eps_Q  \cFhPhi p^\Phi.
\end{equation}
We will see that $ \eps'_Q = \eps_Q $ in the lemma \ref{lem_epsPrime}. First we need
\begin{claim}\label{claim_ghkdo}
Let $a\in \catFh(X \ctp A, A)$ and $f\in \catFh(Y\ctp A, X\ctp A)$ for some $X,Y \in \catPh$. Then
\begin{enumerate}
\item $ a  \cFhPhi f = a  \cFh f$
\item $ a \tpFhPhi a = a \tpFh a  $
\end{enumerate}
\end{claim}
\begin{proof}
The two expressions
$$ a  \cFh f = \input{\cesta/obr32};\quad a  \tpFh a = \input{\cesta/obr33} $$
contain neither the associativity isomorphism nor the braiding so they mean the same in $\catPh$ and in $\catPhPhi$. 
\end{proof}
\begin{lem}\label{lem_epsPrime}
$ \eps_Q \cFhPhi p^\Phi = \eps_Q $
\end{lem}
\begin{proof}
By our definition (see the subsection \ref{Map_on_Split_objects}), $p^\Phi$ can be expressed as
$$ p^\Phi = p + \big( p\cFhPhi p - p\big) \cFhPhi \text{something}.  $$ 
We know that $\eps_Q \cFh p = \eps_Q$ and thus also (claim \ref{claim_ghkdo}) $\eps_Q \cFhPhi p = \eps_Q$. So we just need to show that $\eps_Q \cFhPhi \big( p \cFhPhi p - p \big) = 0$ what is now easy.
\end{proof}

The reason to write the formulas (\ref{eq_Delta_prime}) is that it forces $\Delta_Q'$ and $\eps_Q'$ to become maps between direct summands in $\catFhPhi$
$$ \Delta'_Q : (Q\ctp A, p^\Phi) \to (Q\ctp A, p^\Phi) \tp_A (Q\ctp A, p^\Phi), \quad \eps'_Q: (Q\ctp A, p^\Phi) \to A .$$
It is however not clear whether $\eps'_Q =\eps_Q$ will be a counit for $\Delta'_Q$ , i.e. whether\footnote{If $(X,p)$ is a direct summand then $\id_{(X,p)} = p$.}
$$ \big(\eps_Q \tpFhPhi \id_{(Q\ctp A)} \big) \cFhPhi \Delta'_Q = \big( \id_{(Q\ctp A)} \tpFhPhi \eps_Q \big) \cFhPhi \Delta'_Q = p^\Phi   .$$

\begin{lem}\label{lem_save_counit}
Let's have $C\xto\Delta C\tp C$; $ C\xto\eps I$ (in some monoidal category $\catC$) satisfying that the composition $ C \xto\Delta C\tp C \xto{ \eps \tp\eps } I $ equals $\eps$ and that the maps
$$ r:  C \xto\Delta C\tp C \xto{ \eps \tp\id } C \quad ;\quad s: C \xto\Delta C\tp C \xto{ \id \tp\eps } C $$ 
are invertible.

Then  $\bar\Delta: C \xto\Delta C\tp C \xto{r^{-1} \tp s^{-1}} C\tp C $ and $\bar\eps := \eps$ form a quasi-coalgebra structure on $C$.
\end{lem}
\begin{proof}
First note that we have $ \eps\circ r = (\eps \tp \eps) \circ \Delta = \eps $ and thus $ \eps \circ r^{-1} = \eps$. Using this we can calculate:
\begin{align*}
(\bar\eps \tp \id)\circ \bar\Delta &= (\eps \tp \id)\circ(r^{-1} \tp s^{-1}) \circ \Delta = \Big( (\eps \circ r^{-1}) \tp s^{-1} \Big)\circ \Delta =\\
&= s^{-1} \circ (\eps \tp \id) \circ\Delta =s^{-1} \circ s = \id_C.
\end{align*}
\end{proof}

In order to use this lemma for our would-be quasi-coalgebra $ \Big( (Q\ctp A, p^\Phi), \Delta'_Q, \eps'_Q \Big) $ in $\Split(\catFhPhi)$, we need to check  
\begin{equation}\label{equa_uuuuuu}
\big(\eps_Q \tpFhPhi \eps_Q \big) \cFhPhi \Delta'_Q = \eps_Q 
\end{equation}
and verify that
\begin{equation}\label{equa_f_and_g}
r:= \big(\eps_Q \tpFhPhi \id_{(Q\ctp A)} \big) \cFhPhi \Delta'_Q ; \quad s:= \big( \id_{(Q\ctp A)} \tpFhPhi \eps_Q \big) \cFhPhi \Delta'_Q 
\end{equation}
are invertible as maps $(Q\ctp A, p^\Phi) \to (Q\ctp A, p^\Phi) $ in the category $\Split(\catFhPhi)$. The invertibility is easy, since\footnote{Just use that $\tpFhPhi$, $\cFhPhi$, $\Delta_Q'$ are modulo $\hbar$ the same as $\tpFh$, $\cFh$, $\Delta_Q$.}
$$ r\equiv p \equiv p^\Phi \mod\hbar \quad \text{ and }\quad  s\equiv p \equiv p^\Phi \mod\hbar$$
and so we can define the inverses by the usual series.
\begin{proof}[Proof of (\ref{equa_uuuuuu})]
Just use the claim \ref{claim_ghkdo} and the lemma \ref{lem_epsPrime} repeatedly:
\begin{align*}
\big(\eps_Q \tpFhPhi \eps_Q \big) \cFhPhi \Delta'_Q &= \big(\eps_Q \tpFhPhi \eps_Q \big) \cFhPhi \big(p^\Phi \tpFhPhi p^\Phi \big) \cFhPhi \Delta_Q \cFhPhi p^\Phi =\\
&= \bigg( \big(\eps_Q \cFhPhi p^\Phi \big) \tpFhPhi \big(\eps_Q \cFhPhi p^\Phi \big) \bigg)  \cFhPhi \Delta_Q \cFhPhi p^\Phi =\\
&= \bigg( \big( \eps_Q \tpFh \eps_Q \big) \cFh \Delta_Q \bigg) \cFhPhi p^\Phi = \eps_Q \cFhPhi p^\Phi = \eps_Q.
\end{align*}
\end{proof}
So finally we can use the lemma \ref{lem_save_counit} to equip $(Q\ctp A, p^\Phi)$ with a quasi-coalgebra structure
\begin{equation}\label{equa_def_DeltaQPhi}
\Delta_Q^\Phi:= (r^{-1} \tp_A s^{-1}) \circ \Delta_Q' \quad \text{and} \quad \eps_Q^\Phi:= \eps_Q   .
\end{equation}

\section{The twist}
We have two quasi-monoidal functors $$ F_1: \catPhPhi = \catPh \xto{\forg} \catCTVkh \; ; \quad\quad  F_2 :\catPhPhi \xto{\ctp A} \catFhPhi \xto{\YPhi} \catCTVkh $$. We have also a natural isomorphism between them which we denote by $\alpha$:
$$
\begin{tikzpicture}[label distance =-0.5 mm, baseline = -1cm]	
	\node(lx) at (1.5,0){};
	\node(ly) at (0, -1){};
	\pgfresetboundingbox;

	\node(a00) {$\catPhPhi $};
	\node(a11) at ($(a00)+(lx)+(ly)$) {$\catFhPhi$};
	\node(a20) at ($(a00)+2*(ly)$) {$ \catCTVkh$};
	
	\node(stred) at ($(a00)+0.3*(lx)+(ly)$) [inner sep = 0pt, shape=rectangle, rotate=0, label=above:$\alpha$] {{\Huge $\Rightarrow$}};

	\draw[->] (a00)-- node[above] {$\ctp A$} (a11);
	\path[->] (a00) edge [bend right] node[left] {$ \forg  $} (a20);
	\draw[->] (a11)-- node[right] {$\Yoneda^\Phi $} (a20);
\end{tikzpicture}
\quad :\quad  m\in  M \; \mapsto \; \input{\cesta/obr30}.
$$
The two functors define two quasi-bialgebras. Using $\alpha$ we can identify $\End(F_1)$ and $\End(F_2)$ as algebras, but we will get two quasi-bialgebra structures that differ by a twist.\footnote{See the remark (\ref{rem_twists}).} The goal of this section is to compute this twist up to the first order in $\hbar$. 

First we need some computational tools.
\begin{claim}\label{claim2}
Let $ y\in S^k\g^*$ with $k>0$. Then one has
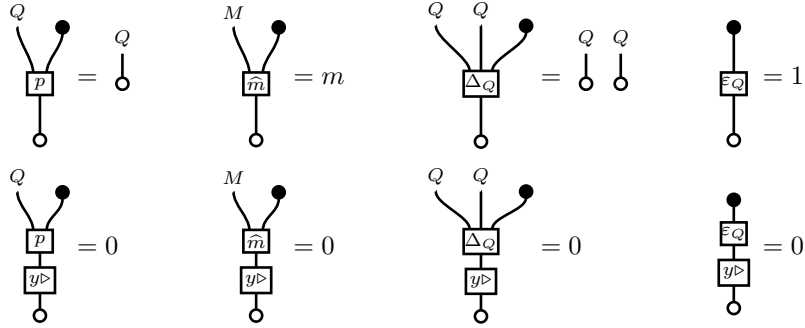
\begin{align*}
\input{\cesta/obr36} &= & \input{\cesta/obr37} &= m  & \input{\cesta/obr38} & =    &  \input{\cesta/obr39} &= 1  \\
\input{\cesta/obr40} &= 0  & \input{\cesta/obr41} &= 0  & \input{\cesta/obr42} &= 0 & \input{\cesta/obr43} & = 0 
\end{align*}
\end{claim}
\begin{proof}
One just uses the defining formulas  (\ref{equa_def_p}), (\ref{equa_def_mhat}), (\ref{equa_def_DeltaQ}), (\ref{equa_def_epsQ}) and the remark \ref{rem_iota}.
\end{proof}

\begin{cor}\label{cor_p_m_y}
If $y\in \g^*$ then we have 
\begin{align*}
\input{\cesta/obr44} & = - \input{\cesta/obr45} & \input{\cesta/obr46} & = - y\act m & \input{\cesta/obr47} &= 0 
\end{align*}
\end{cor}
\begin{proof} One just uses the $\Up$-linearity of $p$, $\hat m$ and $\eps_Q$. For example, the computation for $\hat m$ would look like this:
$$
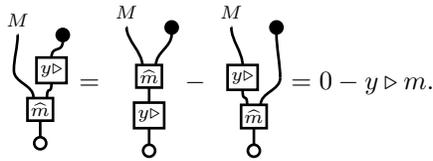
\input{\cesta/obr46} = \input{\cesta/obr41} - \input{\cesta/obr48} = 0 - y\act m .
$$
\end{proof}

\begin{claim}\label{claim4}
One has the following congruences modulo $\hbar^2$:
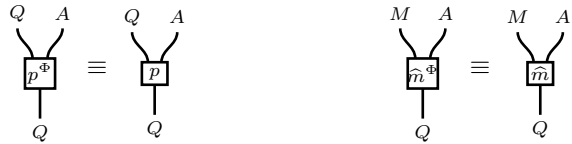
\begin{align*}
\input{\cesta/obr49}  &\equiv \input{\cesta/obr26} & \input{\cesta/obr30} &\equiv \input{\cesta/obr50}   
\end{align*}
\end{claim}

\begin{cor} \label{cor_of_claim4}
The natural isomorphism $\alpha^{-1}: \YPhi(M\ctp A) \to M$ is given (modulo $\hbar^2$) by the formula (see the remark \ref{rem_explicit_iso}):
$$ \alpha^{-1} \Bigg( \input{\cesta/obr16} \Bigg) \equiv \input{\cesta/obr18} \mod \hbar^2 . $$
\end{cor}

\begin{claim}\label{claim5}
\begin{align*}
  \input{\cesta/obr51} &\equiv \input{\cesta/obr52} \mod \hbar^2
\end{align*}
\end{claim}
\begin{proof}
Looking at (\ref{eq_Delta_prime}) we see that
\begin{equation}\label{equa_DeltaPrime_and_Delta}
 \input{\cesta/obr53} = \input{\cesta/obr54}  \equiv \input{\cesta/obr52} \mod \hbar^2 \ .
\end{equation}
So, it is enough to see that $\Delta_Q^\Phi \equiv \Delta'_Q \mod \hbar^2$. By the relation (\ref{equa_def_DeltaQPhi}) it is thus sufficient to show that $r$ and $s$ (see (\ref{equa_f_and_g})) are congruent to $p^\Phi$ modulo $\hbar^2$. On pictures (in $\catPhPhi$), $r$ and $s$ look like
\begin{equation} \label{equa_f_and_g2}
r:= \input{\cesta/obr34}\quad \text{and}\quad  s: = \input{\cesta/obr35} \ .
\end{equation}
One easily sees $s\equiv p \mod \hbar^2$, since (modulo $\hbar^2$) we can  replace $\Delta'_Q$ in the picture (\ref{equa_f_and_g2}) by the expression in (\ref{equa_DeltaPrime_and_Delta}) and then use (\ref{equa_epsQ_is_counit}). 

The diagram defining $r$ contains the braiding which is not congruent to 1 modulo $\hbar^2$ so it becomes more involved.  Note that $r \in \catFhPhi(Q\ctp A, Q\ctp A)$ satisfies $r \cFhPhi p^\Phi = r$ and thus $r\in \YPhi( Q\ctp A)$. We have an isomorphism between $\YPhi( Q\ctp A)$ and $Q$  which is described explicitly ($\mod \hbar^2$) in corollary (\ref{cor_of_claim4}).  Thus to show that $r\equiv p^\Phi \mod \hbar^2$ it is enough to check that
\begin{equation}\label{equa_xxccvbbh}
\input{\cesta/obr55} \equiv \input{\cesta/obr56} \mod \hbar^2 .
\end{equation}
On the right hand side, we replace $p^\Phi$ by $p$ (claim \ref{claim4}) and then it becomes just $$ (claim \ref{claim2}).
Using (\ref{equa_DeltaPrime_and_Delta}) and then the lemma \ref{lem_epsPrime} one gets (diagrams in $\catPhPhi$, congruence modulo $\hbar^2$):
$$\input{\cesta/obr57} \equiv \input{\cesta/obr58} = \input{\cesta/obr59} = \input{\cesta/obr60}$$

Thus we can compute the left hand side of (\ref{equa_xxccvbbh}) as (diagrams in $\catCTVkh$\footnoteremember{foot_cross}{On a diagram in $\catCTVkh$ by $\input{\cesta/obr61}$ we mean $\input{\cesta/obr62}$. It is meaningful only if both strands are $\Up$-modules. In our case it is so, since they are images of objects in $\catP_h$ via the functor $\catPh \xto{\forg} \catCTVkh$. } ):
$$  \input{\cesta/obr63} = \input{\cesta/obr64} \equiv \input{\cesta/obr65} =  + \frac{\hbar}{2} \cdot \sum_i  \Bigg[ \input{\cesta/obr66} \input{\cesta/obr67}  + \input{\cesta/obr68} \input{\cesta/obr69}   \Bigg]$$
Now the first summand in the bracket is 0, since $\eps_A = $ is $\Ug$-linear and the second is 0 by the corollary \ref{cor_p_m_y}.
\end{proof}

Now we are ready to calculate our twist. Let's recall what we mean by it. We have two functors $F_1$ and $F_2$ and a natural isomorphism $\alpha$ between them. $F_1$ is a strict\footnote{By ``strict'' we mean that the quasi-monoidal structure is given by the identity isomorphisms in $\catCTVkh$.} quasi-monoidal functor and let's denote the quasi-monoidal structure on $F_2$ by $\lambda_{M,N} : F_2(M)\otimes F_2(N) \to F_2(M\otimes N)$.  Then we want to compute
\begin{align*}
F_1(M)\otimes F_1(N) & \xto{\alpha_M\otimes \alpha_N} F_2(M)\otimes F_2(N) \xto{\lambda_{M,N}}  F_2(M\otimes N) \to \\
&\quad \xto{\alpha^{-1}_{M\otimes N}} F_1(M\otimes N) = F_1(M) \otimes F_1(N) \  .
\end{align*}
The composition $\lambda_{M,N} \circ ( \alpha_M \otimes \alpha_N  )$ assigns to $m\otimes n \in F_1(M) \otimes F_1(N)$ the following
$$\input{\cesta/obr70} \quad\text{--- diagram in $\catPhPhi$. }$$
By the corollary \ref{cor_of_claim4}, $\alpha^{-1}_{M\otimes N}$ assigns to it (modulo $\hbar^2$)
$$ \input{\cesta/obr71} \quad
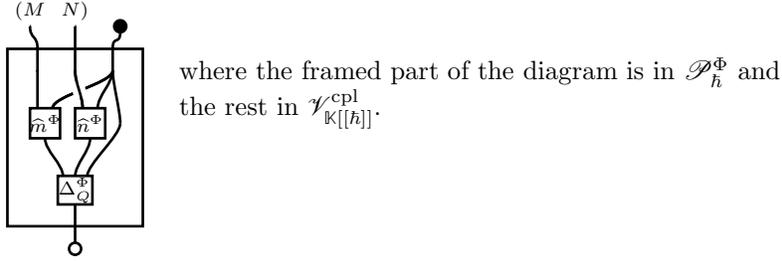
\begin{tikzpicture}
\node[text width=8cm]
{ where the framed part of the diagram is in $\catPhPhi$ and the rest in $\catCTVkh$.
};
\end{tikzpicture}
$$
Since the associativity isomorphisms in $\catPhPhi$ viewed in $\catCTVkh$ are identity modulo $\hbar^2$, we can (modulo $\hbar^2$) consider the above diagram as a diagram entirely in $\catCTVkh$. Using the claim \ref{claim4} we are thus interested in (diagrams in $\catCTVkh$\footnoterecall{foot_cross}): 
\begin{align}
\input{\cesta/obr72} & = \input{\cesta/obr73} = \input{\cesta/obr74} = \nonumber \\
&= \input{\cesta/obr75} +\hbar \Bigg( \input{\cesta/obr76} + \input{\cesta/obr77} \Bigg) +o(\hbar) \label{equa_twist1}
\end{align}
We can use $\input{\cesta/obr78} =  \begin{tikzpicture}[baseline=0.0cm]
\node at (0.0,-0.5) {};
\draw [line width=1pt] (0.0,-0.3) ..controls +(0.0,0.0) and +(0.0,0.0) .. (0.0,0.0);
\node at (0.0,0.5) {};
\draw [line width=1pt] (0.0,0.0) ..controls +(0.0,0.0) and +(0.0,0.0) .. (0.0,0.3);
\end{tikzpicture}
$ and\footnote{The first summand is $0$ by $\Ug$-linearity of $\eps_A$.}
$$\input{\cesta/obr80} =  \sum_i \Bigg( \input{\cesta/obr81} \input{\cesta/obr82} \Bigg)+  \sum_i \Bigg( \input{\cesta/obr83} \input{\cesta/obr84} \Bigg) = 0 +   \sum_i \Bigg( \input{\cesta/obr83} \input{\cesta/obr84} \Bigg)$$
to get that (\ref{equa_twist1}) is equal to
$$ m\otimes n - \frac{\hbar}{2} \sum_i \Bigg( \input{\cesta/obr85}  \input{\cesta/obr86} \Bigg)  - \frac{\hbar}{2} \sum_i \Bigg( \input{\cesta/obr87}  \input{\cesta/obr88} \Bigg) $$
Now by the corollary \ref{cor_p_m_y}, the first sum becomes $\frac{\hbar}{2} \sum_i (e^i \act m) \otimes (e_i \act n)$ and the second is one $0$. We thus finally get our twist 
\begin{equation}\label{formula_twist}
F(m\otimes n) \equiv (1+\frac{\hbar}{2} \sum_i e^i \otimes e_i )\act (m\otimes n) \mod \hbar^2 \ . 
\end{equation}

\section{Putting things together}

\subsection{A deformation of $\Ug[[\hbar]]$.}
We now use the $\kh$-module isomorphism from the subsection \ref{subsection_bijection}
$$ \Ug[[\hbar]] \xto{\simeq }  \End(\catPhPhi \xto{\Yoneda^\Phi} \catCTVkh)$$
to equip $\Ug[[\hbar]]$ with a quasi-bialgebra structure (the right hand side is a quasi-bialgebra by the lemma \ref{lem_quasibi_from_functor}).

We would like to see that this way we get a quantization of the quasi-Lie bialgebra $\bg$, i.e. that the classical limit of $\EndY$ is $\bg$. For that we use a trick of Etingof and Kazhdan. Since $\catPhPhi\xto{\ctp A} \catFhPhi$ is a strong monoidal functor, the lemma \ref{lem_functoriality} gives us a quasi-bialgebra homomorphism 
\begin{equation}\label{hom_incl} 
\Big( \Ugh = \EndY \Big) \to \EndYA \ .
\end{equation}
We need first to understand the second quasi-bialgebra.

\subsection{ Quasi-bialgebra $\EndYA$ and its classical limit.} As a functor, $\YA$ is equal to (isomorphic to) the forgetful functor $\Forg$, just their quasi-monoidal structures are different. By the remark \ref{rem_Drinfeld_construction}, the classical limit of $\EndForg$ is the quasi-Lie bialgebra associated to (see the definition \ref{defn_associated_quasi_lie_bialgebra}) $(p,\Omegap)$. We denote it by $\bp$. Since $\EndYA$ is obtained from $\EndForg$ via twisting by (see the formula(\ref{formula_twist}))
$$ F = 1+ \hbar \half e^i\otimes e_i + o((\hbar)\ ,$$
we get from the theorem \ref{thm_classical_limit} of Drinfeld that the classical limit $\bpprim$ is obtained from $\bp$ via twisting by
$$ f= \half(e_i\otimes e^i - e^i \otimes e_i )\ . $$

\subsection{$\EndY$ is a quantization of $\bg$.}
Let's denote by $\bgprim$ the classical limit of $\Ugh = \EndY$. The goal is to show that $\bgprim=\bg$. The quasi-bialgebra homomorphism (\ref{hom_incl}) induces a quasi-Lie bialgebra homomorphism
$$ \bgprim \to \bpprim $$
that, thought about as a linear map $\g\to \p$, is just the ordinary inclusion by the lemma \ref{lem_incl}. In other words, the quasi-Lie bialgebra structure $\bgprim$ is just the restriction of the quasi-Lie bialgebra structure $\bpprim$. And this is precisely $\bg$ by the lemma \ref{lem_twisting_p}.

\appendix
\section{Topological vector spaces}\label{section_Top_vect_spaces}
We consider our base field $\field$ to have discrete topology. By saying ``topological vector space'', we also mean that the topology is linear, i.e. it has a basis of neighborhoods of 0 consisting of vector subspaces. Most results mentioned here can be found in \cite{TopVec}.

\subsection{Completion}
\begin{defn}
By a completion $\hat M$ (also denoted $\compl(M)$) of a topological vector space $M$ we mean 
$$\hat M := \varprojlim_U M/U $$
where we take the limit with respect to the partially ordered set of open vector subspaces of $M$. (Note that the spaces $M/U$ have discrete topology.)
\end{defn}

\forme{
This algebraic definition is equivalent to the topological one:

\begin{defn}
Let $M$ be a topological vector space.
\begin{itemize}
\item A Cauchy filter is such a filter $\calF$ that $\forall U$ --- neighborhood of 0, $\exists A\in \calF$ s.t. $(A-A) \subset U$.
\item Two Cauchy filters $\calF, \calG$ are called equivalent if $\forall U $ --- neighborhood of 0, $\exists A\in F$; $\exists B\in G$ s.t. $(A - B) \subset U$.
\item The completion of $M$ is the set of Cauchy filters modulo the above equivalence. For any open $U\subset M$ we define $\hat U$ as the set of all Cauchy filters containing $U$. The system of subsets $\hat U \subset M$ gives a topology on $\hat M$.
\end{itemize}
\end{defn}
}

\begin{lem}\label{lem_density_and_limits}
Let $\catJ$ be a category and $F:\catJ\to \catTop$ a functor. We denote by $\pi_j$ the canonical maps $\pi_j : \lim F \to F(j); \; j\in \ob(\catJ) $. Let $M\in \ob(\catTop)$ and let's have a natural (in $j\in \ob(\catJ)$) transformation $\phi_j: M\to F(j)$. Let's denote by $\tilde\phi: M\to \lim F$ the unique continuous map satisfying $\phi_j = \pi_j\circ \tilde\phi; \; \forall j\in \ob(\catJ)$. 

If all $\phi_j$ have a dense image and $\catJ$ is cofiltered\footnote{Actually we need just that $\forall a,b \in  \ob(\catJ)$ we can find an object $c$ and morphisms $\alpha:c\to a$; $\beta:c \to b$. } then $\tilde\phi$ has also a dense image.
\end{lem}
\begin{proof}
The topology on $\lim F$ is generated by a subbase
$$\calS := \set{\; \pi_j^{-1}(U)\;\; \big| \quad  j\in \ob(\catJ) \;\text{ and }\; U\subset F(j) \;\text{open}  \;}\ . $$
Let's show that $\calS$ is actually a base. Take $a,b\in \ob(\catJ)$ and $A\subset F(a),\, B\subset F(b)$ open subsets. We want to show that $\pi_a^{-1}(A) \cap \pi_b^{-1}(B) \in \calS$. Since $\catJ$  is cofiltered, $\exists j \in \ob(\catJ)$ and morphisms $ \alpha \in\catJ(j,a); \; \beta \in \catJ(j,b)$. The commutative diagram 
$$
\begin{tikzpicture}[label distance =-0.5 mm]	
	\node(lx) at (4,-0.5){};
	\node(ly) at (3.5, 1.5){};
	\node(lz) at (0, 2.5){};
	\pgfresetboundingbox;

	\node(a000) {$\lim F$};
	
	\node(a001) at ($(a000)+(lz)$)  {$F(j) $};
	\node(a101) at ($(a001)+(lx)$) {$F(b) $};
	\node(a011) at ($(a001)+(ly)$) {$F(a)$};

	\draw[->] (a001)-- node[above left] {$F\alpha$} (a011);

	\draw[<-] (a001)-- node[left] {$\pi_j$} (a000);
	\draw[<-] (a101)-- node[right] {$\pi_b$} (a000);	
	\draw[<-] (a011)-- node[left, near end] {$\pi_a$} (a000);
	
	\draw[->, white, line width = 10pt] (a001)-- (a101);	
	\draw[->] (a001)-- node[above, near end] {$F\beta$} (a101);
\end{tikzpicture}
$$
gives us 
$$ \pi_a^{-1}(A) = \pi_j^{-1} \Big( (F\alpha)^{-1} (A) \Big) \quad \text{and}\quad \pi_b^{-1}(B) = \pi_j^{-1} \Big( (F\beta)^{-1} (B) \Big)\ .$$ 
Thus we get
$$ \pi_a^{-1} (A) \cap \pi_b^{-1}(B) \ =\  \pi_j^{-1} \Big( (F\alpha)^{-1} (A) \cap (F\beta)^{-1} (B)  \Big)\, \in\  \calS\ .$$

Now to show that $\Ima(\tilde\phi)$ is dense in $\lim F$, we want for any $\pi_j^{-1}(U) \in \calS$ find an element $m\in M$ s.t. $\tilde\phi(m) \in \pi_j^{-1}(U)$. That is, we want an $m \in M$ s.t. $\phi_j(m)\in U$. And it exists because $\phi_j$ has a dense image.
\end{proof}

\begin{cor}\label{cor_dense_i}
The canonical map $M\xto{i} \hat M$ has a dense image.
\end{cor}

\begin{defn}
A topological vector space $M$ is called complete if $M = \hat M$. Note that a complete vector space is automatically Hausdorff. 
\end{defn}

\begin{defn}
Denote by $\catV$ the category of vector spaces and linear maps, by $\catTV$ the category of topological vector spaces and continuous linear maps and by $\catHTV$, $\catCTV$ its full subcategories of Hausdorff vector spaces and of complete vector spaces. All these categories are enriched over $(\catV, \otimes)$. 
\end{defn}

\begin{prop}\label{prop_universal_completion}
Let $M \in \ob(\catTV)$.
The natural map $M\xto{i} \hat M$ has the following universal property. If $N$ is complete and $f:M \to N$ is continuous linear, then there is a unique continuous $\hat f: \hat M \to N$ that satisfies $f = \hat f \circ i$. 
$$
\begin{tikzpicture}[label distance =-0.5 mm]	
	\node(lx) at (2,0){};
	\node(ly) at (0, -1.5){};
	\pgfresetboundingbox;

	\node(a00) {$M$};
	\node(a01) at ($(a00)+(lx)$) {$ \hat M $};
	\node(a11) at ($(a00)+(lx)+(ly)$) {$N$};	
	
	\draw[->] (a00)-- node[above] {$i$} (a01);
	\draw[->] (a00)-- node[left,below] {f} (a11);
	\draw[->] (a01)-- node[right] {$\exists! \hat f$} (a11);
\end{tikzpicture}
$$

\end{prop}
\begin{proof}
The uniqueness comes from $N$ being Hausdorff and $i$ having a dense image (\ref{cor_dense_i}). Since $N = \varprojlim_V N/V $, in order to define $\hat f$ we need to define maps $\hat M \to N/V$ in some coherent way. We take them to be the compositions 
$$\hat M \to M/f^{-1}(V) \xto{f/V} N/V $$
where the first map is the canonical map $ \varprojlim M/U \to M/f^{-1}( V)$ and $f/V$ is defined by the commutativity of the diagram:
$$
\begin{tikzpicture}[label distance =-0.5 mm]	
	\node(lx) at (3,0){};
	\node(ly) at (0, -1.5){};
	\pgfresetboundingbox;

	\node(a00) {$M$};
	\node(a01) at ($(a00)+(lx)$) {$N$};
	\node(a10) at ($(a00)+(ly)$) {$M/f^{-1} V$};
	\node(a11) at ($(a00)+(lx)+(ly)$) {$N/V$};	
	
	\draw[->] (a00)-- node[above] {$f$} (a01);
	\draw[->>] (a00)-- node[left] {} (a10);
	\draw[->>] (a01)-- node[right] {} (a11);
	\draw[->] (a10)-- node[below] {$f/V$} (a11);
\end{tikzpicture}
$$

\end{proof}

\forme{
\begin{prop}
 If $M$ is Hausdorff, then $M\xto i \hat M$ is injective and $\hat M$ can be characterized as a complete vector space containing $M$ as a dense subspace in such a way that the original topology on $M$ coincides with the one induced from $\hat M$. 
\end{prop}
}

As a consequence, the forgetful functor $\catCTV \to \catTV$ is right adjoint to $M\mapsto \hat M: \catTV \to \catCTV$ and thus preserves limits. We formulate this as
\begin{prop}\label{prop_limit_complete}
A limit (in $\catTV$) of complete vector spaces is again complete.
\end{prop}
\begin{cor}
A topological vector space is complete if and only if it can be written as a limit of discrete vector spaces.
\end{cor}
Similarly, as the functor $M\mapsto \hat M: \catTV \to \catCTV $ is a left adjoint, we get:
\begin{prop}\label{prop_compl_colim}
The functor $M\mapsto \hat M: \catTV \to \catCTV $ preserves (small) colimits.
\end{prop}

\subsection{Tensor products}\label{subsection_tensor_products}
Let $M,N \in \ob(\catTV)$ be topological vector spaces (not necessarily complete). We equip $M\otimes N$ with a linear topology generated by neighborhoods of 0 of the form
$$( U\otimes N )+ (M\otimes V)  $$
where $U$ and $V$ range over the open linear subspaces of $M$ and $N$ respectively.
This turns $\catTV$ into a symmetric monoidal category with discrete $\field$ for the unit object.
Note that $M \otimes N$ can be defined as a space with a universal bilinear map $M\times N \to M\otimes N$, continuous in some sense. 

We also define $\ctp$ as 
$M\ctp N := \widehat{ M\otimes N} \ .$
It satisfies the usual universal property in $\catCTV$ and turns it into a symmetric monoidal category.

\begin{rem}\label{rem_compl_and_ctp}
From the universal properties one can show that for $M,N\in \ob(\catTV)$
$$ M\ctp N =  \hat M \ctp \hat N \ .$$
\end{rem}

\subsection{Tensor product and colimits}
Recall that in the category of modules over some commutative ring, the functor $ \argument\tp X $ is a left adjoint and thus commutes with all (small) colimits.
\begin{example} In $\catTV$, $X\tp \argument$ does not commute with infinite coproducts. As an example, take $Y$ infinite dimensional discrete. If $(y_i)_{i\in I}$ is a basis of $Y$ then $$ Y = \bigoplus_{i\in I} \field \cdot y_i \ .$$
$X\tp Y$ has a topology base of formed by subspaces of the form $V \tp Y$, where $V\subset X$ is an open subspace.
On the other hand the topology on $\bigoplus_i (X \otimes y_i)$ is generated by the subspaces of the form $\bigoplus_i  (V_i \otimes y_i)$ where $V_i\subset X$ are open subspaces. Thus $\bigoplus_i (X \otimes y_i)$ has in general much more open sets than $X \tp Y $. (Take for example $X$ infinite dimensional with the base formed by subspaces of finite codimension.)
\end{example}

\begin{prop}\label{prop_tp_and_finite_colims}
$X\tp \argument : \catTV \to \catTV $ commutes with finite colimits.
\end{prop}
\begin{proof}
It is enough to show this for direct sums and cokernels.

 To show $X\tp (M\oplus N) \cong (X\tp M)\oplus (X\tp N)$ for $M,N \in \ob(\catTV)$ we observe that the two topology bases are equivalent:
$$ \set{ \big( Y\tp(M\oplus N)\big) + \big( X\tp (U\oplus V) \big) \; |\; Y \os X, U\os M , V \os N   } $$
$$ \set{ \big( Y_1\tp M +  X \tp U  \big) \oplus \big( Y_2 \tp N + X\tp V  \big) \; |\; Y_1,Y_2 \os X, U\os M , V \os N   } $$

The topologies of $X\tp (M/N) $ and of $(X\tp M)/ (X\tp N) $ for $N\subset M \in \ob (\catTV)$ are given by the bases:
$$ \set{ \big( Y \tp (M/N) \big) + \big( X \oplus U/N ) \; |\; Y \os X, N\subset U\os M   } $$
\begin{align*}
\Bigg\{
(Y\tp M + X \tp U) / (X \tp N) \; \bigg| \; & Y \os X, N\os M \text{ such that }\\ 
& (Y \tp M + X \tp U) \supset X \tp N
\Bigg\}
\end{align*}
Since the condition $(Y \tp M + X \tp U) \supset X \tp N$ just means $N\subset U$, these two bases are easily seen to be equal.
\end{proof}

\begin{cor}\label{cor_ctp_and_finite_colims}
In the category of complete vector spaces, $X\ctp \argument : \catCTV \to \catCTV $ commutes with finite colimits for any $X\in \ob(\catCTV)$.
\end{cor}
\begin{proof}
Take a functor $\calI \to \catCTV: i\mapsto N_i$ where $\calI$ is a finite category and $i\in \ob(I)$. Let's denote by
$$\colim^{\catCTV} N_i \quad \text{and}\quad \colim^{\catTV} N_i$$ the colimits in $\catCTV$ and in $\catTV$ respectively. We use the previous proposition, the fact that completion commutes with colimits (Proposition \ref{prop_compl_colim}), the remark \ref{rem_compl_and_ctp} and $\compl(N_i) = N_i$ to calculate
\begin{align*}
\colim^{\catCTV} (X\ctp N_i) & = \colim^{\catCTV} \compl ( X\tp N_i) \cong \compl \colim^\catTV( X\tp N_i) =\\
&\cong \compl( X\tp \colim^\catTV N_i ) =  X\ctp \colim^\catTV N_i =\\
&=  X\ctp \compl\big( \colim^\catTV N_i \big) \cong X\ctp \colim^{\catCTV} \compl N_i = X\ctp \colim^\catCTV N_i .
\end{align*}
\end{proof}

\subsection{Strong topology on $\catTV(M,N)$}\label{subs_strong_topology}

\begin{defn}\label{defn_bs}
A topological vector space $K$ is called \emph{totally bounded} (we abbreviate this to \emph{bounded}) iff $K/V$ is finite dimensional $\forall V\os K$. K is called \emph{compact} if it is bounded and complete.
\end{defn}

\begin{claim}\label{claim_bounded_discrete}
If $X$ is bounded and discrete then it is finite dimensional.
\end{claim}

\begin{defn}
The \emph{strong topology} on $\catTV(M,N)$ is defined by the basis of neighborhoods of 0 of the form
$$ \set{ f:M \to N \quad \text{s.t.}\quad f(K )\subset V  } $$
where $K\bs M$ and $V\os N$.
\end{defn}
\begin{rem} \label{rem_compact-open-initial}
In other words, the strong topology is the initial linear topology with respect to the maps (the targets are discrete):
$$\hhom( M,N ) \to \hhom(K, N/V)$$
where $K\subset M$ ranges through the bounded and $V\subset N$ trough the open subspaces.  
\end{rem}

We denote $M^*:=\hhom(M,\field)$ equipped with the strong topology.

\begin{rem}
Note that if $K$ is bounded then $K^*$ is discrete. If $X$ is discrete then $X^*$ is compact.
\end{rem}

\subsection{$\catTV$-categories.}

We will equip the hom-sets $\catTV(X,Y)$ with the strong topology.
Note that the composition maps 
$$ \catTV( Y,Z ) \otimes \catTV( X,Y ) \to \catTV(X,Z)  $$
are not continuous in general and so $\catTV$ is not enriched over itself (by considering the strong topologies). One can however introduce the following notion:
\begin{defn}\label{defn_TV-category}
By a $\catTV$-category we mean a monoidal category $\catC$ enriched over vector spaces whose hom-sets are also equipped with linear topologies in such a way that
\begin{itemize}
\item the tensor product maps
$$ \catC(X_1, Y_1) \otimes \catC(X_2, Y_2) \to \catC(X_1 \otimes X_2, Y_1\otimes Y_2)  $$
are continuous.
\item The composition maps are non necessarily continuous, but the maps
\begin{align*}
g \mapsto g\circ f \quad \text{ for $f$ fixed } \\
f \mapsto g\circ f \quad \text{ for $g$ fixed }
\end{align*}
are continuous.
\end{itemize}
\end{defn}

\begin{example}
$\catTV$ and $\catCTV$ are $\catTV$-categories.
\end{example}
The following lemma is the reason why we like this type of categories:
\begin{lem}
Let $\catC$ be a $\catTV$-category and $C\in\ob(\catC)$. Then the functor $\Yoneda_C$ represented by $C$ is a functor to $\catTV$. A coalgebra structure on $C$ makes $\Yoneda_C$ into a quasi-monoidal functor to $\catTV$. More generally, the same is true for functors represented by direct summands.
\end{lem}

\subsection{Some properties of strong topology}
The goal of this subsection is Claim \ref{claim_ch} that basically says that under some conditions one has $\Hom(M, N) = N \ctp M^* $.

\begin{claim}\label{claim_ca}
If $X$ is discrete then $M \ctp X = \limita_V \big( (M/V) \tp X  \big)$; $V\os N$.
\end{claim}

\begin{claim}\label{claim_cb}
If $K$ is  bounded and $N$ is complete then (as topological vector spaces) 
$$\catTV(K,N)  = \varprojlim_{V} \catTV(K, N/V)$$
where $V$ ranges through the open subspaces of $N$.  
\end{claim}
\begin{proof}
The equality of sets follows from $N= \limita_V N/V $ and the definition of the limit.
The equality of topologies comes from Remark \ref{rem_compact-open-initial}.
\end{proof}

\begin{claim}\label{claim_cc}
If $K$ is bounded and $X$ discrete then $\catTV(K,X) = X\tp K^*$. (Since both $X$ and $K^*$ are discrete, it equals also $X\ctp K^*$.) 
\end{claim}
\begin{proof}
Since both $\catTV(K,X)$ and $ X\tp K^*$ are discrete, it is enough to show the equality of vector spaces. One has a natural monomorphism $X\tp K^* \to\catTV(K,X) $. We just need to show it is surjective. Take $f \in \catTV(K,X)$. $\Ker(f)$ is open and thus $K/\Ker(f)$ is finite dimensional discrete. Thus the bottom arrow of the following commutative diagram is an isomorphism:
$$
\begin{tikzpicture}[label distance =-0.5 mm]	
	\node(lx) at (4.5,0){};
	\node(ly) at (0, 2){};
	\pgfresetboundingbox;

	\node(a00) {$X\tp \big(K/\Ker(f)\big)^* $};
	\node(a01) at ($(a00)+(lx)$) {$\Hom\big(K/\Ker(f), X \big)$};
	\node(a10) at ($(a00)+(ly)$) {$X \tp K^*$};
	\node(a11) at ($(a00)+(lx)+(ly)$) {$\catTV(K,X) $};	
	
	\draw[->] (a00)-- node[below] {$\cong$} (a01);
	\draw[right hook->] (a00)-- (a10);
	\draw[right hook->] (a01)--  (a11);
	\draw[->] (a10)--  (a11);
\end{tikzpicture}
$$
One can regard $f$ as an element in $\Hom\big(K/\Ker(f), X \big)$ and a simple diagram chase finishes the argument.
\end{proof}

\begin{claim}\label{claim_cd}
If $K$ is bounded and $N$ complete then $\catTV(K,N) = N\ctp K^*$. 
\end{claim}
\begin{proof}
$$ \hhom(K,N) = \limita_V \hhom( K, N/V ) = \limita_V \Big( (N/V) \ctp K^* \Big) = N \ctp K^*.  $$
The first equality is from Claim \ref{claim_cb}, the second from Claim \ref{claim_cc} and the last from Claim \ref{claim_ca}
\end{proof}

\begin{claim}\label{claim_ce}
Let $M = \bigoplus_i M_i $ for $M_i$ Hausdorff, $i\in I$ and let $K$ be a bounded subspace of $M$. Then $K$ is contained in a sum of finitely many $M_i$. 
\end{claim}
\begin{proof}
From the universal property of $\bigoplus$ we see that the projections $\pi_m: \bigoplus_i M_i \to M_m$ are continuous for each $m$. We want to show that the set
$$J := \set{ m\in I \; |\; \pi_m K \text{ is nonzero} } $$
is finite. Suppose it is not the case. We can replace each $M_j$, $j\in J$ by $\pi_j (K)$ and thus suppose that $\pi_j(K) = M_j$. Further we choose whatever nonzero continuous functional $f_j : M_j \to \field$ (here we need $M_j$ Hausdorff). This will give us a continuous map $\bigoplus_{j\in J} f_j : \bigoplus_{i\in I} M_i \to \bigoplus_{j\in J} \field_j$ (here all $\field_i =\field$) and thus we can replace all  $M_i$ by $\field_j$ and $K$ by its image. Finally we are in a situation where we have a bounded vector subspace $K$ of $\bigoplus_{j\in J} \field_j$ s.t. $\pi_j (K) =\field$. Since $K$ is a subspace of a discrete space, it is itself discrete and thus finite dimensional (claim \ref{claim_bounded_discrete}). But then only finitely many of the projections $\pi_j (K)$ can be nonzero.
\end{proof}

\begin{claim}\label{claim_cf}
Let $M_i, N \in \ob(\catHTV)$. Then  $\hhom(\bigoplus_i M_i , N) = \prod_i \hhom (M_i, N) $ as topological vector spaces.
\end{claim}
\begin{proof}
Surely, we have the equality of vector spaces (just use the universal property of $\bigoplus$). So we have two topologies $\calT$, $\calT'$ on the same space $\hhom(\bigoplus_i M_i , N)$. The basis of neighborhoods of 0 in $\calT$ is formed by the subspaces
$$ U_{K,V} := \set{   f\in \hhom(\bigoplus_i M_i , N)\; \Big|\quad f(K)\subset V   } $$
where $K\bs\bigoplus M_i$ and $V\os N$.
In $\calT'$, the basis of neighborhood of 0 is given by the sets:
$$ U'_{K_{i_1},\ldots, K_{i_m} ;V} := \set{   f\in \hhom(\bigoplus_i M_i , N)\; \Big|\quad f( \bigoplus_{j=1}^m K_{i_j} ) \subset V   } $$
where $\set{ i_1, \ldots, i_m }$ is a finite subset of the index set, $K_{i_j} \bs M_{i_j}$ and $V\os N$. 

Since $U'_{K_{i_1},\ldots, K_{i_m} ;V} = U_{K_{i_1} \oplus \ldots \oplus K_{i_m}, V}$, we see that $\calT' \subset \calT$. To show $\calT \subset \calT'$ we need to any $U_{K,V}$ find a smaller $U'_{K_{i_1},\ldots, K_{i_m} ;V'}$. From the claim \ref{claim_ce} we have indices $i_1,\ldots, i_m$, s.t. $K\subset \bigoplus_{j=1}^m M_{i_j}$ and thus it is enough to take $K_{i_j} := \pi_{i_j}( K)$ (here $\pi_k : \bigoplus_i M_i \to M_k$ denotes the natural projection) and $V':=V$.
\end{proof}

\begin{claim}\label{claim_cg}
Let $X_i$ be discrete topological vector spaces indexes by $i\in I $ and $N$ be a complete vector space. Then $\prod_i ( N \ctp X_i ) = N \ctp \prod_i X_i$ as topological vector spaces. 
\end{claim}
\begin{proof}
The subspaces $\prod_{k\in I\setminus J} X_k$, where $J$ ranges through finite subsets of $I$, form a base of neighborhoods of 0 in $\prod_i X_i$. Thus ($V$ in the limits are open subspaces of $N$ and $J$ are finite subsets of $I$ )
\begin{align*}
N\ctp \prod_i X_i & \xlongequal 1 \limita_{V,J} N/V \tp \Big( \prod_{i\in I} X_i / \prod_{k\in I\setminus J} X_k  \Big) \xlongequal 2 \limita_J \limita_V \Big( N/V \tp \prod_{j\in J} X_j \Big) =\\
&\xlongequal 3 \limita_J \prod_{j\in J} \limita_V \Big( N/V \tp X_j \Big) \xlongequal 4 \prod_{i\in I} N \ctp X_i .
\end{align*}
In 4 we used Claim \ref{claim_ca}.
\end{proof}

\begin{claim}\label{claim_ch}
Let $M = \bigoplus_i K_i$ where $K_i$ are bounded and Hausdorff and let $N$ be complete. Then $\catTV(M,N) = N\ctp M^*$.
\end{claim}
\begin{proof}
\begin{align*}
\hhom(M,N) &= \hhom (\bigoplus_i K_i, N) \xlongequal 1 \prod_i \hhom ( K_i , N) \xlongequal 2 \prod_i \big( N \ctp K_i^* \big) = \\
 &\xlongequal 3 N \ctp \prod_i K_i^* \xlongequal 4 N \ctp \Big( \bigoplus_i K_i \Big)^* = N\ctp M^*
\end{align*}
Here the equalities 1,2,3 and 4 come from Claims \ref{claim_cf}, \ref{claim_cd}, \ref{claim_cg} and \ref{claim_cf} respectively.
\end{proof}

\newcommand{\Sym}{\mathrm{Sym}}
\section{Universal enveloping algebra}
Let $V\in \ob(\catHTV)$. We can construct a topological vector space
$ TV := \bigoplus_{n\geq 0} V^{\tp n}$.
We can define topological $SV$ as a closed subspace of $TV$. Since $TV$ decomposes as a topological vector space into a direct sum  $TV = SV \oplus \Ker(\Sym)$ where $\Sym$ is the symmetrization map, the topology on $SV$ can be defined as the quotient topology relative to $TV\to SV$.

\begin{thm}\label{thm_UL}
Let $L$ be a topological Lie algebra. Then as topological vector spaces, $TL\cong \Ker(\pi) \oplus S L$ where $\pi: TL \to \Uni L$ denotes the natural projection. So if we equip $\Uni L$ with the quotient space topology, the $PBW$ map $SL\to \Uni L$ will be an isomorphism of topological vector spaces.
\end{thm}
\begin{proof} We will show that for every $n\in \N$,  $T^{\leq n}  L = \Ker(\pi^{\leq n} ) \oplus S^{\leq n} L$ as topological vector spaces. From the ordinary PBW theorem we know that the composition 
$$ S^{\leq n} L \hookrightarrow T^{\leq n} L \xto{\pi^{\leq n}} \Unik n L$$
is a vector space isomorphism and thus we have $T^{\leq n}  L = \Ker(\pi^{\leq n} ) \oplus S^{\leq n} L$ as vector spaces. To see that this is also a direct sum of topological vector spaces we will show that the projection on the second summand
$$ p^{\leq n}: T^{\leq n}  L = \Ker(\pi^{\leq n} ) \oplus S^{\leq n} L \to S^{\leq n} L $$
is continuous. We do it by induction on $n$. Assume that $p^{\leq n-1}$ is continuous. We have a topological decomposition
$$ T^{\leq n}  L = T^{\leq n-1 } L \oplus S^{ n} L \oplus A^n L $$
where $A^n L$ is the kernel of the symmetrization map $T^n L \to T^n L$. Since the restrictions
$$ p^{\leq n} \Big|_{T^{\leq n-1} L} = p^{\leq n-1} \quad \text{and} \quad p^{\leq n}\Big|_{S^n L} = \id_{S^n L} $$
are continuous, it is enough to check the continuity of $  p \big|_{A^n L}: A^n L \to S^{\leq n} L$.

Denote by $A_{i,j}\subset A^n$ the subspace of tensors antisymmetric in the $i$-th and $j$-th indices. We will show in a moment that $p|_{A_{i,j}}$ is continuous. One has $A = \sum A_{i,j}$ as a vector space but this does not imply that $p|_{A^n}$ is also continuous. One can however find a decomposition of $A$ into a direct sum (in $\catTV$) $A= \bigoplus_\alpha A_\alpha$ of finitely many subspaces $A_\alpha$ s.t. each $A_\alpha$ is a subset of some $A_{i,j}$. Thus $p|_{A_\alpha}$ are continuous and consequently $p|_{A^n}$ also is.

One such decomposition of $A$ comes from the standard decomposition $L^{\tp n} = \bigoplus_T L^{\tp n} \cdot c_T$ using the Young symmetrizers $c_T$ corresponding to the standard tableaux $T$. Each $c_T$ is a composition of the symmetrization with respect to the rows of $T$ followed by the antisymmetrization in columns. Thus each $L^{\tp n}\cdot c_T$ different from $S^n L$ will be a subset of some $A_{i,j}$.

The last thing is the continuity of $p|_{A_{i,j}}$ which can be proven by finding explicit formulas involving only the Lie brackets and $p^{\leq n-1}$ which is continuous by the induction hypothesis. For example $A_{1,2}$ is generated by elements of the form
$(x_1\tp x_2 \tp\ldots) - (x_2\tp x_1\tp\ldots)$.
One has 
$$\pi\Big( (x_1\tp x_2 \tp\ldots) - (x_2\tp x_1\tp\ldots)\Big) = \pi\Big([x_1,x_2]\tp\ldots \Big)$$ and thus (because $\Ker(p) = \Ker(\pi)$) also
$$p \Big( (x_1\tp x_2 \tp\ldots) - (x_2\tp x_1\tp\ldots)\Big) = p^{\leq n-1}\Big([x_1,x_2]\tp\ldots \Big).$$ 
Similarly, on $A_{1,3}$ we get a formula
\begin{gather*}
 p\Big( (x_1\tp x_2\tp x_3 \tp \ldots) - (x_3\tp x_2 \tp x_1\tp\ldots) \Big) = \\
 = p^{\leq n-1}\Big(  \big( [ x_1,x_2 ] \tp x_3\tp\ldots\big) + \big(x_2\tp [ x_1, x_3 ]\tp\ldots \big) + \big( [x_2,x_3] \tp x_1\tp \ldots  \big)  \Big). 
\end{gather*}
One gets longer formulas when $i$ and $j$ are still further apart.
\end{proof}

\begin{rem}
From the proof of the theorem \ref{thm_UL} one can see that as a topological vector space $\Uni L = \colim_k \Unik k L$ and one has topological  PBW isomorphisms $S^{\leq k} L \xto\cong \Unik k L$.
\end{rem}

We will need the following generalization of the above theorem.
\begin{thm}\label{thm_UL1}
Let $L$ be a topological Lie algebra and $\pi: TL \to \Uni L$ the natural projection. Assume also that $L$ decomposes into a direct sum $L= L_1\oplus L_2$ of topological vector spaces $L_1,L_2$. Then $TL$ decomposes as a topological vector space into a sum
$$ TL = \Ker(\pi) \oplus  \Big( \bigoplus_{k,l} S^k L_1 \tp S^l L_2 \Big). $$
This implies $ \Uni L \cong   \bigoplus_{k,l} S^k L_1 \tp S^l L_2 $ as a topological vector space.
\end{thm}
\begin{rem}
One can generalize this to any finite number of summands $L = L_1 \oplus \ldots \oplus L_k$.
\end{rem}
\begin{proof}
We want to show for all $n \in \N$ that as a topological vector space
$$ T^{\leq n}L = \Ker(\pi^{\leq n}) \oplus  \Big( \bigoplus_{\substack{k,l \\ k+l\leq n} } S^k L_1 \tp S^l L_2 \Big). $$
As before we show by induction that the projection $p^{\leq n}$ on the second factor is continuous. In the induction step we need to show the continuity of $p\big|_{T^n L}$. We decompose $T^n L$ as
\begin{equation}
  T^n L = \bigoplus_{ f: \set{1,\ldots, n} \to \set{1,2} } L_{f(1)} \tp L_{f(2)} \tp \ldots \tp L_{f(n)} .
\end{equation}
We call a pair $i<j \in \set{1,\ldots,n}$ an inversion of  $f: \set{1,\ldots, n} \to \set{1,2}$ if $f(i) > f(j)$. We will show the continuity of $p\big|_{L_{f(1)} \tp \ldots \tp L_{f(n)}}  $ by induction on the number of inversions of $f$.

As the 0-th step assume that $f$ has no inversions. This means that $L_{f(1)} \tp \ldots \tp L_{f(n)} = T^k L_1 \tp T^l L_2$ for some $k+l = n$. We decompose it as
$$ T^k L_1 \tp T^l L_2 = \big( S^k L_1 \tp S^l L_2 \big) \oplus  \big( A^k L_1 \tp S^l L_2 \big) \oplus  \big( S^k L_1 \tp A^l L_2 \big) \oplus  \big( A^k L_1 \tp A^l L_2 \big)  $$
and proceed as in the proof of the theorem \ref{thm_UL}.

As the induction step we want to prove the continuity of $p\big|_{L_{f(1)} \tp \ldots \tp L_{f(n)}}  $ assuming the continuity of $p\big|_{L_{f'(1)} \tp \ldots \tp L_{f'(n)}}  $ for all $f'$ with less inversions than $f$. From the induction on $n$ we also assume that $p^{\leq n-1}$ is continuous. We can find a $k \in \set{1,\ldots n-1}$ s.t. $f(k) = 2$ and $f(k+1) = 1$. For simplicity assume that $k = 1$. Define continuous maps
\begin{align*}
\alpha : L_2 \tp L_1 \tp L_{f(3)} \tp \ldots \tp L_{f(n)} &\to L_1 \tp L_2 \tp L_{f(3)} \tp \ldots \tp  L_{f(n)} \\
x_1 \tp x_2 \tp x_3 \ldots &\mapsto x_2 \tp x_1 \tp x_3 \tp \ldots
\end{align*}
\begin{align*}
\beta : L_2 \tp L_1 \tp L_{f(3)} \tp \ldots \tp L_{f(n)} &\to L \tp L_{f(3)} \tp \ldots \tp  L_{f(n)} \\
x_1 \tp x_2 \tp x_3 \ldots &\mapsto [x_1, x_2] \tp x_3 \tp \ldots
\end{align*}
Then for any $t\in {L_{f(1)} \tp \ldots \tp L_{f(n)}}$ we have $ t - \alpha(t) - \beta(t) \in \Ker (\pi) = \Ker(p)$ and thus 
$$ p^{\leq n}(t) = p^{\leq n} \big(\alpha(t) \big) + p^{\leq n-1} \big(\beta(t) \big).$$
But $p^{\leq n} \circ \alpha$ and $p^{\leq n-1} \circ \beta$ are continuous from the induction hypotheses.
\end{proof}

\begin{cor}\label{cor_heielrt}
Recall that $\p = \g \oplus \g^*$ as a topological vector space and $\g$ is a Lie subalgebra of $\p$. This gives us a PBW isomorphism of vector spaces (or even $\Ug$-modules) $\Ug \tp S\g^* \xto\cong \Up$. On the topological level one gets\footnote{Note that since the tensor product (of topological vector spaces) does not commute with colimits, one actually does not have $\Up \cong \Ug \tp S\g^*$. (A simpler example of this phenomenon is that $S(V_1\oplus V_2) \neq S V_1 \tp SV_2$ as topological vector spaces.)}
$$ \Up \cong \colim_{k,l} \Unik k\g \tp S^{\leq l} \g^*. $$
\end{cor}
\begin{proof}
One has a commutative diagram
$$
\begin{tikzpicture}[label distance =-0.5 mm]	
	\node(lx) at (4,0){};
	\node(ly) at (0, 1.5){};
	\pgfresetboundingbox;

	\node(a00) {$\colim_{k,l} \Unik k\g \tp S^{\leq l} \g^*$};
	\node(a01) at ($(a00)+(lx)$) {$ \Up $};
	\node(a11) at ($(a00)+(ly)$) {$\colim_{k,l}  S^{\leq k}\g \tp S^{\leq l} \g^*$};	
	
	\draw[->] (a00)--  (a01);
	\draw[<-] (a00)-- node[left] {$ \cong$} (a11);
	\draw[<-] (a01)-- node[right, above] {$\cong$} (a11);
\end{tikzpicture}
$$
where the two downward pointing maps are topological isomorphisms by the theorems \ref{thm_UL} and \ref{thm_UL1}.
\end{proof}

\begin{rem}
$TL, SL, \Uni L$ are not complete. If we wanted them to be, we could either complete them or (equivalently) define $\hat TV := \bigoplus_{n\geq 0} V^{\ctp n}$ and then take $\hat SL, \hat \Uni L$ as its quotient spaces. We will however not need it and reserve the notation $\hat S \g$ for a different space (it will be again a completion of $S\g$ but w.r.t. a different topology).
\end{rem}

Note that the multiplication map $TV \tp TV \to TV$ will not be continuous in general (the problem is that the topological tensor product does not commute with the infinite direct sums). Our $TV$, $SV$ and $\Uni L$ will be however topological filtered algebras from the following definition.

\begin{defn}
By a filtered algebra in $\catTV$ we mean a filtered topological vector space $B = \colim_i B_i$, $B_0 \subset B_1 \subset\ldots$ that is also a filtered algebra over $\field$ and the restrictions of the multiplication
$$ \mu\big|_{B_k \tp B_l} : B_k \tp B_l \to B_{k+l} $$
are continuous for all $k,l$.
\end{defn}

\section{Equicontinuous modules}
\begin{reminder}
Recall form [EK1] that a system of maps $\set{f_x :Y\to Z }_{x\in X}$ where $Y,Z$ are topological vector spaces and $X$ a set, is said to be equicontinuous if $$\forall V\os Z;\; \exists U\os Y;\; \forall x \in X:\; f_x(U)\subset V . $$
In our case, $X$ will also be a topological vector space and the system  $\set{f_x :Y\to Z }_{x\in X}$ will come from a map $f: X\tp Y \to Z$. The equicontinuity then means the continuity of $ f:X_{\text{disc}} \tp Y \to Z $ where $X_{\text{disc}}$ is the same vector space as $X$ but with discrete topology.
\end{reminder}

Let $ L$ be a topological Lie algebra. Following [EK1], an $ L$-module $M$ is called an equicontinuous $ L$-module if the map $ L \to \catTV(M,M)$ is continuous and the system of maps $\set{x\act\argument:M \to M}_{x\in  L }$ is equicontinuous.

In this definition, there are two natural choices for the topology on $\catTV(M,M)$: a) the strong topology and b) the weak topology (this is how [EK1] do it). The choice of a) can be expressed as continuity of $ L\tp M\xto\act M$ w.r.t. the topology on $ L\tp M$ where a vector subspace $U\subset  L\tp M$ is open if
\begin{itemize}
\item[i)] $\exists V\os M$ s.t. $( L\tp V) \subset U$ and
\item[ii)] $\forall K\bs M;\; \exists T\os  L$ s.t. $(T\tp K)\subset U$. 
\end{itemize}
Similarly, to get b) just replace ii) with
\begin{itemize}
\item[ii')] $\forall m\in M;\; \exists T\os  L$ s.t. $(T\tp m)\subset U$. 
\end{itemize}
Now we can see that the two possibilities are actually equivalent.
\begin{proof}
Obviously, i) + ii) $\Rightarrow$ i) + ii'), so we just need the opposite implication. Let $U\subset  L\tp M$ satisfy i) and ii') and let $K\bs M$. By i) we can find $V\os M$ s.t. $ L\tp V\subset U$. Since $V\cap K$ is open in $K$, $K/(V\cap K)$ is finite dimensional (K is bounded) and so there exists a finite number of elements $m_1,\ldots, m_k \in K$ s.t. $K = (V\cap K) + \field m_1+ \ldots + \field m_k$. By ii') we find $T_1,\ldots, T_k \os M$ s.t. $T_i\tp m_i\subset U$. If we take $T:= \bigcap_{i=1}^k T_i$ we have
\begin{align*}
T\tp K &= \;T\tp \Big( (V\cap K) + \field m_1 + \ldots + \field m_k \Big) \;\subset \\
&\subset \;  L\tp V + T_1\tp m_1 +\ldots + T_k\tp m_k \;\subset\; U
\end{align*} 
\end{proof}

\begin{notation}
We denote the above topological vector space by $ L\oslash M$. More generally, if $M,N\in\ob(\catTV)$ we define a topological vector space $M\oslash N$ whose underlying vector space is $M\tp N$ and where a subspace $U\subset M\oslash N$ is open iff 
\begin{itemize}
\item[i)] $\exists V\os N$ s.t. $M\tp V \subset U$;
\item[ii)] $\forall K\bs N$, $\exists T\os M$ s.t. $T\tp K \subset U$
\end{itemize} 
or equivalently (the same proof as above) ii) can be replaced by
\begin{itemize}
\item[ii')] $\forall n \in N$; $\exists T\os M$ s.t. $T\tp n \subset U$.
\end{itemize}
\end{notation}

Note that $M\oslash N$ is not isomorphic to $N\oslash M$ in general. 

So a succinct definition of an equicontinuous module just says that $ L\oslash M\xto\act M$ is continuous.

\begin{claim}\label{claim_two_tensor_topologies}
If $M$ is discrete or if $N$ is bounded then $M\oslash N = M\otimes N$.
\end{claim}

\begin{rem} \label{rem_equi_g_module}
A definition of an equicontinuous $ L$-module that is often easy to check is the following:
\begin{itemize}
\item[i)] The system of maps $\set{x\act\argument:M \to M}_{x\in  L }$ is equicontinuous and
\item[ii)] for any $m\in M$ the map $\argument\act m:  L\to M$ is continuous.
\end{itemize}
\end{rem}

\begin{example}\label{example_Ug_is_g_module}
If $ L$ is a topological Lie algebra then $\Uni L$ is an equicontinuous $ L$-module (via left multiplication).
\end{example}
\begin{proof}
We use the definition from Remark \ref{rem_equi_g_module}. To prove i) we fix $V\os  \Uni L$ and want to find $U\os  \Uni L$ s.t. $ L\cdot U\subset V$. Since $ \Uni L$ is a filtered algebra, the multiplication maps $ L\tp \Unik k  L \to \Unik{k+1} L $ are continuous. Thus there exist $U_k\os \Unik k L$ s.t. $ L\cdot U_k \subset V$. Consequently, we can take $U:= \sum_{k=0}^\infty U_k$. It will be open in $ \Uni L$ since its intersection with any $\Unik k  L$ is open and the topology on $ \Uni L$ is the colimit of topologies on $\Unik k  L$.

To prove ii), fix an element $x\in \Uni L$. We want to show that $\argument\cdot x: L\to  \Uni L$ is continuous. And it really is, since $x$ is in some $\Unik k L$ and the multiplication map $ L\tp\Unik k  L\to  \Uni L$ is continuous.
\end{proof}

\begin{claim}\label{claim_tgottn}
Let $M$ be an equicontinuous $L$-module. Then for any $k\in \N$, the action maps 
$$ \Unik k  L \oslash M \xto\act M $$
are continuous. In other words
\begin{enumerate}
\item[a)] the system of maps $\set{ x\act\argument : M\to M }_{x \in \Unik k  L}$ is equicontinuous;
\item[b)] $\forall m\in M$, the map $\argument \act m : \Unik k  L \to M$ is continuous.
\end{enumerate}
\end{claim}
\begin{proof}
We will show a) and b) for $k=2$, the proof for higher $k$ being analogous.

Let's start with a). We know that
\begin{equation}\label{equa_dhfkfr}
\forall U\os M; \; \exists U' \os M;\;\text{s.t.} \; ( L \act U') \subset U
\end{equation}
and want to prove
$$ \forall U\os M; \; \exists U' \os M;\;\text{s.t.} \; ( L\act  L \act U') \subset U $$
This amounts just to using (\ref{equa_dhfkfr}) two times. 

Now let's show b) for $k=2$. Since $\Unik 2  L$ is a quotient space of $T^{\leq 2}  L$, we want to show continuity of ($m\in M$ is fixed)
$$  L\tp L \to M: x\tp y \mapsto x \act y \act m . $$
This means that for a given $U\os M$ we want to find $V_1, V_2 \os  L$ s.t.
$$ V_1 \act  L \act m \subset U \quad \text{and} \quad  L\act V_2 \act m \subset U . $$
Start with $V_2$. We know we can find $U' \os M$ s.t. $  L \act U' \subset U  $. From the continuity of $\argument \act m : L \to M$ we can find $V_2 \os  L$ s.t. $V_2 \act m \subset U'$.

Now let's find $V_1$. Note that 
$$ V_1 \act  L \act m \;\subset\;  \big( L\act V_1 \act m \big) + \big(  [V_1, L] \act m  \big). $$ 
From the continuity of $ L\tp L \to M : x\tp y \mapsto [x,y] \act m$ we can find $V_1' \os  L$ s.t. $[V'_1,  L] \act m \subset U $. Now take $V_1 := V_1' \cap V_2$.
\end{proof}

\begin{claim}
If $M,N$ are equicontinuous (Hausdorff) $L$-modules then $\hat M, M\tp N, M\ctp N$ are also equicontinuous $L$-modules.
\end{claim}

\newcommand{\HUgp}{\Hom^C_\Ug(\Up, M)}

\section{Some properties of $\srdiecko:\catH\to\catP$}
We gather here some topological claims needed in Section \ref{section_our_categories}. Throughout, we consider $\Ug$ and $S\g^*$ as subspaces in $\Up$ (see also Corollary \ref{cor_heielrt}).
\begin{claim}\label{claim_baza_v_srdiecku}
 Let $M$ be an equicontinuous $\g$-module.
Then the two systems of subspaces
$$ V_{k,U} := \set{ f\in  \HUgp \; \big| \; f(\Uni^{\leq k}\p) \subset U}; \quad k\in \N;\; U\os M$$
and
$$ \tilde V_{k,U} := \set{ f\in  \HUgp \; \big| \; f(S^{\leq k}\g^*) \subset U}; \quad k\in \N;\; U\os M$$
form two bases of neighborhoods of 0 in $\HUgp$ (with the strong topology).
\end{claim}
\begin{proof}
By the definition of the strong topology we know that the system 
$$ V'_{K,U} := \set{ f\in  \HUgp \; \big| \; f(K) \subset U} ; \quad K\bs \Up;\; U\os N$$
forms a basis of neighborhoods of 0 in $\HUgp$.
We will show that in the sequence of bases 
$$\set{V'_{K,U}}; \;  \set{\tilde V_{k,U}}; \; \set{V_{k,U}}; \; \set{V'_{K,U}}; \;$$ 
each base is finer that the next one.

1. $\set{V'_{K,U}}$ is finer than $\set{\tilde V_{k,U}}$. This is clear since $S^{\leq k} \g^*$ is bounded ($\g^*$ is compact and the tensor product of bounded spaces is bounded).

2. $\set{\tilde V_{k,U}}$ is finer than $\set{V_{k,U}}$.
Fix $V_{n,U}$. Since $M$ is an equicontinuous $\Ug$-module, we can find an open subspace $U'$ s.t. $(\Uni^{\leq n} \g \act U') \subset U$. We want to show that $\tilde V_{n,U'} \subset V_{n,U}$. Really, $\Uni^{\leq n} \p\  \subset \ (\Uni^{\leq n}\g \cdot S^{\leq n} \g^*)  $ and thus for any $f \in \tilde V_{n,U'}$ we have
$$ f(\Uni^{\leq n} \p) \;\subset \; \Uni^{\leq n} \g \act f(S^{\leq n} \g^*) \;\subset \; \Uni^{\leq n} \g \act U' \;\subset\; U . $$

3.  $\set{V_{k,U}}$ is finer than  $\set{V'_{K,U}}$.  Fix $V'_{K,U}$. We will find an $n$ s.t. $ K\subset \Uni^{\leq n} \p $, which will imply $V_{n,U}\subset V'_{K,U}$. Note that under the topological isomorphism $\bigoplus_{k=0}^\infty S^k \p = S\p \xto\cong \Up$, $\Uni^{\leq n} \p$ corresponds to $\bigoplus_{k=0}^n S^k \p$. So one can use the claim \ref{claim_ce}.
\end{proof}
\begin{cor}\label{cor_srdiecko_equicont}
If $M$ is an equicontinuous $\g$-module then $\srdiecko M$ is an equicontinuous $\p$-module. 
\end{cor}
\begin{proof} We use the definition of an equicontinuous module from Remark \ref{rem_equi_g_module}. To prove i) we need for any $V_{k,U}$ from Claim \ref{claim_baza_v_srdiecku} find a $V'\os \srdiecko M$ s.t. $\p\act V'\subset V_{k,U}$. It is enough to take $V'= V_{k+1,U}$.

To check ii) we fix $f\in \srdiecko M$ and want to show that $\argument\act f : \Up\to \srdiecko M$ is continuous. That is, for any any $V_{k,U}$ from Claim \ref{claim_baza_v_srdiecku} we need to find $V'\os \Up$ s.t. $V'\act f \subset V_{k,U}$ what just means $f(\Unik k \p\cdot V')\subset U$. Now $f$ is continuous so $U':=f^{-1}(U)$ is an open subspace of $\Up$. So we just need to find $V'\os \Up$ s.t. $\Unik k \p \cdot V' \subset U'$. And we can really find it since $\Up$ is an equicontinuous $\p$-module (example \ref{example_Ug_is_g_module}). 
\end{proof}

\begin{claim}\label{claim_pfhg} 
Let $M$ be an equicontinuous $\g$-module. An $f\in \Hom_\Ug(\Up,M)$ is continuous iff its restriction $f\big|_{S\g^*}$ is.
\end{claim}
\begin{proof}
Suppose $f':=f\big|_{S\g^*}$ is continuous. One has (corollary \ref{cor_heielrt}) $$\Up \cong \colim_{k,l}\; \Unik k \g \tp S^{\leq l} \g^* $$
so to verify the continuity of $f$ it is enough to check the continuity of all $f^{(k,l)}:= f\big|_{\Unik k \g \tp S^{\leq l} \g^*}$. We can reconstruct $f^{(k,l)}$ from $f'$ as a composition
\begin{equation}\label{equaaaaaanfjfkf}
 f^{(k,l)}: \Unik k \g \tp S^{\leq l} \g^* \xto{\id \tp f'} \Unik k \g \tp M \xto\act M . 
\end{equation}
Now $\id\tp f'$ is continuous since $f'$ is. The second map is continuous as a map $\Unik k \g \oslash M \xto\act M$ but since $\Unik k \g$ is discrete, $\Unik k\g \oslash M = \Unik k \g \tp M $ as topological vector spaces (claim \ref{claim_two_tensor_topologies}). Hence both maps in (\ref{equaaaaaanfjfkf}) are continuous and so the composition also is.
\end{proof}

\begin{claim}\label{claim_fhitr}
Let $M\in \ob(\catH)$. The restriction map $\srdiecko M = \Hom^C_\Ug (\Up, M) \to \Hom^C_\field( S\g^*, M)$ is an isomorphism of topological vector spaces.
\end{claim}
\begin{proof} 
The restriction map is bijective (claim \ref{claim_pfhg}) and continuous. So it is enough to check that it is also open. This is easy if we use the base of neighborhood of 0 in $\Hom^C_\Ug (\Up, M)$ formed by the sets $\tilde V_{k,U}$ from the claim \ref{claim_baza_v_srdiecku}.
\end{proof}

\begin{claim}\label{claim_cSg}
The topological algebra $(S\g^*)^*$  is equal to $\cSg := \limita_k S^{\leq k}\g$ where $S^{\leq k} \g$ is discrete with algebra structure $S^{\leq k} \g := S\g / S^{>k}\g$.
\end{claim}
\begin{proof}
$S^{\leq k} \g^*$ is a sub-coalgebra of $S\g^*$. This induces algebra homomorphisms $(S\g^*)^*\to (S^{\leq k}\g^*)^*$ and thus an algebra homomorphism $(S\g^*)^*\to \limita_k (S^{\leq k}\g^*)^*$. This map is a topological isomorphism by
$$(S\g^*)^* = (\bigoplus_k S^k \g^*)^* =  \prod_k (S^{\leq k}\g^*)^* = \limita_k \prod_{i=0}^k  (S^{k}\g^*)^* = \limita_k (S^{\leq k}\g^*)^*$$
where the second and the last equalities come from Claim \ref{claim_cf}.
One has further
 $$(S^{\leq k} \g^*)^* = (\compl( S^{\leq k} \g^*))^* = ((S^{\leq k} \g))^*)^* = S^{\leq k} \g $$
where the second and third equality come from the following claim
\begin{claim}\label{claim_dual_dualu}
If $X,Y$ are discrete topological vector spaces, then
\begin{enumerate}
\item[a)] $X^* \ctp Y^* = (X\tp Y)^*$
\item[b)] $(X^*)^* = X$.
\end{enumerate}
\end{claim}
\end{proof}

\begin{claim}\label{claim_some_continuity}
In the assignment (\ref{map_srdiecko_tensor}),
the composition on the right hand side is continuous and the assignment itself is continuous w.r.t the strong topologies.
\end{claim} \
\begin{proof}
By the claim \ref{claim_fhitr} it is enough to check the same for the map
\begin{align*}
\Hom^C_\field(S\g^*, M)  &\longto \Hom^C_\field(S\g^*, M) \\
s &\longmapsto \Bigg( S\g^* \xto{\Delta} S\g^* \otimes S\g^* \xto{S\otimes s} S\g^* \otimes M \xto{\act} M \Bigg) .
\end{align*}
 Since the topology on $\Hom^C_\field(S\g^*, M)$ is initial w.r.t. the maps $\Hom^C_\field(S\g^*, M) \to \Hom^C_\field(S^{\leq k}\g^*, M)$, it is actually enough to look at the map
\begin{align*}
\Theta:\Hom^C_\field(S\g^*, M)  &\longto \Hom^C_\field(S^{\leq k}\g^*, M) \\
s &\longmapsto \Bigg( S^{\leq k}\g^* \xto{\Delta} S^{\leq k}\g^* \otimes S^{\leq k}\g^* \xto{S\otimes s} S^{\leq k}\g^* \otimes M \xto{\act} M \Bigg) .
\end{align*}
First we fix $s$ and check the continuity of $\Theta s$. We can write $\Theta s$ as a composition
$$ \Theta s : S^{\leq k}\g^* \xto{\Delta} S^{\leq k}\g^* \otimes S^{\leq k}\g^* \xto{S\otimes s} S^{\leq k}\g^* \otimes s(S^{\leq k}\g^*) \xto{\act} M $$
The first two maps are continuous and the third also is because it is continuous as a map $S^{\leq k}\g^* \oslash s(S^{\leq k} \g^*) \xto{\act} M$ and $S^{\leq k}\g^* \oslash s(S^{\leq k}\g^*) = S^{\leq k}\g^* \tp s(S^{\leq k}\g^*)$ by Claim \ref{claim_two_tensor_topologies} ($s(S^{\leq k}\g^*)$ is bounded).

Now we check continuity of $s\mapsto \Theta s$. Since $S^{\leq k} \g^*$ is bounded, we have a basis of neighborhoods of 0 in $\Hom^C_\field \big( S^{\leq k} \g^*, M \big)$:
$$ V_U:= \set{ f: S^{\leq k} \g^* \to M \text{ s.t. } \Ima(f) \subset U} $$
indexed by open subspaces $U\os M$. Fix one such $V_U$. We want to find $U' \os M$ s.t. 
$$ s(S^{\leq k} \g^*) \subset U' \implies (\Theta s)(S^{\leq k} \g^*) \subset U . $$
For that, it is enough to take $U'\os M$ satisfying $\big(S^{\leq k} \g^* \act U' \big) \subset U$ which exists by the part a) of Claim \ref{claim_tgottn}.
\end{proof}

\section{Categories of $\kh$-modules}\label{section_kh_modules}
We equip $\kh$ with the $\hbar$-adic topology. This way it becomes an algebra in $\catCTV$.

\begin{defn}
We denote $\catTVkhp:= \LM{\kh}(\catTV)$ --- the category of topological $\kh$-modules with $\otimes_\kh$ as a monoidal structure.

For a topological $\kh$-module $M$ we denote by $\tilde M := \complh M:= \varprojlim_k M/\hbar^k M$ its $\hbar$-adic completion.
If $M  = \complh M $ we call it $\hbar$-complete.  

By $\catTVkh$ we denote the full subcategory of $\catTVkhp$ consisting of $\hbar$-complete modules. We equip $\catTVkh$ with the $\hbar$-completed tensor product
$$ M \hctp N := \complh( M\tp_\kh N )\ .$$

$\catCTVkh$  will denote the category of complete topological $\kh$-modules. It has a monoidal product
$$ M\ctp_\kh N := (M \tp_\kh N )\hat \ .  $$
\forme{Is not $\catCTVkh = \LM{\kh}(\catCTV)$ ?}
\end{defn}
One can see $\complh : \catTVkhp \to \catTVkh$ as the left adjoint functor to the forgetful $\catTVkh \to \catTVkhp $ as a consequence of the following:
\begin{prop}
The natural map $M\xto{i} \complh M$ has the following universal property. If $N$ is complete and $f:M \to N$ is continuous linear, then there is a unique continuous $\kh$-linear map $\tilde f: \complh M \to N$ that satisfies $f = \tilde f \circ i$.
\end{prop}
\begin{proof}
From the lemma \ref{lem_density_and_limits} we see that the map $M\xto{i} \complh M$ has a dense image w.r.t. the $\hbar$-adic topology on $\complh M$. Since $\complh M$ is Hausdorff in $\hbar$-adic topology, this implies that $\tilde f$ is unique if it exists. 

The proof of the existence is the same as in the proposition \ref{prop_universal_completion}.
\end{proof}

\begin{prop} \label{prop_complh_tensor}
If $M,N \in \ob(\catTVkhp)$ then
$$\complh \Big( \complh(M) \tp_\kh \complh(N) \Big) = \complh ( M\tp_\kh N )\ . $$
In other words, $\complh: ( \catTVkhp , \tp_\kh) \to (\catTVkh, \hctp) $ is a strong monoidal functor.
\end{prop}
\begin{proof}
Both sides can be characterized by the same universal property.
\end{proof}

\begin{claim} \label{claim_h_topology_has_fewer_open_sets}
Let $M\in \ob(\catTVkhp )$ be a topological $\kh$-module and $U\subset M$ an open vector subspace. Then there exists an $n\in \N$ s.t. $\hbar^n M \subset U$. 
\end{claim}
\begin{proof}
The $\kh$-module structure $\kh \tp M \to M$ is continuous. Now just use the definition of the topology on the tensor product (see the subsection \ref{subsection_tensor_products}). 
\end{proof}

\begin{cor}
If $M$ is complete then it is $\hbar$-complete.
\end{cor}
\begin{proof}
Using the claim we have for any fixed $U$ a map $$\tilde M \to M/ \hbar^k \to M/U \ .$$
This defines a canonical map $\tilde M \to \hat M$. Obviously, the canonical map $M\to \hat M$ factorizes as $M\to \tilde M \to \hat M$. If $M$ is complete then $M = \hat M$ and  so we have 
\begin{equation} \label{equa_one_composition}
(M \to \tilde M \to M) = \id_M\ .
\end{equation}
To finish the proof we need to see also
\begin{equation} \label{equa_another_composition}
(\tilde M \to M \to \tilde M) = \id_{\tilde M} \ .
\end{equation}
Since $M \to \tilde M$ has a $\hbar$-dense image and $\tilde M$ is $\hbar$-Hausdorff, it is enough to check (\ref{equa_another_composition}) when both sides are precomposed with $M\to \tilde M$. That is, we should see $ ( M \to \tilde M \to M \to \tilde M ) = (M \to \tilde M) $ which is a consequence of (\ref{equa_one_composition}).
\end{proof}

\begin{prop} \label{prop_Mh_complete}
We have a strong-monoidal functor preserving finite colimits
$$ (\catTV, \tp) \to (\catTVkh , \hctp) : M\mapsto M\hh := \complh ( M\tp \kh ) \ .$$ 
If $M \in \catCTV$ then $ M\hh \in \catCTVkh$.
\end{prop}
\begin{proof}
The first part is clear since our functor is the composition
$$  \catTV \xto{\argument \otimes \kh } \catTVkhp \xto \complh \catTVkh \ .$$
and both these functors are strong monoidal (Propositions \ref{prop_free_is_monoidal} and \ref{prop_complh_tensor}) and preserve finite colimits (Proposition \ref{prop_tp_and_finite_colims} and $\complh$ is a left adjoint).

To prove the last statement, assume that $M$ is a complete vector space. We can write
$$ M\hh = \complh( M\otimes \kh ) = \varprojlim_k \Big(  M \otimes \big( \field[\hbar] / \hbar^k \big) \Big) \ .$$
Since $\field[\hbar] / \hbar^k$ is finite dimensional, we get $M\hh$ as a limit of complete vector spaces. Consequently $M\hh$ si complete by  the proposition \ref{prop_limit_complete}.
\end{proof}

\subsection{$\catTVkh$-categories} \label{subsection_TVkh-categories}
If we replace $\field$-vector spaces by $\kh$-modules in the definition \ref{defn_TV-category} we get a notion of a $\catTVkh$-category. We now show how to construct $\catTVkh$-categories from $\catTV$-categories.

Let $\catC$ be a $\catTV$-category We denote by $\catC\hh$ the category with the same objects as $\catC$, with the hom-sets
$$ \catC\hh(X,Y) := \Big( \catC (X,Y) \Big) \hh  $$
and with the obvious composition of morphisms and monoidal structure.
\begin{prop}
If $\catC$ was a $\catTV$-category then $\catC\hh$ is a $\catTVkh$-category.
\end{prop}
\begin{rem}
If $\catC$ is a preabelian category then $\catC\hh$ si not necessarily a preabelian category. This poses a problem if we want to construct the category of $A$-modules for a commutative algebra $A$ in $\catC\hh$. However, we can still form $\RFree A (\catC\hh)$.
\end{rem}
\begin{observation}
Let $A$ be a commutative algebra in $\catC$. Then 
$$ \RFree A \big(\catC\hh \big) = \Big( \RFree A (\catC) \Big) \hh\ . $$
\end{observation}

\bibliographystyle{amsplain}


\end{document}